\documentclass[11pt]{amsart}
\usepackage{amsmath}
\usepackage{amsfonts}
\usepackage{amssymb}
\usepackage{mathrsfs}                  
\usepackage{amsthm}
\usepackage{enumitem}
\usepackage{amscd}
\usepackage{xcolor}
\usepackage[hyperfootnotes=false]{hyperref}
\usepackage{mathtools}
\usepackage{geometry}
\usepackage{bbm}
\usepackage{bm}
\usepackage{bbm}
\numberwithin{equation}{section}
\theoremstyle{plain}
\newtheorem{theorem}{Theorem}[section]
\newtheorem{proposition}[theorem]{Proposition}
\newtheorem{lemma}[theorem]{Lemma}
\newtheorem{corollary}[theorem]{Corollary}
\newtheorem{assumption}{Assumption}

\theoremstyle{definition}

\newtheorem{example}[theorem]{Example}

\theoremstyle{remark}
\newtheorem{remark}[theorem]{Remark}

\newcommand{\rev}[1]{\textcolor{black}{#1}}

\newcommand{\cO}{{\mathcal O}}
\newcommand{\IR}{{\mathbb R}}

\newcommand{\IN}{{\mathbb N}}
\newcommand{\IP}{{\mathbb P}}

\newcommand{\E}{\mathbb{E}}
\newcommand{\R}{\mathbb{R}}

\newcommand{\N}{\mathbb{N}}

\newcommand{\Hess}{\ensuremath{\operatorname{Hess}}}
\newcommand{\Trace}{\ensuremath{\operatorname{Trace}}}

\renewcommand{\P}{\mathbb{P}}

\newcommand{\Fc}{\mathcal{F}}

\newcommand{\M}{\mathcal{M}}

\begin{document}

%\title[DNN Approximation of SDEs with Jumps]{Deep ReLU Neural Network Approximation \\
%for Stochastic Differential Equations with Jumps}
\title[DNNs overcome the CoD for PIDEs]{Deep ReLU neural networks overcome the curse of dimensionality for
partial integrodifferential equations}

\author{Lukas Gonon}
\address{Department of Mathematics, University of Munich, Theresienstrasse 39, 80333 Munich, Germany}
\email{gonon@math.lmu.de}

\author{Christoph Schwab}
\address{Seminar for Applied Mathematics, ETH Z\"urich, R\"amistrasse 101, CH-8092 Z\"urich, Switzerland}
\email{christoph.schwab@sam.math.ethz.ch}

\maketitle
\begin{abstract}
Deep neural networks (DNNs) with ReLU activation
function are proved to be able to 
express viscosity solutions of linear 
partial integrodifferental equations (PIDEs) 
on state spaces of possibly high dimension $d$.
Admissible PIDEs comprise Kolmogorov
equations for high-dimensional diffusion, advection, 
and for pure jump L\'{e}vy processes.
We prove for such PIDEs arising from a class of jump-diffusions on $\IR^d$,
that for any \rev{suitable measure $\mu^d$ on $\IR^d$}, 
there exist constants $C,{\mathfrak{p}},{\mathfrak{q}}>0$ 
such that 
for every $\varepsilon \in (0,1]$ and for every $d\in \IN$
the DNN \rev{$L^2(\mu^d)$}-expression error 
of viscosity solutions of the PIDE is of size $\varepsilon$ 
with 
DNN size bounded by $Cd^{\mathfrak{p}}\varepsilon^{-\mathfrak{q}}$.

In particular, the constant $C>0$ is 
independent of $d\in \IN$ and of $\varepsilon \in (0,1]$ 
and depends only on the coefficients in the PIDE 
and the measure used to quantify the error.
This establishes that 
ReLU DNNs can break the curse of dimensionality (CoD for short)
for viscosity solutions of linear, possibly degenerate 
PIDEs corresponding to \rev{suitable} Markovian jump-diffusion processes. 

As a consequence of the employed techniques we also obtain that 
expectations of a large class of 
path-dependent functionals of the underlying jump-diffusion processes 
can be expressed without the CoD.

\end{abstract}
{\tiny 
\tableofcontents
}

\keywords{Keywords: Deep neural network, Jump-diffusion process, Option pricing, Partial integrdifferential equation, Expression rate, Curse of dimensionality. 
	
Mathematics Subject Classification (2020): 68T07, 45K05, 60H35}
\newpage
%%%%%%%%%%%%%%%%%%%%%%%%%%%%%%%%%%%%%%%%%%%%%%%%%%%%%%%%
\section{Introduction}
\label{sec:Intro}
%%%%%%%%%%%%%%%%%%%%%%%%%%%%%%%%%%%%%%%%%%%%%%%%%%%%%%%%
\subsection{Problem Formulation}
\label{sec:Scope}
%%%%%%%%%%%%%%%%%%%%%%%%%%%%%%%%%%%%%%%%%%%%%%%%%%%%%%%%

Numerous models in science and engineering are based on
stochastic differential equations (SDEs for short) with 
integrators being either diffusions, jump-processes or
a combination of both. 
We mention only financial modelling of markets with exogenous
shocks 
or noisy systems in life-sciences or economics.
In the present paper, we consider a very general class of SDEs which 
comprises many of the models arising in the mentioned applications.

Specifically, for each $x \in \R^d$ for a system with state space
of finite, but possibly high, dimension $d \in \N$ we consider an 
$\R^d$-valued  stochastic process $(X_t^{x,d})_{t \geq 0}$ 
satisfying the stochastic differential equation (SDE for short)
\begin{equation} \label{eq:SDE} 
\begin{aligned}
d X_t^{x,d} 
=
b^d(X_{t_-}^{x,d}) d t 
+ 
\sigma^d(X_{t_-}^{x,d}) d B^d_t 
+ 
\int_{\|z\|<1} f^d(X_{t_-}^{x,d},z) \tilde{N}^d(dt,dz) 
+ 
\int_{\|z\|\geq 1} g^d(X_{t_-}^{x,d},z) N^d(dt,dz),
\end{aligned} 
\end{equation}
for $t >0$ and with $X^{x,d}_0 = x$. 
In \eqref{eq:SDE}, $B^d$ is a $d$-dimensional standard Brownian motion and 
$N^d$ is an independent Poisson random measure on $\R_+ \times (\R^d \setminus \{0\})$ 
with intensity $\nu^d$ and compensated measure $\tilde{N}^d(dt,dz) = N^d(dt,dz) - dt \, \nu^d(dz)$, 
both defined on a filtered probability space $(\Omega,\Fc,\P,(\Fc_t)_{t \geq 0})$ 
satisfying the usual conditions and independent of $\Fc_0$. 
The coefficient functions $b^d \colon \R^d \to \R^d$, 
$\sigma^d \colon \R^d \to \R^{d \times d}$, 
$f^d,g^d \colon \R^d \times (\R^d \setminus \{0\}) \to \R^d$ are assumed to be measurable 
and to satisfy some regularity and growth conditions specified later on. 
Here and throughout, 
unless explicitly stated otherwise, $\| \circ \|$ shall denote the 
Euclidean norm of a vector $x\in \IR^d$, i.e. $\| x \|^2 = \sum_{i=1}^d |x_i|^2$.

The goal of this paper is to derive, \rev{under suitable conditions}, 
deep neural network expression rates for 
functionals of \rev{certain} diffusions with jumps as in \eqref{eq:SDE}
(sometimes also \rev{referred to as} ``jump-diffusions'' or \rev{as} ``It\^o processes'') 
on $\R^d$, with expression rate bounds which are explicit in the state space dimension $d$. 
\rev{This then also yields deep neural network expression rate bounds which are free from the CoD
for viscosity solutions of Kolmogorov PIDEs associated with the process $(X_t^x)_{t \in [0,T]}$. 
These equations are of type 
	\begin{equation}
	\label{eq:LevyPDE}
	\begin{aligned}
	(\tfrac{\partial}{\partial t}u_d)(t,x)
	& = - \tfrac 12 \Trace \! \big(\sigma^d(x)[\sigma^d(x)]^{*}(\Hess_{x} u_d )(t,x)\big) -  \big\langle b^d(x), (\nabla_x u_d) (t,x)\big\rangle 
	\\ 
	& \quad  
	- \int_{\R^d } \left[u_d(t,x+\gamma^d(x,z))-u_d(t,x)-\big\langle \gamma^d(x,z), (\nabla_x u_d) (t,x)\big\rangle \mathbbm{1}_{\{\|z\|< 1\}} 
	\right] \nu^d(d z) 
	\\
	u_d(T,x) &= \varphi_d(x)
	\end{aligned}
	\end{equation}
for $t \in [0,T)$, $x \in \R^d$, with terminal condition $\varphi_d$ and with $\gamma^d(y,z)$ related to $f^d$, 
$g^d$ in \eqref{eq:SDE} according to 
$\gamma^d(y,z) = f^d(y,z) \mathbbm{1}_{\{\|z\|< 1\}} + g^d(y,z) \mathbbm{1}_{\{\|z\|\geq 1\}}$,
see for instance \cite[Section~6.7]{Applebaum2009}.	
The connection between the SDE \eqref{eq:SDE} and the PIDE \eqref{eq:LevyPDE} 
is provided by a suitable Feynman-Kac formula \rev{which is}
well-known 
(e.g. \cite{FnmKacLevyGlau}, \cite[Theorem 3.4]{BBP1997},
\cite[Chap. 6.7.2]{Applebaum2009}, \cite{KharrPham15}).}

We emphasize already at this point that 
we do not assume that $\sigma^d$ in \eqref{eq:SDE} has full rank. 
In \rev{particular}, 
matrix functions $\sigma^d:\IR^d\to \IR^{d\times r}$ for some $0\leq r <d$
are admissible, with $r=0$ corresponding to $\sigma^d=0$, 
by padding $\sigma^d$ with zero entries to a $d\times d$ array.
\rev{O}ur analysis will \rev{also} cover the pure diffusion case, 
where $\tilde{N}^d=0$ and $N^d=0$ in \eqref{eq:SDE}.
\rev{For the combined degeneracy $\sigma^d=0$ and 
    $\tilde{N}^d=0$ and $N^d=0$ in \eqref{eq:SDE}, our results
    cover deterministic, pure advection for which DNN expression rates were studied 
    recently in \cite{Laakmann2020}.
Our results also}
generalize recent expression rate bounds 
%for the deterministic case 
from 
%\cite{Laakmann2020} and also generalize 
\cite{HornungJentzen2018}, \cite{GS20_925}.
The class of processes described by stochastic differential equations \eqref{eq:SDE} 
contains a large class of Markovian semimartingales on $\R^d$, 
see for instance \cite{Jac2003}. 
Specifically,
L\'evy and affine jump-diffusion processes in $\IR^d$, 
and 
various local and stochastic volatility models with jumps are covered. 
These processes are widely used in quantitative finance 
\rev{for} modeling \rev{(log-returns of) prices of risky assets} 
which potentially exhibit jumps
and 
to evaluate hedging strategies of \rev{financial} 
derivatives written on these assets. 
\subsection{Previous Results}
\label{sec:PrevRes}
Classically, in financial applications modelled by stochastic differential equations \eqref{eq:SDE},
a parametric model is chosen within this class 
and subsequently, parameters are calibrated to observed market prices of options. 
The calibrated model is then used to compute prices and hedging strategies of derivatives, 
or these model quantities at least serve as a basis for actual trading decisions. 
For sophisticated models this procedure may require extensive computational resources, 
which has classically limited their usage in practice. 
However, 
a series of recent works exploit computational advances and efficient implementations 
of deep neural networks (DNNs for short) in order to \textit{learn} the steps involved 
in the procedure sketched above (calibration, pricing or hedging)
using deep neural networks, 
see for example 
\cite{Buehler2018}, \cite{Cuchiero2019}, \cite{Horvath2019}, \cite{Bayer2018}, 
\cite{Becker2019}, \cite{Hernandez2017} and the surveys \cite{Ruf2020}, \cite{Germain2021}, \cite{Beck2020}. 
These methods have been shown to work very well 
and are being widely adopted in industry, 
in particular in applications involving large baskets of assets
which corresponds to high dimension $d$.
However,
many questions regarding theoretical foundations of their performance, 
in particular for large $d$, are still open. 
Important progress has been made recently by mathematical results 
on deep neural network approximations for partial differential equations (PDEs)
and option prices in certain stochastic models. 
We refer to, for example,
\cite{EGJS18_787},
\cite{Gonon2019},
\cite{HornungJentzen2018},
\cite{HutzenthalerJentzenKruse2019},
\cite{HJKNvW2020},
\cite{ReisingerZhang2019}, \cite{GS20_925}, \rev{\cite{Gonon2021}}
and the references there.
The results in these references 
show in particular that DNNs are capable of approximating functions 
$u$ of type $x \mapsto u(0,x) = \E[\phi(X_T^x)]$ 
\emph{without the curse of dimensionality (CoD)}.
In these references,
$(X_t^x)_{t \in [0,T]}$ is a diffusion process starting at $X_0^x =x$ 
and in \cite{GS20_925}, \rev{\cite{Gonon2021}} it is a L\'evy process. 
In particular, DNN approximations for models with jumps have only been considered in \cite{GS20_925}, \rev{\cite{Gonon2021}}.
In some of the mentioned works,
the DNN expression rate results are also formulated for $u$ being a 
viscosity solution of the Kolmogorov P\rev{(I)}DE associated with the process $(X_t^x)_{t \in [0,T]}$\rev{, as in \eqref{eq:LevyPDE}}.
Accordingly, we also address this aspect.

In a financial modelling context, the function $x \mapsto u(0,x)$ 
is the price at time $0$ of a derivative with payoff $\phi$ 
at maturity $T$ and underlying $X^x$ with initial price $x$ 
(at least if $\P$ is a risk-neutral measure for $X^x$). 
From a perspective of applications in financial modelling, however, 
it is often relevant not only to learn the prices as a function of the initial value $x$, 
but rather as a function of the parameters specifying the derivative. 
For instance, one is interested in learning (for fixed $x$) 
the strike-to-call price map $K \mapsto \E[(X_T^x-K)^+]$ by a neural network. 
With the exception of \cite{GS20_925} where 
DNN expression rates for geometric L\'{e}vy models were proved
with arguments based on stationarity and time-homogeneity of L\'{e}vy processes 
this question does not seem to have been considered in the literature. 
One of the contributions of this paper 
is to provide alternative proofs, which do not rely on stationarity, 
for deep neural network approximation rates overcoming the CoD.
The present arguments extend also to \textit{parametric} payoffs. 
In addition, 
we introduce several practically relevant features 
not treated in the literature previously: 
we consider \textit{path-dependent options} and assume that the 
stochastic model is driven by a general 
\textit{SDE with jumps} as in \eqref{eq:SDE}.
This comprises the case of L\'evy processes considered in \cite{GS20_925}, \rev{\cite{Gonon2021}},
but also (non stationary)
diffusion-driven models considered in the mentioned references.
%%%%%%%%%%%%%%%%%%%%%%%%%%%%%%%%%%%%%%%%%%%%%%%%%%%%%%%%%%%%%55
\subsection{Contributions}
\label{sec:Contr}
%%%%%%%%%%%%%%%%%%%%%%%%%%%%%%%%%%%%%%%%%%%%%%%%%%%%%%%%%%%%%%%
A principal contribution of this paper is a proof that deep ReLU NNs 
of feedforward type 
are able to approximate \rev{viscosity solutions of PIDEs \eqref{eq:LevyPDE} 
with approximation rates that are
free from the CoD. The latter means that the number of DNN parameters needed to achieve approximation accuracy $\varepsilon>0$ grows at most polynomially in $\varepsilon^{-1}$ and in the dimension $d$ of the PIDE. More precisely, under suitable conditions we prove that for any suitable probability measure $\mu^d$ on $\R^d$ 
there exist constants $C,{\mathfrak{p}},{\mathfrak{q}}>0$ 
such that 
for every $\varepsilon \in (0,1]$ and for every $d\in \IN$
the DNN $L^2(\mu^d)$-expression error 
of viscosity solutions of the PIDE is of size $\varepsilon$ 
with 
DNN size bounded by $Cd^{\mathfrak{p}}\varepsilon^{-\mathfrak{q}}$. In particular, the constant $C>0$ is 
independent of $d\in \IN$ and of $\varepsilon \in (0,1]$ 
and depends only on the coefficients in the PIDE 
and the measure $\mu^d$ used to quantify the error. This establishes that 
ReLU DNNs do not suffer from the CoD in the approximation of viscosity solutions of linear, possibly degenerate 
PIDEs corresponding to \rev{suitable} Markovian jump-diffusion processes.}
\rev{
This result partially unifies and extends recent results}
%of this type for deterministic advection in \cite{Laakmann2020}
%or 
also for diffusions in 
\cite{EGJS18_787},
\cite{Gonon2019},
\cite{HornungJentzen2018},
\cite{HutzenthalerJentzenKruse2019},
\cite{HJKNvW2020},
\cite{ReisingerZhang2019},
%\cite{GS20_925},
\cite{RegDeDeAQ2019}
and the references there. 
\rev{The key situations in which such approximation results are relevant is when the solution of a PIDE needs to be learnt from observational data or when a DNN-based algorithm is employed to numerically solve the PIDE. In particular, in the former case a DNN-based algorithm does not require any knowledge about the coefficients of the underlying PIDE.} 

\rev{
More generally, we prove under suitable conditions that ReLU DNNs are able to approximate expectations of parametric, 
path-dependent functions (respectively option prices) in 
the general class of stochastic differential equations 
with jumps \eqref{eq:SDE} with approximation rates that are
free from the CoD. 
}
The present results \rev{thereby} contribute in particular to an improved theoretical understanding 
of the success of deep learning methods currently employed in high-dimensional
option pricing in finance. \rev{Let us point out that the DNN expression rate bounds for viscosity solutions $x\mapsto u_d(t,x)$ of PIDEs \eqref{eq:LevyPDE} follow from these more general results by specializing to 
initial (respectively terminal) data 
$\varphi_d$ (which is, e.g., a payoff function in option pricing applications)
that only depends on $X^x_T$ (i.e.\ the path-dependent option is in fact a European option) 
with no parametric dependence.}

We also note that in the particular case that $\sigma^d=0$ (no diffusion), $\nu^d = 0$ (no jumps), \eqref{eq:LevyPDE}
reduces to a first order, linear advection equation as was recently considered in \cite[Eqn. (1.1)]{Laakmann2020}. In the case $\nu^d = 0$ (no jumps), \eqref{eq:LevyPDE} reduces to a diffusion equation. \rev{Therefore, Kolmogorov
	equations for high-dimensional diffusion, advection, 
	and for pure jump L\'{e}vy processes are covered}. 
The presently obtained expression results comprise thus in particular those obtained in \cite{Laakmann2020} without source term and partially extend \cite{EGJS18_787}, \cite{Gonon2019}, \cite{HornungJentzen2018}, \cite{HutzenthalerJentzenKruse2019}, \cite{HJKNvW2020}, \cite{ReisingerZhang2019}. \rev{Such DNN expression rate results are a crucial step towards obtaining a full error analysis free from the CoD for neural network-based algorithms for PDEs, as obtained in \cite{Gonon2021} for certain non-degenerate Black-Scholes-type PDEs.} 

\rev{In addition}, the article develops techniques to address DNN approximations of non-local terms as in \eqref{eq:LevyPDE}. These terms need to be handled considerably differently (see the outline below) than the diffusion equations studied in \cite{EGJS18_787}, \cite{Gonon2019}, \cite{HornungJentzen2018}, \cite{HutzenthalerJentzenKruse2019}, \cite{HJKNvW2020}, \cite{ReisingerZhang2019}. The techniques developed here for non-local terms promise to be useful also for \rev{analysing DNN expression rates for} non-linear, non-local PDEs in future research.
\rev{
More specifically, as a key tool for the DNN expression rate results derived in the article we develop novel (strong) approximation results for the Euler-Maruyama method for general SDEs with jumps \eqref{eq:SDE}. 
%Let us relate also these auxiliary results to the literature. 
In several situations, the Euler scheme for stochastic differential equations of type \eqref{eq:SDE} has been well-studied. 
We refer to \cite{higham2005} and \cite{Platen2010} for an extensive treatment of the case 
when the random measure has finite intensity. 
In \cite{protter1997}, \cite{KUHN20192654} the case of L\'evy-driven SDEs 
(which corresponds to a multiplicative structure of $f^d, g^d$) is studied. 
For general Feller processes a convergence result for the Euler scheme (albeit without convergence rates) 
is proved in \cite{boettcher2011}.  
However, on the one hand, from these works it is not straightforward to extract the constants 
(we need here bounds which make explicit the dependence on the dimension) and, on the other hand, 
none of the articles provides convergence rates in the generality as treated here 
(we allow both multiplicative and more general structure for the jump measure). 
Hence, we develop the required error bounds with strong rates, and explicit dependence of 
constants on the dimension $d$ in a self-contained fashion.}

\rev{Finally, let us point out in passing that the results proved here contribute
to the theoretical understanding of approximation capabilities 
of deep neural networks. 
In recent years, 
DNN based approximation schemes have been shown to be able to overcome the 
CoD for approximating certain classes of functions. 
\rev{In general,} these do not include generic smoothness classes (see, e.g., \cite{Mhaskar1996}, \cite{Maiorov1999}, \cite{Yarotsky2017}, \cite{E2021}), 
but in addition to the solutions of PDEs (see the references above) 
these include for instance the function classes introduced in the seminal work \cite{Barron1993} 
and certain compositional functions \cite{Poggio2017WhyAW}. We refer, e.g., to \cite{Bach2017}, \cite{E2022}, \cite{SiegelXu2022} for further results and references regarding such classes of functions.
The present results identify certain families of \textit{non-local} 
PDEs as alternative classes of functions 
which can be approximated DNNs without the CoD. 
}
\subsection{Layout}
\label{sec:Layout}
The main results are contained in Section~\ref{sec:main}. 
In Section~\ref{sec:setting} we introduce the setting and notation used throughout the article. 
Section~\ref{sec:DNNs} introduces notation and definitions on the DNNs which are
used in the approximation results.
In Section~\ref{sec:prelim} we prove some approximation results that are 
needed for the proof of the main results. 
In particular, we derive (strong) approximation results 
for the Euler scheme for the SDE \eqref{eq:SDE} and combine it with further approximation steps 
to truncate the small jumps of the L\'evy measure 
to a set $A_\delta = \{z \in \R^d \colon \|z\|>\delta \} $ for some $\delta > 0$ 
(as e.g.\ in \cite{AsRos2001} for univariate L\'evy processes), 
to approximate the coefficients of the SDE and to 
provide a Monte Carlo approximation of an integral involving $\nu^d$.
In some situations, these approximations are in principle well-known, 
but they are required here in more general situations and 
with dimension-explicit bounds for the proof of our main result. 
Using these approximation results we can approximate the solution of the underlying SDE 
by a process whose sample paths can be emulated by a DNN. 
In the case of a L\'evy-driven SDE the multiplicative structure, \rev{due to which $\int_{\R^d} \gamma^d(X_{t_-}^{x,d},z) \tilde{N}^d(dt,dz) = F^d(X_{t_-}^{x,d})  \int_{\R^d} G^d(z) \tilde{N}^d(dt,dz)$},  allows for a simpler DNN emulation approach.
Hence, in this case (corresponding to Assumption~\ref{ass:jumps}(i) below) the approximation step, 
in which the small jumps of the  L\'evy measure are truncated, 
is not required and we work with $\delta=0$ instead. 
In other words, the two different hypotheses under which we work 
(Assumption~\ref{ass:jumps}(i) and Assumption~\ref{ass:jumps}(ii)) 
require different approximation methods for the proof of the main result, Theorem~\ref{thm:main} below.

%%%%%%%%%%%%%%%%%%%%%%%%%%%%%%%%%%%%%%%%%%%%%%%%%%%%%%%%%%%%%%%5
\section{Setting and Notation}\label{sec:setting}
%%%%%%%%%%%%%%%%%%%%%%%%%%%%%%%%%%%%%%%%%%%%%%%%%%%%%%%%%%%%%%%%
This section contains various preparatory ingredients about 
It\^o processes and PIDEs.
We also introduce various assumptions that will be required 
later on for proving DNN expression rate bounds for these processes. 
Finally, we address existence and uniqueness of solutions of 
\eqref{eq:SDE} and \eqref{eq:LevyPDE}. \rev{Throughout the article we let $T>0$ denote a fixed time horizon.}
%%%%%%%%%%%%%%%%%%%%%%%%%%%%%%%%%%%%%%%%%%%%%%%%%%%%%%%%%%%55
\subsection{It\^o processes}
\label{sec:Ito}
%%%%%%%%%%%%%%%%%%%%%%%%%%%%%%%%%%%%%%%%%%%%%%%%%%%%%%%%%%%%%
We recall that 
$B^d$ is a $d$-dimensional standard Brownian motion and
$N^d$ is an independent Poisson random measure on $\R_+ \times (\R^d \setminus \{0\})$
with intensity $\nu^d$ and compensated measure $\tilde{N}^d(dt,dz) = N^d(dt,dz) - dt \, \nu^d(dz)$,
both defined on a filtered probability space $(\Omega,\Fc,\P,(\Fc_t)_{t \geq 0})$
satisfying the usual conditions and independent of $\mathcal{F}_0$.
For each $d \in \N$ and $x\in \IR^d$, 
we consider the SDE  
\begin{equation} \label{eq:SDEnew} 
\begin{aligned}
d X_t^{x,d} 
= 
\beta^d(X_{t_-}^{x,d}) d t 
+ 
\sigma^d(X_{t_-}^{x,d}) d B^d_t 
+ 
\int_{\R^d} \gamma^d(X_{t_-}^{x,d},z) \tilde{N}^d(dt,dz),
\end{aligned} 
\end{equation}	
for some measurable coefficient functions 
$\beta^d:\IR^d\to \IR^d$, $\sigma^d:\IR^d \to \IR^{d\times d}$, 
$\gamma^d:\IR^d \times (\IR^d\backslash \{ 0 \})\to \IR^d$. 
We also recall that $\nu^d$ is a $\sigma$-finite measure on 
$(\IR^d\backslash \{ 0 \},\mathcal{B}(\IR^d\backslash \{ 0 \}))$ 
with respect to which the function $z \mapsto 1 \wedge \|z\|^2$ is integrable. 
To simplify notation we will consider $\nu^d$ as a measure on $\R^d$ with $\nu^d(\{0\})=0$.

Under the integrability conditions that we are going to impose
on $\gamma^d$ in Assumption~\ref{ass:LipschitzGrowth}(ii), 
the SDE \eqref{eq:SDE} 
can always be rewritten as \eqref{eq:SDEnew} and vice versa. 
To simplify notation in what follows we will thus work with the SDE \eqref{eq:SDEnew} 
and formulate our assumptions in terms of the coefficient functions $\beta^d, \sigma^d, \gamma^d$.  
The coefficients of the SDEs \eqref{eq:SDE}, \eqref{eq:SDEnew}
are related via
\begin{equation} \begin{aligned} \label{coeffModified}
 \beta^d(y)  & = b^d(y) + \int_{\|z\|\geq 1} g^d(y,z) \nu^d(dz)
\\ \gamma^d(y,z) & = f^d(y,z) \mathbbm{1}_{\{\|z\|< 1\}} + g^d(y,z) \mathbbm{1}_{\{\|z\|\geq 1\}}.
\end{aligned}
\end{equation}
and conversely 
\begin{equation} \begin{aligned} \label{coeffModified2}
b^d(y)  & = \beta^d(y) - \int_{\|z\|\geq 1} \gamma^d(y,z) \nu^d(dz)
\\ 
f^d(y,z) & = \gamma^d(y,z) \mathbbm{1}_{\{\|z\|< 1\}}, 
          \quad  g^d(y,z) = \gamma^d(y,z) \mathbbm{1}_{\{\|z\|\geq 1\}}.
\end{aligned}
\end{equation}
In addition, inserting \eqref{coeffModified} into the PIDE \eqref{eq:LevyPDE} it can be rewritten as 
\begin{equation}
\label{eq:LevyPDEnew}
\begin{aligned}
(\tfrac{\partial}{\partial t}u_d)(t,x)
& = - \tfrac 12 \Trace \! \big(\sigma^d(x)[\sigma^d(x)]^{*}(\Hess_{x} u_d )(t,x)\big) -  \big\langle \beta^d(x), (\nabla_x u_d) (t,x)\big\rangle
\\ 
& \quad  - \int_{\R^d } \left[u_d(t,x+\gamma^d(x,z))-u_d(t,x)-\big\langle \gamma^d(x,z), (\nabla_x u_d) (t,x)\big\rangle \right] \nu^d(d z) 
\\
u_d(T,x) &= \varphi_d(x).
\end{aligned}
\end{equation}	
%%%%%%%%%%%%%%%%%%%%%%%%%%%%%%%%%%%%%%%%%%%%%%%%%%%5
\subsection{Regularity and Growth Conditions}
\label{sec:RegGrwt}
%%%%%%%%%%%%%%%%%%%%%%%%%%%%%%%%%%%%%%%%%%%%%%%%%%%%
We describe regularity and growth conditions on the coefficient functions 
which we shall assume 
for most of the results in the article; we will always state explicitly which Assumptions are required.
\begin{assumption}\label{ass:LipschitzGrowth}[Lipschitz and Growth Conditions]
There exists a constant $L>0$ so that for each $d \in \N$
the coefficient functions satisfy
\begin{itemize}
\item[(i)] [Global Lipschitz condition] 
for all $x,y \in \R^d$ 
\[ \begin{aligned}
	\|\beta^d(x)-\beta^d(y)\|^2+\|\sigma^d(x)-\sigma^d(y)\|_F^2 
        + 
        \int_{\R^d } \|\gamma^d(x,z)-\gamma^d(y,z)\|^2 \nu^d(d z) 
        \leq L \|x-y\|^2,
    \end{aligned} 
\]
\item[(ii)] 
[Linear growth condition]
for all $x \in \R^d$, $i,j \in \{1,\ldots,d\}$,
\[
|\beta^d_i(x)|^2+|\sigma^d_{i,j}(x)|^2 
+ 
\int_{\R^d} |\gamma^d_{i}(x,z)|^2 \nu^d(dz) 
\leq 
L (1+\|x\|^2).	
\]
\end{itemize}
\end{assumption}

Recall that a (strong) solution to \eqref{eq:SDE} 
is a c\`adl\`ag adapted process $X^{x,d}$ taking values in $\R^d$ such that $\P$-a.s.\ 
(the integrated version of) \eqref{eq:SDE} holds for all $t \in [0,T]$.
Assumption~\ref{ass:LipschitzGrowth} guarantees that there exists 
a $\IP$-a.s.\ (pathwise) unique solution to \eqref{eq:SDE}, 
see, e.g.,
\cite[Theorem~6.2.9]{Applebaum2009}, \cite[Chapter~5]{Protter2004}. \rev{The next remark explains in detail why Assumption~\ref{ass:LipschitzGrowth} implies that \cite[Theorem~6.2.9]{Applebaum2009} can be applied.}
\begin{remark} \label{remk:ApplB} 
The assumptions of \cite[Theorem~6.2.9]{Applebaum2009} 
are formulated slightly differently than in Assumption~\ref{ass:LipschitzGrowth} above. 
However, Assumption~\ref{ass:LipschitzGrowth} ensures 
(C1) and (C2) in \cite[Theorem~6.2.9]{Applebaum2009} are satisfied for any fixed $d \in \N$. 
To see this, define the seminorm $\|a\|_1 := \sum_{i=1}^d |a_{i ,i}|$ for $a \in \R^{d \times d}$ 
and the matrices $a(x,y)=\sigma^d(x) (\sigma(y)^d)^\top$,  $x,y \in \R^d$, 
then 
\[ \begin{aligned}
\|a(x,x)-2a(x,y)+a(y,y)\|_1  
& = \sum_{i=1}^d |\sum_{k=1}^d \sigma^d_{i,k}(x) \sigma^d_{i,k}(x) - 2 \sigma^d_{i,k}(x) \sigma^d_{i,k}(y) + \sigma^d_{i,k}(y) \sigma^d_{i,k}(y)|
\\ 
& = \sum_{i=1}^d \sum_{k=1}^d (\sigma^d_{i,k}(x)  - \sigma^d_{i,k}(y))^2 = \|\sigma^d(x)-\sigma^d(y)\|_F^2.
\end{aligned} \]
Therefore, the Lipschitz-condition in Assumption~\ref{ass:LipschitzGrowth}(i) 
coincides with the Lipschitz condition (C1) in \cite[Theorem~6.2.9]{Applebaum2009}. 
In addition, Assumption~\ref{ass:LipschitzGrowth}(ii) implies for all $y \in \R^d$
\[ 
\begin{aligned}
\|\beta^d(y)\|^2 + \|a(y,y)\|_1 + \int_{\R^d} |\gamma^d(y,z)|^2 \nu^d(dz)  
& \leq 2 d L (1+\|y\|^2) + \sum_{i=1}^d \sum_{k=1}^d |\sigma^d_{i,k}(y)|^2 \leq (2 d + d^2) L (1+\|y\|^2).
\end{aligned} 
\]
Thus the growth condition (C2) in \cite[Theorem~6.2.9]{Applebaum2009} is satisfied.
\end{remark}

\begin{remark} 
\label{remk:purejmp}[Pure Jump Process]
Neither Assumption~\ref{ass:LipschitzGrowth} nor the ensuing
Assumptions impose any non-degeneracy condition on the coefficient $\sigma^d$. 
The case of \textit{degenerate} $\sigma^d$ is admissible so that in particular 
the \textit{pure-jump case $\sigma^d =0$} is included in our setting. 
In particular, 
we have assumed without loss of generality that the Brownian motion 
$B^d$ is $d$-dimensional and $\sigma(x) \in \R^{d \times d} $ for $x \in \R^d$. 

Let us now argue that this also covers the case when $B^d$ in \eqref{eq:SDEnew} 
is replaced by an $r$-dimensional Brownian motion $\tilde{B}$ for some $1\leq r < d$
and $\sigma^d$ is replaced by $\tilde{\sigma}$ with $\tilde{\sigma} \colon \R^d \to \R^{d \times r}$. 
Indeed, 
to include this case we simply set 
$\sigma^d_{i,j}(x)= \tilde{\sigma}_{i,j}(x)$ for $i=1,\ldots,d$, $j=1,\ldots,r$ 
and $\sigma^d_{i,j}(x)=0$ otherwise. 
Then 
$\int_0^t \sigma^d(X_{s_-}^{x,d}) d B^d_s = \int_0^t \tilde{\sigma}(X_{s_-}^{x,d}) d \tilde{B}_s$ 
and so SDE \eqref{eq:SDEnew} coincides with the modified SDE. 
\end{remark}
\begin{remark}
\label{remk:puredrift}[Pure Drift Process, Linear Advection]
Assumption ~\ref{ass:LipschitzGrowth} and the ensuing 
Assumptions \ref{ass:PointwiseLipschitz}, \ref{ass:jumps}, \ref{ass:NNApproxCoeff} admit in \eqref{eq:SDE} 
and in all expression rate estimates in Section \ref{sec:main} ahead 
also the case of deterministic, initial-value ODEs, where in \eqref{eq:SDE} 
$\sigma^d = 0$, $f^d = g^d = 0$. 
In this case, the PIDE \eqref{eq:LevyPDE} 
reduces to a (deterministic) linear transport equation.
Theorem \ref{thm:main} and Corollary \ref{cor:main} ahead therefore
apply also to this setting. 
Our expression rate results therefore comprise 
the CoD-free DNN expression rate bounds obtained recently in \cite{Laakmann2020}
(albeit in pure drift case with less explicit bounds on the exponents $\mathfrak{p}$ 
and $\mathfrak{q}$ in Theorem \ref{thm:main} ahead than in \cite{Laakmann2020}).
\end{remark}
\begin{remark}
\rev{The Lipschitz condition in Assumption~\ref{ass:LipschitzGrowth} requires, in particular, that$ \|\sigma^d(x)-\sigma^d(y)\|_F^2 
\leq L \|x-y\|^2$ with Lipschitz constant $L$ that does not depend on $d$. Intuitively, this means that, as $d$ increases,  only $\cO(d)$ components of the matrix $\sigma^d(\cdot)$ are not constant -- at least up to an orthogonal transformation that may depend on $d$, by invariance of $\|\cdot\|_F$ with respect to such transformations. This hypothesis on $\sigma^d$ covers many relevant high-dimensional examples, for instance from mathematical finance: it includes heat type equations ($\sigma^d$ constant as considered in \cite{Gonon2019}), Black-Scholes type equations ($\sigma^d(x)=\mathrm{diag}(\beta_{d,1} x_1,\ldots,\beta_{d,d} x_d) B^d$ with $\sup_{d \in \N, i \in \{1,\ldots,d\}} |\beta_{d,i}| < \infty$ and $B^d \in \R^{d \times d}$ satisfying for all $i \in \{1,\ldots,d\}$ that $\sum_{j=1}^d |B^d_{i,j}|^2=1$, as considered in \cite{HornungJentzen2018}) and exponential L\'evy-models ($\sigma^d(x)= \mathrm{diag}(x_1,\ldots,x_d) \sqrt{A^d}$ with $\sup_{d \in \N, i,j \in \{1,\ldots,d\}} |A_{i,j}^d| < \infty$ as considered in \cite{GS20_925}). However, it is not limited to such models and also covers coefficients $\sigma^d$ that have a band-structure (tridiagonal, ...), i.e., coefficients for which there exists $m \in \N_{0}$ such that for all $d \in \N$, $x \in \R^d$, $i,j \in \{1,\ldots,d\}$ with $|i-j|>m$ it holds that $\sigma^d(x)_{i,j}=0$. Further examples include matrices with sufficiently strong off-diagonal decay, other types of sparsity or high-dimensional financial models driven  effectively by only $m < d$ assets, as often observed in practice and discussed, e.g.,  in \cite{Hilber2009}. 
}	
\end{remark}
When studying the PIDE \eqref{eq:LevyPDE} 
an additional hypothesis will be used to ensure well-posedness of the PIDE.

\begin{assumption}\label{ass:PointwiseLipschitz}[Pointwise Lipschitz and integrability condition] 
For each $d \in \N$ there exists a constant $C_1(d)>0$ such that
		%Note to ourselves: this condition is compatible w. the other conditions
	%imposed, and ensures uniquenes of viscosity solutions, being Assumption 
	%(A.2 vi) of [BBP97]
	%	
	for all $x,y\in \IR^d, z\in \IR^d\backslash\{ 0 \}$ holds
	\[\begin{aligned}	 
	\| \gamma^d(x,z) \| & \leq C_1 (1\wedge \| z \|)  \;, \\
	\| \gamma^d(x,z) - \gamma^d(y,z) \| & \leq C_1 \| x-y \|(1\wedge \| z \|) \;.
	\end{aligned}
	\]
\end{assumption}

\begin{remark} \label{remk:Condo}
\rev{In contrast to Assumption~\ref{ass:LipschitzGrowth}, the} constant $C_1(d)$ in Assumption \ref{ass:PointwiseLipschitz}
may depend on the dimension $d$ in an unspecific way \rev{and it does not appear in the estimates below}. 
For each fixed $d$, however, \rev{the pointwise Lipschitz and growth conditions on $\gamma^d$ formulated in} Assumption \ref{ass:PointwiseLipschitz} 
\rev{are} stronger than the \rev{Lipschitz and growth} conditions imposed on $\gamma^d$ in Assumption~\ref{ass:LipschitzGrowth}, 
as one easily verifies using the fact that $\int_{\R^d } (1\wedge \| z \|)^2 \nu^d(d z) < \infty$.
Assumption  \ref{ass:PointwiseLipschitz} is identical to the
pointwise assumption \cite[{Section~1}]{BBP1997}.
It is shown in \cite[Theorem 3.5]{BBP1997} 
to ensure uniqueness of viscosity solutions of polynomial growth 
for the PIDE \eqref{eq:LevyPDE}. 
Overcoming the CoD in DNN approximation rate bounds
requires 
conditions (i) and (ii) in  Assumption \ref{ass:LipschitzGrowth} on the precise $d$-(in)dependence.
\end{remark}
In order to derive neural network expression rates of solutions, 
we will express the jump part in the Euler scheme as a neural network. 
This is straightforward in the case of a L\'evy-driven SDE, 
which corresponds to Assumption~\ref{ass:jumps}(i) below. 
Alternatively, 
under a certain non-degeneracy condition (see Assumption~\ref{ass:jumps}(ii) below), 
we will be able to carry out a compound Poisson approximation of the small jumps also for general diffusions with jumps.
\begin{assumption} \label{ass:jumps}
	At least one of the following two conditions holds:
	\begin{itemize}
	\item[(i)] (L\'evy-driven SDE) 
         For all $d \in \N$ there exist functions $F^d \colon \R^d \to \R^{d \times d}$, $G^d \colon \R^d \to \R^d$ 
         such that for all $y,z \in \R^d$ 
		\[
		\gamma^d(y,z) = F^d(y)G^d(z).
		\] 
		\item[(ii)]  (Dimension-explicit control of small jumps) 
         There exist 
         $\tilde{L},\bar{p},\bar{q}>0$ such that for all $d \in \N, x \in \R^d, \delta \in (0,1)$
	\begin{align}\label{eq:smallJumpDecay}
	& \int_{\|z\|\leq \delta} \|\gamma^d(x,z)\|^2 \nu^d(dz) \leq \tilde{L} \delta^{\bar{p}} d^{\bar{q}}  (1+\|x\|^2),
	\\  \label{eq:LevymeasureGrowth} 
           &  \int_{\R^d } (1 \wedge \|z\|^2) \nu^d(d z) \leq \tilde{L} d^{\bar{q}}.
	\end{align}
	\end{itemize}
\end{assumption}
Condition \eqref{eq:LevymeasureGrowth}  requires that the 
L\'evy integral $\int_{\R^d } (1 \wedge \|z\|^2) \nu^d(d z)$ 
only grows polynomially in the dimension $d \in \N$.
To further illustrate condition \eqref{eq:smallJumpDecay}  
we now provide a sufficient condition 
for Assumption~\ref{ass:jumps}(ii).
Condition \eqref{eq:smallJumpDecaySuff} requires a stable-like behaviour at the origin. 
\begin{example}
Suppose there exists $\rho \in (0,2), \tilde{L}>0, \bar{q}>0$ such that 
for all $d \in \N, x \in \R^d$ condition \eqref{eq:LevymeasureGrowth} holds 
and
		\begin{equation}\label{eq:smallJumpDecaySuff}
\int_{\|z\|\leq 1} \frac{\|\gamma^d(x,z)\|^2}{\|z\|^\rho} \nu^d(dz) \leq \tilde{L} d^{\bar{q}} (1+\|x\|^2).  
\end{equation}
Then Assumption~\ref{ass:jumps}(ii) is satisfied. Indeed, for any  $d \in \N, x \in \R^d, \delta \in (0,1)$ we estimate
\begin{equation}\label{eq:smallJumpDecaySuff2} 
\begin{aligned}
\int_{\|z\|\leq \delta} \|\gamma^d(x,z)\|^2 \nu^d(dz)
& = \int_{\|z\|\leq \delta} \frac{\|\gamma^d(x,z)\|^2}{\|z\|^\rho} \|z\|^\rho \nu^d(dz)
\\ &  \leq \delta^\rho \int_{\|z\|\leq 1} \frac{\|\gamma^d(x,z)\|^2}{\|z\|^\rho} \nu^d(dz) 
\\ &  \leq \tilde{L} \delta^{\rho} d^{\bar{q}}  (1+\|x\|^2).
\end{aligned}
\end{equation}
\end{example}

%%%%%%%%%%%%%%%%%%%%%%%%%%%%%%%%%%%%%%%%%%%%%%%%%%%%%%%%%%%%%%%%%%%%%%%%%%%%%%%
\subsection{Existence and Uniqueness}
\label{sec:ExUniq}
%%%%%%%%%%%%%%%%%%%%%%%%%%%%%%%%%%%%%%%%%%%%%%%%%%%%%%%%%%%%%%%%%%%%%%%%%%%%%%%
Assumptions \ref{ass:LipschitzGrowth} and \ref{ass:PointwiseLipschitz} ensure existence
and uniqueness of solutions of both the SDEs \eqref{eq:SDE}, \eqref{eq:SDEnew}
and the Kolmogorov equation \eqref{eq:LevyPDE}. We briefly recapitulate 
the corresponding results, going back to \cite[Theorems 2 and 3]{FujiKuni85}, from \cite{BBP1997}. 
In case of the SDEs \eqref{eq:SDE}, \eqref{eq:SDEnew} Assumption~\ref{ass:LipschitzGrowth} 
is sufficient to guarantee existence and uniqueness of solutions, by \cite[Theorem~6.2.9]{Applebaum2009}, 
as pointed out in Section~\ref{sec:RegGrwt}.
\begin{proposition}\label{eq:ExUniqSDE}
Under Assumption \ref{ass:LipschitzGrowth}, 
\eqref{eq:SDE} and \eqref{eq:SDEnew} each admit a unique global solution.
\end{proposition}
This result is, 
with Assumption \ref{ass:LipschitzGrowth}, 
\cite[Theorem~6.2.9]{Applebaum2009} (cf.\ the discussion in Remark~\ref{remk:ApplB}). 
With Assumptions \ref{ass:LipschitzGrowth} and \ref{ass:PointwiseLipschitz},
the result is \cite[Proposition 1.1]{BBP1997}.
We also note that the SDEs \eqref{eq:SDE}, \eqref{eq:SDEnew} 
are contained in the abstract backward SDE setting of \cite{BBP1997} with 
$f_i=0$ and $\gamma_i = 0$ in \cite[(A.2)]{BBP1997}.
This implies 
that all items in 
Assumptions \cite[(A.1), (A.2)]{BBP1997} are trivially satisfied
and all conclusions of \cite{BBP1997} apply in the present setting 
under Assumptions \ref{ass:LipschitzGrowth} and \ref{ass:PointwiseLipschitz}.
\begin{proposition}\label{prop:ExUniq} 
Let $\varphi_d \colon \R^d \to \R$ be continuous and at most polynomially growing.
Under Assumptions  \ref{ass:LipschitzGrowth} and \ref{ass:PointwiseLipschitz}
there exists a unique viscosity solution (in the sense of \cite[Definition 3.1]{BBP1997})
of the PIDEs \eqref{eq:LevyPDE}, \eqref{eq:LevyPDEnew} with polynomial growth as $|x|\to \infty$.
\end{proposition}
This assertion is \cite[Proposition 2.5]{BBP1997} (polynomial growth) and 
\cite[Theorem 3.4]{BBP1997} (Existence) and \cite[Theorem 3.5]{BBP1997} (Uniqueness), 
upon observing that 
\eqref{eq:LevyPDEnew} coincides with \cite[(3.1)]{BBP1997} since $f_i =0$.
In particular,
Assumptions \ref{ass:LipschitzGrowth} and \ref{ass:PointwiseLipschitz} 
imply the assumptions in \cite[Section 1]{BBP1997}.
%%%%%%%%%%%%%%%%%%%%%%%%%%%%%%%%%%%%%%%%%%%%%%%%%%%%%%%%%%%%%%%%%%%%%%%%%%%%%%%55
\section{Deep neural networks (DNNs)}
\label{sec:DNNs}
%%%%%%%%%%%%%%%%%%%%%%%%%%%%%%%%%%%%%%%%%%%%%%%%%%%%%%%%%%%%%%%%%%%%%%%%%%%%%%%%%
We present in Section \ref{sec:DNNDef} 
notation and assumptions on the deep neural networks (DNNs) on which the ensuing approximation
rate estimates of path-dependent functionals of the SDE \eqref{eq:SDEnew} and 
viscosity solutions of the PIDE \eqref{eq:LevyPDE} will be based. 
Section \ref{sec:DNNCoeff} then provides our precise assumptions on the expression 
rates of the coefficients in the SDE \eqref{eq:SDEnew} respectively PIDE \eqref{eq:LevyPDE}.
%%%%%%%%%%%%%%%%%%%%%%%%%%%%%%%%%%%%%%%%%%%%%%%%%%%%%%%%%%%%%%%%%%%%%%%%%%%%%%%
\subsection{Notation and Definitions of DNNs}
\label{sec:DNNDef}
%%%%%%%%%%%%%%%%%%%%%%%%%%%%%%%%%%%%%%%%%%%%%%%%%%%%%%%%%%%%%%%%%%%%%%%%%%%%%%%
Throughout the article we will consider deep neural networks with the ReLU activation function
$\varrho \colon \R \to \R$ given by $\varrho(x)= x_+ := \max(x,0)$.  \rev{This choice is not essential for the results, which also hold for more general activation functions. From the point of view of applications in mathematical finance, however, the ReLU activation function is the most natural choice, see, e.g.\ Remark~\ref{rmk:European} below.}
For any $d \in \N$ we can lift $\varrho$ to a mapping $\R^d \to \R^d$ 
by the specification $x \mapsto (\varrho(x_i))_{i=1,\ldots,d}$. 
We denote this mapping also by the same symbol $\varrho$. 

Let $d, L \in \N$, $N_0:=d$, $N_1,\ldots,N_L \in \N$ and $b^\ell \in \R^{N_\ell}$, 
$A^\ell \in \R^{N_\ell \times N_{\ell-1}}$ for $\ell = 1,\ldots,L$. 
A (feedforward) deep neural network (DNN) with 
activation function $\varrho$, $L$ layers, $d$-dimensional input, 
weight matrices $A^1,\ldots,A^L$ and biases $b^1,\ldots,b^L$ is the function $\phi \colon \R^d \to \R^{N_L}$
\begin{equation}\label{eq:phidef}
\phi(x) = W_L \circ (\varrho \circ W_{L_1}) \circ \cdots \circ (\varrho \circ W_1)(x), \quad x \in \R^d, 
\end{equation}
where 
$W_\ell \colon \R^{N_{\ell-1}} \to \R^{N_{\ell}}$ 
denotes the affine map
$W_\ell(y) = A^{\ell} y + b^{\ell}$ for $y \in \R^{N_{\ell-1}}$ and $\ell = 1,\ldots,L$. 
Such a function is often simply called a deep neural network. 
The total number of non-zero entries of the weights and biases is called the \textit{size} of the DNN. 
Thus, for a DNN as above we let
\begin{equation}\label{eq:sizedef}
\mathrm{Size}(\phi) := |\{(i,j,\ell) \colon A^\ell_{i,j} \neq 0 \}| + |\{(i,\ell) \colon b^\ell_{i} \neq 0 \}|. 
\end{equation}
We also denote by 
$\mathrm{Size}_{out}(\phi) :=  |\{(i,j) \colon A^L_{i,j} \neq 0 \}| + |\{i \colon b^L_{i} \neq 0 \}|$ 
the number of non-zero entries of the weights and biases of the last layer of the DNN. 
Finally, we denote by $\mathrm{depth}(\phi):=L+1$ the number of layers of the DNN. 

A DNN is often defined as collection of parameters 
$\Phi = ((A^1,b^1),\ldots,(A^L,b^L))$, distinct from
the function $\phi$ in \eqref{eq:phidef} built from $\Phi$.
The latter is referred to as \emph{realization of the DNN $\Phi$}. 
See, e.g., \cite{PETERSEN2018296}, \cite{Opschoor2020}, \cite{GS20_925}. 
Here we follow the notationally lighter approach of \cite{OSZ19_839} 
and do not distinguish between the neural network and its parameter set, 
as the parameter set is (always at least implicitly) part of the definition. 
Note that there may be several parameter choices that lead to the same realization.
In the expression rate bounds under consideration in the present article,
this is not an issue and pathological choices are excluded by the requirements 
that we impose on DNN size. 
%%%%%%%%%%%%%%%%%%%%%%%%%%%%%%%%%%%%%%%%%%%%%%%%%%%%%%%%%%%%%%%%%%%%%%%%
\subsection{DNN approximations of the coefficients}
\label{sec:DNNCoeff}
%%%%%%%%%%%%%%%%%%%%%%%%%%%%%%%%%%%%%%%%%%%%%%%%%%%%%%%%%%%%%%%%%%%%%%%%
We introduce assumptions on the DNN approximation for the coefficients. 
In the case of a L\'evy-driven SDE \eqref{eq:SDE} (i.e., when Assumption~\ref{ass:jumps}(i) holds), 
these assumptions mean that $F^d$ can be approximated well by a neural network. 
In general, $\gamma^d$ is approximated by a neural network. 
\begin{assumption}
\label{ass:NNApproxCoeff}
[NN expression rates of coefficient functions $\beta^d,\sigma^d,\gamma^d$]
Assumption~\ref{ass:jumps} holds and 
there exist constants $C>0$, $p,q,\rev{\hat{q}}\geq 0$ and, for each $d \in \mathbb{N}$, 
and 
for each $\varepsilon \in (0,1]$, 
there exist neural networks $\beta_{\varepsilon,d} \colon \R^d \to \R^d$, 
$\sigma_{\varepsilon,d,j} \colon \R^d \to \R^d$, $j=1,\ldots,d$, 
and functions $\gamma_{\varepsilon,d} \colon \R^d \times \R^d \to \R^d$, 
\rev{such that}
for each $d \in \mathbb{N}$, $\varepsilon \in (0,1]$ \rev{it holds that}
        \begin{itemize} 
        \item[(i)] for all $x \in \R^d$ 
		\[ 
                \begin{aligned}
		\|\beta^d(x)-\beta_{\varepsilon,d}(x)\|^2+\|\sigma^d(x)-\sigma_{\varepsilon,d}(x)\|_F^2 
                 +  
                 \int_{\R^d}\|\gamma^d(x,z)-\gamma_{\varepsilon,d}(x,z)\|^2 \nu^d(d z) 
                 & \leq \varepsilon^{4q+1} C d^p (1+\|x\|^2),
                \\
		\|\beta_{\varepsilon,d}(x)\|^{\rev{2}}+\|\sigma_{\varepsilon,d}(x)\|_F^{\rev{2}}
                +  
                %\left(
                \int_{\R^d} \|\gamma_{\varepsilon,d}(x,z)\|^2 \nu^d(d z)
                %\right)^{1/2} 
                & \leq  C (d^p \varepsilon^{-q} + \|x\|^{\rev{2}}),		
		\\ 
		\mathrm{size}(\beta_{\varepsilon,d}) + \sum_{j=1}^d \mathrm{size}(\sigma_{\varepsilon,d,j}) 
                 & \leq C d^p \varepsilon^{-\rev{\hat{q}}},
		\end{aligned} 
                \]
	\item[(ii)] \begin{itemize} \item \rev{if} Assumption~\ref{ass:jumps}(i) holds, 
                then $\gamma_{\varepsilon,d}(y,z) = F_{\varepsilon,d}(y) G^d(z)$ 
                for DNNs $F_{\varepsilon,d,j} \colon \R^d \to \R^d$, $j=1,\ldots,d$ 
                satisfying $\mathrm{size}(F_{\varepsilon,d,j}) \leq C d^p \varepsilon^{-\rev{\hat{q}}}$\rev{,} 
	\item \rev{o}therwise Assumption~\ref{ass:jumps}(ii) 
  holds and furthermore $\gamma_{\varepsilon,d}$ is a DNN
  with $\mathrm{size}(\gamma_{\varepsilon,d}) \leq C d^p \varepsilon^{-\rev{\hat{q}}}$. 
		\end{itemize}
	\end{itemize}
\end{assumption}
In the case of a L\'evy-driven SDE the 
conditions on the function $\gamma_{\varepsilon,d}$ 
which we imposed in (i) are in fact conditions on $F_{\varepsilon,d}$.
\section{Dimension-explicit bounds for SDEs with jumps}
\label{sec:prelim}
This section provides approximations for the Euler scheme for the SDE \eqref{eq:SDE} 
as well as further approximation steps to truncate the small jumps of the L\'evy measure 
to a set 
$A_\delta = \{z \in \R^d \colon \|z\|>\delta \}$ for some $\delta > 0$, 
to approximate the coefficients of the SDE and to 
provide a Monte Carlo approximation of an integral involving $\nu^d$. 

These approximations and bounds on errors incurred by them
are needed for the proof of the main results in Section~\ref{sec:MainRes} 
in order to approximate the underlying SDE by a process whose sample paths can be emulated by a DNN. 
In the case of a L\'evy-driven SDE, the multiplicative structure 
allows for a simpler DNN emulation approach and hence 
the small jumps of the L\'evy process do not need to be truncated, 
i.e., we may use $\delta=0$. 
Thus, two different approaches are used for the 
two alternative hypotheses Assumption~\ref{ass:jumps}(i) 
and Assumption~\ref{ass:jumps}(ii), corresponding to choosing $\delta=0$ and $\delta>0$ below.
\subsection{Discrete-time approximation}
\label{sec:timestp}
The following auxiliary results are crucial ingredients for our subsequent analysis of DNN expression rates. 

We start with a lemma that provides bounds on the moments of $X^{x,d}$ 
and shows that under Assumption~\ref{ass:LipschitzGrowth} 
the second moments grow at most polynomially in $d$ and $\|x\|$.
\rev{The constants $c_1,c_2 >0$ in Lemma~\ref{lem:momentBound} only depend on $L$ and $T$. Without stronger assumptions the dependence on $L$ and $T$ may be exponential in general.} 
\begin{lemma}\label{lem:momentBound}
Suppose Assumption~\ref{ass:LipschitzGrowth} holds. 
Then there exist constants $c_1,c_2 >0$ 
such that for all $d \in \N, x \in \R^d, t \in [0,T]$, 
\begin{equation}\label{eq:momentBound}
\E[\|X^{x,d}_t\|^2] \leq c_1 \|x\|^2 + c_2 d^2.
\end{equation}
\end{lemma}
\begin{proof} Let $d \in \N$ and $x \in \R^d$.
	We start by establishing 
\begin{equation} \label{eq:XTintegralFinite}
\E\left[\int_0^T \|X^{x,d}_t\|^2 dt \right] < \infty.
\end{equation}
This essentially follows from \cite[Corollary~6.2.4]{Applebaum2009}. 
More specifically, 
by \cite[Corollary~6.2.4]{Applebaum2009} it follows that $X^{x,d}_t$ 
is indeed in $L^2(\Omega,\Fc,\P)$ for all $t \in [0,T]$ and 
\begin{equation}\label{eq:ApplebaumMoment}
\E[\|X^{x,d}_t\|^2]\leq 2 \max(1,C(t)^2)(1+\|x\|^2)
\end{equation}
where  
$C(t)$ is given from the proof of \cite[Theorem~6.2.4]{Applebaum2009} 
as $C(t)=\sum_{n=1}^\infty \frac{C_2(t)^{n/2}K_3^{n/2}}{(n!)^{1/2}} $ 
with $C_2(t)=t \max(3t,12)$, $K_3 =L(1+\|x\|^2)$. 
\rev{Now note that $C_2(t)\leq C_2(T)$ for all $t \in [0,T]$ and consequently  $C(t) \leq C(T)$. But $C(T)$ is finite by the ratio test and therefore \eqref{eq:ApplebaumMoment} implies \[
\E\left[\int_0^T \|X^{x,d}_t\|^2 dt \right] \leq 	
	T \sup_{t \in [0,T]} \E[\|X^{x,d}_t\|^2]\leq 2 T \max(1,C(T)^2)(1+\|x\|^2) < \infty.
\]
}
	
%We cannot directly use this estimate, to deduce \eqref{eq:momentBound}, 
%but we can employ it to estimate using Tonelli's theorem 
%and the Minkowski integral inequality
%\[ \begin{aligned}
%\E\left[\int_0^T \|X^{x,d}_t\|^2 dt \right] & \leq 2(1+\|x\|^2)\left(1+\int_0^T C(t)^2 dt\right)
%\\ & \leq 2(1+\|x\|^2)\left(1+\left[\sum_{n=1}^\infty \left(\int_0^T \frac{C_2(t)^{n}K_3^{n}}{n!} dt\right)^{\frac{1}{2}} \right]^2 \right)
%\\ & \leq 2(1+\|x\|^2)\left(1+\left[\sum_{n=1}^\infty \left(\int_0^T \frac{t^{n}[(3t)^n+12^n]K_3^{n}}{n!} dt\right)^{\frac{1}{2}} \right]^2 \right)
%\\ & \leq 2(1+\|x\|^2)\left(1+\left[\sum_{n=1}^\infty \left(\frac{[\frac{3^n T^{2n+1}}{2n+1}+\frac{T(12T)^n}{n+1}]K_3^{n}}{n!} \right)^{\frac{1}{2}} \right]^2 \right),
%\end{aligned} \]
%and so \eqref{eq:XTintegralFinite} follows by the ratio test.
Having established \eqref{eq:XTintegralFinite}, 
we can employ Assumption~\ref{ass:LipschitzGrowth} and \eqref{eq:XTintegralFinite} to verify for all $i,j \in \{1,\ldots,d\}$
\begin{equation} \label{eq:auxEq1} \begin{aligned}
\E\left[\int_0^T |\sigma^d_{i,j}(X_{t}^{x,d})|^2 dt \right]  & \leq 2\E\left[\int_0^T |\sigma^d_{i,j}(X_{t}^{x,d})-\sigma^d_{i,j}(0)|^2 dt \right] + 2 T |\sigma^d_{i,j}(0)|^2  
\\ & \leq 2 L \E\left[\int_0^T \|X_{t}^{x,d}\|^2 dt\right]+ 2 T |\sigma^d_{i,j}(0)|^2  < \infty
\end{aligned}
\end{equation}
and similarly
\begin{equation} \label{eq:auxEq2} \begin{aligned}
\E\left[\int_0^T \int_{\R^d} |\gamma^d_i(X_{t_-}^{x,d},z)|^2 \nu^d(dz) dt  \right]  & \leq 2\E\left[\int_0^T L \|X_{t}^{x,d}\|^2  dt \right] + 2 T \int_{\R^d} |\gamma^d_i(0,z)|^2 \nu^d(dz)  < \infty.
\end{aligned}
\end{equation}

Set $\rev{\bar{G}}(t)=\E[\|X^{x,d}_t\|^2]$ for $t \in [0,T]$. 
Using Minkowski's inequality, \eqref{eq:auxEq1}, \eqref{eq:auxEq2}, 
It\^o's isometry and the Minkowski integral inequality we obtain
\begin{equation}\label{eq:auxEq6} \begin{aligned}
\rev{\bar{G}}(t)^{1/2} & \leq \|x\| + \E\left[\left\|\int_0^t \beta^d(X_{s}^{x,d}) ds\right\| ^2\right]^{1/2} + \E\left[\left\|\int_0^t \sigma^d(X_{s}^{x,d}) d B^d_s \right\| ^2\right]^{1/2} \\ & \quad + \E\left[\left\|\int_0^t \int_{\R^d} \gamma^d(X_{t_-}^{x,d},z) \tilde{N}^d(ds,dz) \right\| ^2\right]^{1/2}
\\ & \leq \|x\| + \int_0^t \E[ \|\beta^d(X_{s}^{x,d})\|^2]^{1/2} ds + \left[\int_0^t \E[\|\sigma^d(X_{s}^{x,d})\|_F^2]ds\right]^{1/2} \\ & \quad + \left(\int_0^t \int_{\R^d} \E[\|\gamma^d(X_{t_-}^{x,d},z)\|^2] \nu^d(dz) ds \right)^{1/2}.
\end{aligned}
\end{equation} 	
We now consider these terms separately. For the first integral, 
\begin{equation}\label{eq:auxEq3} \begin{aligned}
\int_0^t \E[ \|\beta^d(X_{s}^{x,d})\|^2]^{1/2} ds & \leq  \int_0^t \E[ \|\beta^d(X_{s}^{x,d})-\beta^d(0)\|^2]^{1/2} ds +t\|\beta^d(0)\|
\\ & \leq L T^{1/2} \left(\int_0^t \rev{\bar{G}}(s) ds\right)^{1/2} +  t d^{1/2} L^{1/2}.
\end{aligned} \end{equation}
For the second one, we similarly estimate
\begin{equation}\label{eq:auxEq4} \begin{aligned}
\left[\int_0^t \E[\|\sigma^d(X_{s}^{x,d})\|_F^2]ds\right]^{1/2} & \leq   \left[2 \int_0^t \E[\|\sigma^d(X_{s}^{x,d})-\sigma^d(0)\|_F^2]ds + 2 t \|\sigma^d(0)\|_F^2  \right]^{1/2}
\\ & \leq \left[2 L \int_0^t \rev{\bar{G}}(s) ds + 2 t d^2 L  \right]^{1/2}.
\end{aligned} \end{equation}
For the last one, we obtain analogously
\begin{equation}\label{eq:auxEq5} \begin{aligned}
\left(\int_0^t \int_{\R^d} \E[\|\gamma^d(X_{t_-}^{x,d},z)\|^2] \nu^d(dz) dt \right)^{1/2} & \leq  \left(2 \int_0^t \E \left[ \int_{\R^d} \|\gamma^d(X_{t_-}^{x,d},z)-\gamma^d(0,z)\|^2 \nu^d(dz) \right] ds + 2 t d L \right)^{1/2} 
\\ & \leq \left[2 L \int_0^t \rev{\bar{G}}(s) ds + 2 t d^2 L  \right]^{1/2}.
\end{aligned} \end{equation}
Inserting \eqref{eq:auxEq3}, \eqref{eq:auxEq4} and \eqref{eq:auxEq5} in \eqref{eq:auxEq6} and using that for all $a,b,c\geq0$ it holds $(a+b+c)^2\leq 3(a^2+b^2+c^2)$ we obtain
\begin{equation}\label{eq:auxEq7} \begin{aligned}
\rev{\bar{G}}(t) & \leq 3(\|x\|^2+12Td^2L) + 30 L \int_0^t \rev{\bar{G}}(s) ds.
\end{aligned}
\end{equation}
Gr\rev{\"o}nwall's inequality and \eqref{eq:XTintegralFinite} hence prove that for all $t \in [0,T]$ we have $\rev{\bar{G}}(t) \leq a \exp(b t)$ with $a=3(\|x\|^2+12Td^2L)$ and $b=30 L$. Setting $c_1 = 3 \exp(bT)$ and $c_2 =36 T L \exp(bT)$, this proves the assertion. 
\end{proof}
\begin{remark}
Note that the estimate \eqref{eq:ApplebaumMoment} can not be directly used to deduce \eqref{eq:momentBound}, 
because $K_3$ depends on the Euclidean norm of the initial value $x$ 
and 
the constant $C(t)$ in \eqref{eq:momentBound} 
is lower bounded by
$C(t)\geq \sum_{n=1}^\infty \frac{C_2(t)^{n/2}K_3^{n/2}}{n!} 
 = 
 \exp(C_2(t)^{1/2}K_3^{1/2}) -1$. 
Recall that 
$C_2(t)=t \max(3t,12)$, $K_3 =L(1+\|x\|^2)$ 
and so \rev{from the estimate \eqref{eq:ApplebaumMoment} we could only obtain a bound that is exponential in $\|x\|$. This, in turn, %for instance for $x=(1,\ldots,1)$ we have 
%$C(t) \geq \exp(c d)$ for some $c>0$ not depending on $d \in \N$. 
%Hence,  
would lead to far stronger conditions on the probability measure $\mu^d$ (used below to measure the approximation error) or to constants that grow exponentially in $d$.}
\end{remark}
In the next step 
we carry out a first approximation step based on the Euler-Maruyama scheme. 
To do so, let $h=\frac{T}{N}$, $N \in \N$, denote a step size 
and, 
for $t \in [0,T]$, 
let $\lfloor t \rfloor_h = \max \{ s \in h \N \,:\, s \leq t \}$ 
denote the largest discretization time below or equal to $t$.
The Euler discretization of $X^{x,d}$ 
is then defined by $\hat{X}_{0}^{x,d,h} = x$ and for $n=1,\ldots,N$,
\begin{equation} \label{eq:Euler}
\hat{X}_{n}^{x,d,h} 
= 
\hat{X}_{n-1}^{x,d,h} + \beta^d(\hat{X}_{n-1}^{x,d,h}) h 
+ 
\sigma^d(\hat{X}_{n-1}^{x,d,h}) (B^d_{nh}-B^d_{(n-1)h}) 
+ 
\int_{(n-1) h}^{nh}\int_{\R^d} \gamma^d(\hat{X}_{n-1}^{x,d,h},z) \tilde{N}^d(dt,dz).
\end{equation}
To prove that $X^{x,d}_{n h} \approx \hat{X}_{n}^{x,d,h}$ in a suitable sense
we define the interpolation (or continuous-time Euler) approximation 
as the solution to stochastic differential equation
$\bar{X}_0^{x,d,h} = x$, and
\begin{equation} \label{eq:SDEEuler} \begin{aligned}
 	d \bar{X}_t^{x,d,h} 
        = 
        \beta^d(\bar{X}_{\lfloor t- \rfloor_h}^{x,d,h}) d t 
        + 
         \sigma^d(\bar{X}_{\lfloor t- \rfloor_h}^{x,d,h}) d B^d_t 
        + \int_{\R^d} \gamma^d(\bar{X}_{\lfloor t- \rfloor_h}^{x,d,h},z) \tilde{N}^d(dt,dz), \quad t \in (0,T]
\end{aligned} \end{equation}
with $\bar{X}_{\lfloor t- \rfloor_h}^{x,d,h}:=\lim_{s \to t, s<t}\bar{X}_{\lfloor s \rfloor_h}^{x,d,h} $. 
Then $\bar{X}^{x,d,h}$ is an adapted c\`adl\`ag process and 
by definition $\bar{X}_{nh}^{x,d,h} = \hat{X}_{n}^{x,d,h}$ for all $n=0,\ldots,N$ 
and so $\bar{X}^{x,d,h}$ can be viewed as pathwise temporal 
interpolation of $\hat{X}_{n}^{x,d,h}$.

The next lemma proves that 
under Assumption~\ref{ass:LipschitzGrowth} 
the Euler scheme approximates $X^{x,d}$ without the CoD. 
We remark that the supremum that appears in \eqref{eq:EulerRateCont} is indeed measurable
(we assumed that our probability space is complete and both processes are adapted and c\`adl\`ag).
\rev{The constants $c_3,c_4 >0$ in Lemma~\ref{lem:Euler} only depend on $L$ and $T$.}
\begin{lemma} \label{lem:Euler} 
Suppose that Assumption~\ref{ass:LipschitzGrowth} holds. 
Then there exist constants $c_3,c_4 >0$ such that for all $d \in \N, x \in \R^d, h > 0$ 
the Euler discretization with step size $h$ satisfies
\begin{align} \label{eq:EulerRateCont}
\E\left[\sup_{t \in [0,T]}\|X_t^{x,d} - \bar{X}_t^{x,d,h} \|^2\right] & \leq h(c_3 d^4 + c_4 d^2\|x\|^2).
\end{align}
\end{lemma}

\begin{proof} 	
Let $d \in \N$, $x \in \R^d$, $h>0$.
Define $\rev{\bar{G}}(t) = \E[\sup_{s \in [0,t]}\|X_s^{x,d} - \bar{X}_s^{x,d,h} \|^2]$ for $t \in [0,T]$. \rev{The Lipschitz properties assumed in Assumption~\ref{ass:LipschitzGrowth}, the triangle inequality, \eqref{eq:XTintegralFinite} and the fact that $\int_0^T \E[\|\bar{X}_{\lfloor r \rfloor_h}^{x,d,h}\|^2] dr < \infty$ 
(which can be deduced by an inductive argument) show that the processes $\int_0^\cdot \sigma^d(X_{r_-}^{x,d})-\sigma^d(\bar{X}_{\lfloor r- \rfloor_h}^{x,d,h}) d B^d_r$ and $\int_0^\cdot \int_{\R^d} \gamma^d(X_{r_-}^{x,d},z)-\gamma^d(\bar{X}_{\lfloor r- \rfloor_h}^{x,d,h},z) \tilde{N}^d(dr,dz) $ are martingales. Next, note that for any $d$-dimensional martingale $M$ its norm $\|M\|$ is a submartingale and hence, by Doob's martingale inequality, $\E[\left(\sup_{0\leq s \leq t}\|M_s\|\right)^2]^{1/2} \leq 2 \E[\|M_t\|^2]^{1/2}$. }
Inserting  \eqref{eq:SDEEuler} and \eqref{eq:SDEnew} we \rev{thus} obtain by the triangle inequality and 
Doob's martingale inequality %(using that for any martingale $M$ its norm $\|M\|$ is a submartingale and hence $\E[]$)
\begin{equation}\label{eq:auxEq8} \begin{aligned}
\rev{\bar{G}}(t) & \leq 3 \E\left[\sup_{s \in [0,t]}\left\|\int_0^s \beta^d(X_{r_-}^{x,d}) - \beta^d(\bar{X}_{\lfloor r- \rfloor_h}^{x,d,h}) d r\right\|^2 \right]  + 3\E\left[\sup_{s \in [0,t]}\left\|\int_0^s \sigma^d(X_{r_-}^{x,d})-\sigma^d(\bar{X}_{\lfloor r- \rfloor_h}^{x,d,h}) d B^d_r\right\|^2\right]  \\ & \quad +  3\E\left[\sup_{s \in [0,t]}\left\|\int_0^s \int_{\R^d} \gamma^d(X_{r_-}^{x,d},z)-\gamma^d(\bar{X}_{\lfloor r- \rfloor_h}^{x,d,h},z) \tilde{N}^d(dr,dz) \right\|^2\right]
\\ & \leq 3 t\int_0^t \E\left[\left\|\beta^d(X_{r_-}^{x,d}) - \beta^d(\bar{X}_{\lfloor r- \rfloor_h}^{x,d,h})\right\|^2\right] d r   + 12\E\left[\left\|\int_0^t \sigma^d(X_{r_-}^{x,d})-\sigma^d(\bar{X}_{\lfloor r- \rfloor_h}^{x,d,h}) d B^d_r\right\|^2\right]  \\ & \quad +  12\E\left[\left\|\int_0^t \int_{\R^d} \gamma^d(X_{r_-}^{x,d},z)-\gamma^d(\bar{X}_{\lfloor r- \rfloor_h}^{x,d,h},z) \tilde{N}^d(dr,dz) \right\|^2\right]
\\ & = 3 t\int_0^t \E\left[\left\|\beta^d(X_{r_-}^{x,d}) - \beta^d(\bar{X}_{\lfloor r- \rfloor_h}^{x,d,h})\right\|^2\right] d r   + 12\int_0^t \E\left[\left\|\sigma^d(X_{r_-}^{x,d})-\sigma^d(\bar{X}_{\lfloor r- \rfloor_h}^{x,d,h})\right\|_F^2 \right] dr \\ & \quad +  12\int_0^t \int_{\R^d}\E\left[\left\| \gamma^d(X_{r_-}^{x,d},z)-\gamma^d(\bar{X}_{\lfloor r- \rfloor_h}^{x,d,h},z) \right\|^2\right] \nu^d(dz) dr 
\\ & \leq 3\max(3t,12) L \int_0^t \E\left[ \left\| X_{r}^{x,d}-\bar{X}_{\lfloor r \rfloor_h}^{x,d,h} \right\|^2 \right] d r.
\end{aligned}
\end{equation}
This \rev{implies, in particular, that $\sup_{t \in [0,T]} \rev{\bar{G}}(t)  \leq 3\max(3T,12) L \int_0^T \E\left[ \left\| X_{r}^{x,d}-\bar{X}_{\lfloor r \rfloor_h}^{x,d,h} \right\|^2 \right] d r$ and hence} the triangle inequality, the square integrability 
established in \eqref{eq:XTintegralFinite} and the fact that 
$\int_0^{\rev{T}} \E[\|\bar{X}_{\lfloor r \rfloor_h}^{x,d,h}\|^2] dr < \infty$ 
(\rev{see above})  
allow us to conclude that $\rev{\bar{G}} \in L^1([0,T])$. 
In addition, for all $r \in [0,T]$
\begin{equation}\label{eq:auxEq10}
\E\left[ \left\| X_{r}^{x,d}-\bar{X}_{\lfloor r \rfloor_h}^{x,d,h} \right\|^2 \right] 
\leq 
2 \E\left[ \left\| X_{r}^{x,d}-X_{\lfloor r \rfloor_h}^{x,d} \right\|^2 \right]
+ 
2 \E\left[ \left\| X_{\lfloor r \rfloor_h}^{x,d}-\bar{X}_{\lfloor r \rfloor_h}^{x,d,h} \right\|^2 \right].
\end{equation}
To estimate the first term in the bound \eqref{eq:auxEq10},
we apply It\^o's isometry to obtain for any $r \in [0,T]$ 
\begin{equation}\begin{aligned}\label{eq:auxEq11}
 \E\left[ \left\| X_{r}^{x,d}-X_{\lfloor r \rfloor_h}^{x,d} \right\|^2 \right] 
& 
\leq 
3  \E\left[ \left\| \int_{\lfloor r \rfloor_h}^r \beta^d(X_{ t- }^{x,d}) dt \right\|^2 \right] 
+ 3\E\left[\left\|\int_{\lfloor r \rfloor_h}^r \sigma^d(X_{ t- }^{x,d}) d B^d_t\right\|^2 \right] 
\\ & \quad 
+ 3\E\left[\left\|\int_{\lfloor r \rfloor_h}^r \int_{\R^d} \gamma^d(X_{ t- }^{x,d},z) \tilde{N}^d(dt,dz) \right\|^2 \right]
\\ 
& \leq  3 (r-\lfloor r \rfloor_h)\E\left[ \int_{\lfloor r \rfloor_h}^r 
       \left\|\beta^d(X_{ t- }^{x,d})\right\|^2 dt  \right] 
      + 3\int_{\lfloor r \rfloor_h}^r \E[\|\sigma^d(X_{ t- }^{x,d})\|_F^2] dt  
\\ 
& \quad 
      + 3\int_{\lfloor r \rfloor_h}^r \int_{\R^d} \E[\|\gamma^d(X_{ t- }^{x,d},z)\|^2] \nu^d(dz) d t
\\ 
& \leq  
3L [(r-\lfloor r \rfloor_h)d+d^2+d] \int_{\lfloor r \rfloor_h}^r 1+\E[\|X_{ t}^{x,d}\|^2] dt .
\end{aligned}
\end{equation}
Denote by $c_1$,$c_2$ the constants $c_1,c_2 >0$ (independent of $d$ and $x$) 
from Lemma~\ref{lem:momentBound} which satisfy for all $t \in [0,T]$ the bound \eqref{eq:momentBound}.
Inserting \eqref{eq:momentBound} into \eqref{eq:auxEq11}  
we obtain that for any $r \in [0,T]$ holds
\begin{equation}\begin{aligned}\label{eq:auxEq12}
\E\left[ \left\| X_{r}^{x,d}-X_{\lfloor r \rfloor_h}^{x,d} \right\|^2 \right]  
& \leq  9L d^2 (r-\lfloor r \rfloor_h) (1+c_1 \|x\|^2 + c_2 d^2).
\end{aligned}
\end{equation}
Inserting \eqref{eq:auxEq10} and \eqref{eq:auxEq12} into estimate \eqref{eq:auxEq8} 
gives for all $t \in [0,T]$
\[ \begin{aligned}
\rev{\bar{G}}(t) & \leq 6\max(3t,12)L \left( \int_0^t \rev{\bar{G}}(r) dr + \int_0^t \E\left[ \left\| X_{r}^{x,d}-X_{\lfloor r \rfloor_h}^{x,d} \right\|^2 \right] dr \right)
\\ & \leq 6\max(3T,12)L \left( \int_0^t \rev{\bar{G}}(r) dr + 9T L d^2 h (1+c_1 \|x\|^2 + c_2 d^2) \right).
\end{aligned}
\]
Gr\rev{\"o}nwall's inequality thus proves that for all $t \in [0,T]$
\[
\rev{\bar{G}}(t) \leq 6\max(3T,12) 9T L^2 d^2 h (1+c_1 \|x\|^2 + c_2 d^2) \exp(6\max(3T,12)Lt).
\]
Setting $a = 6\max(3T,12) 9T L^2 \exp(6\max(3T,12)LT)$ this proves \eqref{eq:EulerRateCont} with $c_3 = a (1+c_2)$, $c_4= a c_1$.
\end{proof}
\rev{The same techniques can be used to deduce the following moment bound. The constants $\bar{c}_3,\bar{c}_4 >0$ in Corollary~\ref{cor:supL2} only depend on $L$ and $T$.}
\begin{corollary}\label{cor:supL2}
Suppose Assumption~\ref{ass:LipschitzGrowth} holds. 
Then there exist constants $\bar{c}_3,\bar{c}_4 > 0$ 
such that for all $d \in \N, x \in \R^d$ 
\begin{align} \label{eq:supL2}
\E\left[\sup_{t \in [0,T]}\|X_t^{x,d}\|^2\right] & \leq \bar{c}_3 d^2 + \bar{c}_4 \|x\|^2.
\end{align}
\end{corollary}
\begin{proof}
Let $d \in \N$, $x \in \R^d$. 
Consider $\rev{\bar{G}}(t)= \E\left[\sup_{t \in [0,T]}\|X_t^{x,d} - x\|^2\right]$. 
Using precisely the same arguments employed to obtain \eqref{eq:auxEq8} 
and then Assumption~\ref{ass:LipschitzGrowth} as in \eqref{eq:auxEq3}-\eqref{eq:auxEq5},
we get
\[ \begin{aligned}
\rev{\bar{G}}(t) & \leq 3\max(3t,12) \int_0^t \E\left[\left\|\beta^d(X_{r}^{x,d})\right\|^2 + \left\|\sigma^d(X_{r}^{x,d})\right\|_F^2  + \int_{\R^d}\left\| \gamma^d(X_{r}^{x,d},z)\right\|^2\nu^d(dz) \right]  dr 
\\ & \leq 3 \max(3t,12)
\int_0^t 2 L \E\left[\left\|X_{r}^{x,d} \right\|^2\right] + 2\left\|\beta^d(0)\right\|^2 + 2 \left\|\sigma^d(0)\right\|_F^2  + 2\int_{\R^d}\left\| \gamma^d(0,z)\right\|^2\nu^d(dz) dr
\\ & \leq 6 L \max(3T,12) \left( 2 \int_0^t \rev{\bar{G}}(r) dr + 2T\|x\|^2 + Td^2 \right).
\end{aligned}\]
 Gr\rev{\"o}nwall's inequality thus proves that for all $t \in [0,T]$ we have $\rev{\bar{G}}(t) \leq b e^{a t}$ with $a = 12 L \max(3T,12)$, $b = 6 L \max(3T,12)T (2\|x\|^2 + d^2) $. 
This \rev{and the triangle inequality} prove  \eqref{eq:supL2} with $\bar{c}_3 = \rev{12} L \max(3T,12) T e^{a T}$, $\bar{c}_4= \rev{2+4} \bar{c}_3$. 
\end{proof}
\subsection{Small-jump truncation}
\label{sec:SmJmpTrunc}
In a next step we carry out an approximation procedure 
that allows us to remove the small jumps of the process $\bar{X}^{x,d,h}$. 
In case of a L\'evy-driven SDE (that is, when Assumption~\ref{ass:jumps}(i) is satisfied) 
this procedure is not required in the proof of Theorem~\ref{thm:main} 
and so in the current subsection we work exclusively under Assumption~\ref{ass:jumps}(ii).

For $\delta > 0$ we introduce the set of jumps of size at least $\delta$, i.e.
$A_\delta = \{z \in \R^d \colon \|z\|>\delta \}$.
We consider the truncated
continuous-time Euler approximation $Y_t^{x,d,h,\delta}$, 
which is the unique c\`adl\`ag process satisfying $Y_0^{x,d,h,\delta} = x$, 
\begin{equation} \label{eq:SDEEulerTruncated} \begin{aligned}
d Y_t^{x,d,h,\delta} = \beta^d(Y_{\lfloor t- \rfloor_h}^{x,d,h,\delta}) d t + \sigma^d(Y_{\lfloor t- \rfloor_h}^{x,d,h,\delta}) d B^d_t + \int_{A_\delta} \gamma^d(Y_{\lfloor t- \rfloor_h}^{x,d,h,\delta},z) \tilde{N}^d(dt,dz), \quad t \in (0,T].
\end{aligned} \end{equation}

\begin{remark}
Recall that \eqref{eq:SDEEuler} means that for $n=0,\ldots,N-1$,
the interpolation satisfies for  $t \in [t_n,t_{n+1}]$
\begin{equation} 
\label{eq:SDEEulerIntegrated} 
\begin{aligned}
\bar{X}_t^{x,d,h} 
= \hat{X}_n^{x,d,h} + \beta^d(\hat{X}_n^{x,d,h})(t-t_n) 
   + \sigma^d(\hat{X}_n^{x,d,h}) (B^d_{t}-B^d_{t_n}) 
   + \int_{t_n}^t \int_{\R^d} \gamma^d(\hat{X}_n^{x,d,h},z) \tilde{N}^d(dt,dz),
\end{aligned} 
\end{equation}
where $t_n = h n$, $n=0,\ldots,N$.
Similarly, \eqref{eq:SDEEulerTruncated} means that 
\begin{equation} \label{eq:SDEEulerIntegratedTruncated} \begin{aligned}
Y_t^{x,d,h,\delta} = Y_{t_n}^{x,d,h,\delta} + \beta^d(Y_{t_n}^{x,d,h,\delta})(t-t_n)&  + \sigma^d(Y_{t_n}^{x,d,h,\delta}) (B^d_{t}-B^d_{t_n}) \\ & + \int_{t_n}^t \int_{A_\delta} \gamma^d(Y_{t_n}^{x,d,h,\delta},z) \tilde{N}^d(dt,dz), \quad t \in [t_n,t_{n+1}].
\end{aligned} \end{equation}
\end{remark}

\rev{The following lemma bounds the error that arises from truncating the small jumps. The constants $c_5,c_6 >0$ in Lemma~\ref{lem:jumpTruncation} only depend on $L$, $\tilde{L}$ and $T$.}
\begin{lemma}\label{lem:jumpTruncation}  
Let Assumptions~\ref{ass:LipschitzGrowth} and \ref{ass:jumps}(ii) hold. 

Then 
there exist constants $c_5,c_6 >0$ 
such that for all $d \in \N, x \in \R^d$, $h \in (0,1)$, $\delta > 0$ 
\begin{align} \label{eq:SmallJumpTruncation}
\E\left[\sup_{t \in [0,T]}\|{Y}_t^{x,d,h,\delta} - \bar{X}_t^{x,d,h} \|^2\right] 
& \leq c_5 h(d^4 + d^2\|x\|^2) + c_6 \delta^{\bar{p}} d^{\bar{q}} (\|x\|^2+d^2).
\end{align}
\end{lemma}
\begin{proof} Setting $\rev{\bar{G}}(t) = \E[\sup_{s \in [0,t]}\|{Y}_s^{x,d,h,\delta} - \bar{X}_s^{x,d,h} \|^2]$
and employing precisely the same arguments as in \eqref{eq:auxEq8} we obtain 
\begin{equation}\label{eq:auxEq13} \begin{aligned}
\rev{\bar{G}}(t)  & \leq 3 t\int_0^t \E\left[\left\|\beta^d(Y_{\lfloor r- \rfloor_h}^{x,d,h,\delta}) - \beta^d(\bar{X}_{\lfloor r- \rfloor_h}^{x,d,h})\right\|^2\right] d r   + 12\int_0^t \E\left[\left\|\sigma^d(Y_{\lfloor r- \rfloor_h}^{x,d,h,\delta})-\sigma^d(\bar{X}_{\lfloor r- \rfloor_h}^{x,d,h})\right\|_F^2 \right] dr \\ & \quad +  12\int_0^t \int_{\R^d}\E\left[\left\| \gamma^d(Y_{\lfloor r- \rfloor_h}^{x,d,h,\delta},z) \mathbbm{1}_{A_\delta}(z)-\gamma^d(\bar{X}_{\lfloor r- \rfloor_h}^{x,d,h},z) \right\|^2\right] \nu^d(dz) dr 
\\ & \leq 3\max(3t,24) L \int_0^t \E\left[ \left\| Y_{\lfloor r \rfloor_h}^{x,d,h,\delta}-\bar{X}_{\lfloor r \rfloor_h}^{x,d,h} \right\|^2 \right] d r + 24\int_0^t \int_{\R^d\setminus A_\delta}\E\left[\left\| \gamma^d(\bar{X}_{\lfloor r \rfloor_h}^{x,d,h} ,z) \right\|^2\right] \nu^d(dz) dr 
\\ & \leq 3\max(3t,24) L \int_0^t \rev{\bar{G}}(r) d r + 24\int_0^t \E\left[\int_{\R^d\setminus A_\delta}\left\| \gamma^d(\bar{X}_{\lfloor r \rfloor_h}^{x,d,h} ,z) \right\|^2 \nu^d(dz) \right] dr .
\end{aligned}
\end{equation}	
To estimate the last term, we first use the Lipschitz-condition and Assumption~\ref{ass:jumps}~(ii) and then use Lemma~\ref{lem:momentBound} and Lemma~\ref{lem:Euler} to obtain 
\begin{equation}\label{eq:auxEq14} \begin{aligned}
\int_0^T & \E\left[\int_{\R^d\setminus A_\delta}\left\| \gamma^d(\bar{X}_{\lfloor r \rfloor_h}^{x,d,h} ,z) \right\|^2 \nu^d(dz) \right] dr 
\\ &  \leq 2 \int_0^T \E\left[\int_{\R^d} \left\| \gamma^d(\bar{X}_{\lfloor r \rfloor_h}^{x,d,h} ,z)-\gamma^d(X_{\lfloor r \rfloor_h}^{x,d} ,z) \right\|^2 \nu^d(dz) + \int_{\|z\|\leq \delta} \left\| \gamma^d(X_{\lfloor r \rfloor_h}^{x,d} ,z) \right\|^2 \nu^d(dz) \right] dr
\\ &  \leq 2 \int_0^T L \E\left[ \left\| \bar{X}_{\lfloor r \rfloor_h}^{x,d,h} - X_{\lfloor r \rfloor_h}^{x,d} \right\|^2\right] +  \delta^{\bar{p}} d^{\bar{q}} \tilde{L} (1+\E[\|X_{\lfloor r \rfloor_h}^{x,d}\|^2]) dr
\\ &  \leq 2 T L h(c_3 d^4 + c_4 d^2\|x\|^2) + 2 T \delta^{\bar{p}} d^{\bar{q}} \tilde{L} (1+c_1 \|x\|^2 + c_2 d^2),
\end{aligned}
\end{equation}
where $c_1,c_2,c_3,c_4$ denote the constants from Lemma~\ref{lem:momentBound} and Lemma~\ref{lem:Euler}, which do not depend on $d \in \N, x \in \R^d,  h \in (0,1)$.

Gr\rev{\"o}nwall's inequality therefore shows that 
\[
\rev{\bar{G}}(t) \leq b e^{a t}
\]
with $a = 3\max(3T,24) L$, $b = 48 T L h(c_3 d^4 + c_4 d^2\|x\|^2) + 48 T \delta^{\bar{p}} d^{\bar{q}} \tilde{L} (1+c_1 \|x\|^2 + c_2 d^2)$ 
and so \eqref{eq:SmallJumpTruncation} 
follows with $c_5 = 48 T L \exp(3\max(3T,24) LT) \max(c_3,c_4)$, $c_6 = 48 T \tilde{L} \exp(3\max(3T,24) LT) \max(c_1,2c_2)$.
\end{proof}

\begin{remark}
With some further work the bound in Lemma~\ref{lem:jumpTruncation} 
could be improved to $c_5=0$ in \eqref{eq:SmallJumpTruncation}. 
This would require us to prove an analogue of Lemma~\ref{lem:momentBound} for the process ${Y}^{x,d,h,\delta}$. 
This improvement is straigthforward, but not essential for the ensuing developments.
\end{remark}

\subsection{Approximation of coefficients}
\label{sec:ApprCoeff}
In the next approximation step we approximate the coefficients by deep neural networks. 
To this end, for $\varepsilon \in \rev{(0,1]}$, 
we consider the continuous-time process $Z^{x,d,h,\delta,\varepsilon}$. 
Under the integrability condition on $\gamma_{\varepsilon,d}
$ in Assumption \ref{ass:NNApproxCoeff},
this is the unique c\`adl\`ag process satisfying $Z_0^{x,d,h,\delta,\varepsilon} = x$, 
\begin{equation} \label{eq:SDEEulerFinal} 
\begin{aligned}
d Z_t^{x,d,h,\delta,\varepsilon} 
= 
\beta_{\varepsilon,d}(Z_{\lfloor t- \rfloor_h}^{x,d,h,\delta,\varepsilon}) d t 
+ 
\sigma_{\varepsilon,d}(Z_{\lfloor t- \rfloor_h}^{x,d,h,\delta,\varepsilon}) d B^d_t 
+ 
\int_{A_\delta} \gamma_{\varepsilon,d}(Z_{\lfloor t- \rfloor_h}^{x,d,h,\delta,\varepsilon},z) \tilde{N}^d(dt,dz), 
\quad t \in (0,T],
\end{aligned} 
\end{equation}
where we now also allow $\delta = 0$ with the convention that $A_0= \R^d \setminus \{0\}$. 
See also \eqref{eq:auxEq15} below for a more explicit representation of $Z=Z^{x,d,h,\delta,\varepsilon}$.

We first need a moment estimate similar to Lemma~\ref{lem:momentBound}. \rev{The constants $c_7,c_8 >0$ in Lemma~\ref{lem:momentBoundEulerFinal} only depend on $C$ and $T$.}
\begin{lemma}\label{lem:momentBoundEulerFinal}
		Suppose Assumption~\ref{ass:NNApproxCoeff} holds. Then there exist constants $c_7,c_8 >0$ such that for all $d \in \N, x \in \R^d, t \in [0,T]$, $h \in (0,1)$, $\delta \geq 0$, $\varepsilon \in \rev{(0,1]}$ \rev{it holds that}
		\begin{equation}\label{eq:momentBoundEulerFinal}
		\E[\|Z_t^{x,d,h,\delta,\varepsilon}\|^2] \leq c_7 \|x\|^2 + c_8 d^p \varepsilon^{-q}.
		\end{equation}
\end{lemma}
\begin{proof}
Fix $d \in \N$, $x \in \R^d$, $h \in (0,1)$, $\delta \geq 0$, $\varepsilon \in \rev{(0,1]}$. 
To simplify notation write $Z=Z^{x,d,h,\delta,\varepsilon}$.
Let $\rev{\bar{G}}(t) = \sup_{s \leq t } \E[\|Z_s\|^2]$ for $t \in [0,T]$. 
Note that for $t \in [t_n,t_{n+1}]$ we have
\begin{equation} 
\label{eq:auxEq15} 
\begin{aligned}
Z_t = Z_{t_n} + \beta_{\varepsilon,d}(Z_{t_n})(t-t_n)
    &  + \sigma_{\varepsilon,d}(Z_{t_n}) (B^d_{t}-B^d_{t_n})  
       + \int_{t_n}^t \int_{A_\delta} \gamma_{\varepsilon,d}(Z_{t_n},z) \tilde{N}^d(dt,dz)
\end{aligned} 
\end{equation}
and the stochastic integral is well-defined, see, e.g., \cite[Section~4.3.2]{Applebaum2009}.
Thus, if $Z_{t_n} \in L^2(\Omega,\Fc,\P)$, then the triangle inequality, It\^o's isometry 
and the growth hypotheses on 
$\beta_{\varepsilon,d},\sigma_{\varepsilon,d},\gamma_{\varepsilon,d}$ 
in Assumption \ref{ass:NNApproxCoeff}
prove that
\begin{equation} \label{eq:auxEq16} \begin{aligned}
\E[\|Z_t\|^2]^{1/2} & \leq \E[\|Z_{t_n}\|^2]^{1/2} + \E[\|\beta_{\varepsilon,d}(Z_{t_n})\|^2]^{1/2}(t-t_n)  + \E[\|\sigma_{\varepsilon,d}(Z_{t_n}) (B^d_{t}-B^d_{t_n})\|^2]^{1/2}  \\ & \quad \quad + \E\left[\left\|\int_{t_n}^t \int_{A_\delta} \gamma_{\varepsilon,d}(Z_{t_n},z) \tilde{N}^d(dt,dz)\right\|^2\right]^{1/2}
\\ & \leq \E[\|Z_{t_n}\|^2]^{1/2} + C^{\rev{1/2}} [(d^p \varepsilon^{-q})^{\rev{1/2}} + \E[\|Z_{t_n}\|^2]^{1/2}](t-t_n)  + \E[\|\sigma_{\varepsilon,d}(Z_{t_n})\|_F^2(t-t_n)]^{1/2} 
\\ & \quad \quad 
 + \left(\int_{A_\delta} \E\left[\left\| \gamma_{\varepsilon,d}(Z_{t_n},z)\right\|^2\right] \nu^d(dz) (t-t_n) \right)^{1/2}
\\ & \leq (1+3C^{\rev{1/2}})\E[\|Z_{t_n}\|^2]^{1/2} + 3 (C d^p \varepsilon^{-q})^{\rev{1/2}}.
\end{aligned} \end{equation}
Using $Z_{t_0} =x$ we inductively obtain from \eqref{eq:auxEq16} that $Z_{t_n} \in L^2(\Omega,\Fc,\P)$ for $n=0,1,\ldots,N$ and furthermore $\rev{\bar{G}} \in L^1([0,T])$. 
 
Next, we insert the SDE representation \eqref{eq:SDEEulerFinal} 
and apply the same arguments used to obtain \eqref{eq:auxEq16} 
to estimate for any $t \in [0,T]$
\begin{equation} \label{eq:auxEq17} \begin{aligned}
\E[\|Z_t\|^2]^{1/2} & \leq \|x\| + \E\left[\left\|\int_0^t \beta_{\varepsilon,d}(Z_{\lfloor s- \rfloor_h}) ds\right\| ^2\right]^{1/2} + \E\left[\left\|\int_0^t \sigma_{\varepsilon,d}(Z_{\lfloor s- \rfloor_h}) d B^d_s \right\| ^2\right]^{1/2} \\ & \quad + \E\left[\left\|\int_0^t \int_{A_\delta} \gamma_{\varepsilon,d}(Z_{\lfloor s- \rfloor_h},z) \tilde{N}^d(ds,dz) \right\| ^2\right]^{1/2}
\\ & \leq \|x\| + \int_0^t \E[ \|\beta_{\varepsilon,d}(Z_{\lfloor s- \rfloor_h})\|^2]^{1/2} ds + \left[\int_0^t \E[\|\sigma_{\varepsilon,d}(Z_{\lfloor s- \rfloor_h})\|_F^2]ds\right]^{1/2} \\ & \quad + \left(\int_0^t \int_{A_\delta} \E[\|\gamma_{\varepsilon,d}(Z_{\lfloor s- \rfloor_h},z)\|^2] \nu^d(dz) ds \right)^{1/2}
\\ & \leq \|x\| + T (C d^p \varepsilon^{-q})^{\rev{1/2}} + C^{\rev{1/2}} \int_0^t \E[ \|Z_{\lfloor s- \rfloor_h}\|^2]^{1/2} ds + 2 C^{\rev{1/2}} \left[\int_0^t \E[d^p \varepsilon^{-q} + \|Z_{\lfloor s- \rfloor_h}\|^2 ]ds\right]^{1/2}
\\ & \leq \|x\| + T (C d^p \varepsilon^{-q})^{\rev{1/2}} + C^{\rev{1/2}} T^{1/2} \left(\int_0^t \rev{\bar{G}}(s) ds\right)^{1/2} + 2 C^{\rev{1/2}} \left[T d^p \varepsilon^{-q} + \int_0^t \rev{\bar{G}}(s) ds\right]^{1/2}.
\end{aligned} \end{equation}
This shows that
for any $t \in [0,T]$ it holds that
\begin{equation} \label{eq:auxEq18} \begin{aligned}
\rev{\bar{G}}(t)  \leq 6\|x\|^2 + [6 T^2+24T] Cd^p \varepsilon^{-q} + [3C T+24 C ] \int_0^t \rev{\bar{G}}(s) ds
\end{aligned} \end{equation}
(with $C$ as in Assumption~\ref{ass:NNApproxCoeff}) 
and hence, by Gr\rev{\"o}nwall's inequality, 
we conclude 
$\rev{\bar{G}}(T)\leq a\exp(bT)$ with $a= 6\|x\|^2 + [6 T^2+24T] Cd^p \varepsilon^{-q}$, $b=3C T+24 C$. 
This proves \eqref{eq:momentBoundEulerFinal} with $c_7 = 6 \exp(bT) $ and $c_8 = [6 T^2+24T] C \exp(bT) $. 
\end{proof}
\begin{remark}
\rev{In general, the constants $c_1$, $c_2$ in Lemma~\ref{lem:momentBound} depend exponentially on the constant $L$ appearing in the Lipschitz and growth conditions in Assumption~\ref{ass:LipschitzGrowth}. The proof of Lemma~\ref{lem:momentBoundEulerFinal} employs analogous techniques to the proof of Lemma~\ref{lem:momentBound}.}
\rev{By}
\rev{ using only the growth hypotheses on
$\beta_{\varepsilon,d},\sigma_{\varepsilon,d},\gamma_{\varepsilon,d}$ 
(see Assumption \ref{ass:NNApproxCoeff}) the exponential dependence on the Lipschitz constant (but not on $C$) can be avoided.}
\end{remark}
The following result provides an estimate for the error arising 
from the neural network approximation of the coefficients. 
The result holds both for $\delta = 0$ (no truncation of jumps) 
and 
$\delta>0$ (jumps smaller than $\delta$ are removed). 
\rev{The constant $c_9 >0$ in Proposition~\ref{prop:NNApproxEuler} only depends on $C$, $L$ and $T$.}
\begin{proposition}\label{prop:NNApproxEuler}
Suppose Assumptions~\ref{ass:LipschitzGrowth} and \ref{ass:NNApproxCoeff} hold. 

Then there exists $c_9 > 0$ such that for all 
$d \in \N$, $x \in \R^d$, $h \in (0,1)$, $\delta \geq 0$, $\varepsilon \in \rev{(0,1]}$ 
holds
\begin{equation}\label{eq:NNApproxEuler}
\E\left[\sup_{t \in [0,T]}\|{Y}_t^{x,d,h,\delta} - {Z}_t^{x,d,h,\delta,\varepsilon} \|^2\right] 
\leq 
c_9 \varepsilon^{3q+1} d^{2p} (1+  \|x\|^2).
\end{equation}
\end{proposition}
\begin{proof} 
Fix $d \in \N$, $x \in \R^d$, $h \in (0,1)$, $\delta \geq 0$, $\varepsilon \in \rev{(0,1]}$. 
To simplify notation write $Z=Z^{x,d,h,\delta,\varepsilon}$, $Y={Y}^{x,d,h,\delta} $.
Let $\rev{\bar{G}}(t) = \E[\sup_{s \in [0,t]}\|{Y}_s^{x,d,h,\delta} - {Z}_s^{x,d,h,\delta,\varepsilon} \|^2]$ for $t \in [0,T]$. 
Then by the triangle inequality, Doob's martingale inequality 
and It\^o's isometry we obtain
\begin{equation}\label{eq:auxEq19} 
\begin{aligned}
\rev{\bar{G}}(t) & \leq 3 \E\left[\sup_{s \in [0,t]}\left\|\int_0^s \beta^d(Y_{\lfloor r- \rfloor_h}) - \beta_{\varepsilon,d}(Z_{\lfloor r- \rfloor_h}) d r\right\|^2 \right]  + 3\E\left[\sup_{s \in [0,t]}\left\|\int_0^s \sigma^d(Y_{\lfloor r- \rfloor_h})-\sigma_{\varepsilon,d}(Z_{\lfloor r- \rfloor_h}) d B^d_r\right\|^2\right]  \\ & \quad +  3\E\left[\sup_{s \in [0,t]}\left\|\int_0^s \int_{A_\delta} \gamma^d(Y_{\lfloor r- \rfloor_h},z)-\gamma_{\varepsilon,d}(Z_{\lfloor r- \rfloor_h},z) \tilde{N}^d(dr,dz) \right\|^2\right]
\\ & \leq 3 t\int_0^t \E\left[\left\|\beta^d(Y_{\lfloor r- \rfloor_h}) - \beta_{\varepsilon,d}(Z_{\lfloor r- \rfloor_h})\right\|^2\right] d r   + 12\E\left[\left\|\int_0^t \sigma^d(Y_{\lfloor r- \rfloor_h})-\sigma_{\varepsilon,d}(Z_{\lfloor r- \rfloor_h}) d B^d_r\right\|^2\right]  \\ & \quad +  12\E\left[\left\|\int_0^t \int_{A_\delta} \gamma^d(Y_{\lfloor r- \rfloor_h},z)-\gamma_{\varepsilon,d}(Z_{\lfloor r- \rfloor_h},z)\tilde{N}^d(dr,dz) \right\|^2\right]
\\ & = 3 t\int_0^t \E\left[\left\|\beta^d(Y_{\lfloor r- \rfloor_h}) - \beta_{\varepsilon,d}(Z_{\lfloor r- \rfloor_h})\right\|^2\right] d r   + 12\int_0^t \E\left[\left\|\sigma^d(Y_{\lfloor r- \rfloor_h})-\sigma_{\varepsilon,d}(Z_{\lfloor r- \rfloor_h})\right\|_F^2 \right] dr \\ & \quad +  12\int_0^t \E\left[\int_{\R^d}\left\| \gamma^d(Y_{\lfloor r- \rfloor_h},z)-\gamma_{\varepsilon,d}(Z_{\lfloor r- \rfloor_h},z) \right\|^2 \nu^d(dz)\right] dr .
\end{aligned}
\end{equation}
The triangle inequality, the Lipschitz-continuity of $\beta^d$ 
and Assumption~\ref{ass:NNApproxCoeff}(i) 
then yield for any $r \in [0,T]$ 
\begin{equation}
\label{eq:auxEq20} 
\begin{aligned}
\left\|\beta^d(Y_{\lfloor r- \rfloor_h}) - \beta_{\varepsilon,d}(Z_{\lfloor r- \rfloor_h})\right\|^2
& \leq 2 \left\|\beta^d(Y_{\lfloor r- \rfloor_h}) - \beta^d(Z_{\lfloor r- \rfloor_h})\right\|^2 + 2 \left\|\beta^d(Z_{\lfloor r- \rfloor_h}) - \beta_{\varepsilon,d}(Z_{\lfloor r- \rfloor_h})\right\|^2 
\\ & \leq 2 L \left\|Y_{\lfloor r- \rfloor_h} - Z_{\lfloor r- \rfloor_h}\right\|^2 + 2 \varepsilon^{4q+1} C d^p (1+\|Z_{\lfloor r- \rfloor_h}\|^2)
\end{aligned}
\end{equation}
and similarly 
\begin{equation}\label{eq:auxEq21} \begin{aligned}
\left\|\sigma^d(Y_{\lfloor r- \rfloor_h})-\sigma_{\varepsilon,d}(Z_{\lfloor r- \rfloor_h})\right\|_F^2 & + \int_{\R^d}\left\| \gamma^d(Y_{\lfloor r- \rfloor_h},z)-\gamma_{\varepsilon,d}(Z_{\lfloor r- \rfloor_h},z) \right\|^2 \nu^d(dz) \\ &  \quad \quad \leq 2 L \left\|Y_{\lfloor r- \rfloor_h} - Z_{\lfloor r- \rfloor_h}\right\|^2 + 2 \varepsilon^{4q+1} C d^p (1+\|Z_{\lfloor r- \rfloor_h}\|^2). 
\end{aligned}
\end{equation}
Inserting the two estimates \eqref{eq:auxEq20}, \eqref{eq:auxEq21} into \eqref{eq:auxEq19} yields 
\begin{equation}\label{eq:auxEq22} \begin{aligned}
\rev{\bar{G}}(t) & \leq 2(3 t+12)\int_0^t  L \rev{\bar{G}}(r) dr + 2(3 t+12) \int_0^t  \varepsilon^{4q+1} C d^p (1+\E\left[\|Z_{\lfloor r- \rfloor_h}\|^2\right]) d r .
\end{aligned}
\end{equation}
By using Gr\rev{\"o}nwall's inequality in the first step and \eqref{eq:momentBoundEulerFinal} in the second step we therefore conclude that
\[\begin{aligned}
\rev{\bar{G}}(t) & \leq 2(3 T+12) \int_0^T  \varepsilon^{4q+1} C d^p (1+\E\left[\|Z_{\lfloor r- \rfloor_h}\|^2\right]) d r \exp(2(3 T+12) L t)
\\ & \leq 2(3 T+12) T  \varepsilon^{4q+1} C d^p (1+c_7 \|x\|^2 + c_8 d^p \varepsilon^{-q}) \exp(2(3 T+12) L T) 
\end{aligned}\]
which proves \eqref{eq:NNApproxEuler} with $c_9 = \max(a c_7,a (1+c_8))$, $a=2(3 T+12)C  T \exp(2(3 T+12) L T) $.
\end{proof}
\subsection{Monte Carlo approximation of the compensator integral}
\label{sec:MCCompInt}
For $0<\delta<1$ write $Z:=Z^{x,d,h,\delta,\varepsilon}$. 
Then for $t \in [t_n,t_{n+1}]$ the process $Z$ in \eqref{eq:SDEEulerFinal} 
can be written as
\[\begin{aligned}
Z_{t} & = Z_{t_n} + \beta_{\varepsilon,d}(Z_{t_n})(t-t_n)  + \sigma_{\varepsilon,d}(Z_{t_n}) (B^d_{t}-B^d_{t_n})  + \int_{t_n}^t \int_{A_\delta} \gamma_{\varepsilon,d}(Z_{t_n},z) \tilde{N}^d(dt,dz)
\\ & = Z_{t_n} + \beta_{\varepsilon,d}(Z_{t_n})(t-t_n)  + \sigma_{\varepsilon,d}(Z_{t_n}) (B^d_{t}-B^d_{t_n})  + \sum_{t_n \leq s \leq t} \gamma_{\varepsilon,d}(Z_{t_n},\Delta P_s^d) \mathbbm{1}_{A_\delta}(\Delta P_s^d) \\ & \quad - (t-t_n) \int_{A_\delta} \gamma_{\varepsilon,d}(Z_{t_n},z) \nu^d(dz)
\end{aligned} \]
where $P_t^d = \int_{A_\delta} y N^d(t,d y)$, 
see for instance \cite[Section~4.3.2]{Applebaum2009}, and 
$\Delta P_t^d = P_t^d - P_{t-}^d$ is the jump size of $P^d$ at $t$.

The final approximation step that we carry out now 
allows us to approximate the last integral above by a finite sum over random samples.
In case of a L\'evy-driven SDE (that is, when Assumption~\ref{ass:jumps}(i) is satisfied) 
this procedure is not required in the proof of Theorem~\ref{thm:main}
and so, in the current subsection, 
we work exclusively under Assumption~\ref{ass:jumps}(ii).

To this end, notice that 
\begin{equation}
\label{eq:AdeltaMass}
\nu^d(A_\delta) 
= \int_{A_\delta} \frac{1 \wedge \|z\|^2}{1 \wedge \|z\|^2 } \nu^d(dz) 
\leq 
\delta^{-2} \int_{\R^d } (1 \wedge \|z\|^2) \nu^d(d z) 
\leq 
\delta^{-2} \tilde{L} d^{\bar{q}} 
\end{equation}
is finite.
This shows that
%We need to introduce \tilde{\nu}^d nevertheless, since in the next sentence the random variables $V_{i,t_n}$ are drawn from this distribution.
$\tilde{\nu}^d(B):=\frac{\nu^d(B \cap A_\delta )}{\nu^d(A_\delta)}$ for $B \in \mathcal{B}(\R^d)$ 
defines a probability measure on $(\R^d,\mathcal{B}(\R^d))$. 
Let $\mathcal{M}\in \N$ and let $V_{i,t_n}$, $i=1,\ldots,\mathcal{M}$, $n=1,\ldots,N$ 
be i.i.d samples with distribution $\tilde{\nu}^d$, independent of $B^d$ and $N^d$. 
We now define the continuous-time process 
$\hat{Z}:=\hat{Z}^{x,d,h,\delta,\varepsilon,\mathcal{M}}$, 
which is the unique c\`adl\`ag process satisfying $\hat{Z}_0^{x,d,h,\delta,\varepsilon,\mathcal{M}} = x$, 
\begin{equation} \label{eq:SDEEulerFinal2} 
\begin{aligned}
d \hat{Z}_t^{x,d,h,\delta,\varepsilon,\mathcal{M}} 
&= \beta_{\varepsilon,d}(\hat{Z}_{\lfloor t- \rfloor_h}^{x,d,h,\delta,\varepsilon,\mathcal{M}}) dt 
+ \sigma_{\varepsilon,d}(\hat{Z}_{\lfloor t- \rfloor_h}^{x,d,h,\delta,\varepsilon,\mathcal{M}}) d B^d_t 
+ \int_{A_\delta} \gamma_{\varepsilon,d}(\hat{Z}_{\lfloor t- \rfloor_h}^{x,d,h,\delta,\varepsilon,\mathcal{M}},z) N^d(dt,dz) 
\\ & \quad \quad - \frac{\nu^d(A_\delta)}{\M}
\sum_{i=1}^\M \gamma_{\varepsilon,d}(\hat{Z}_{\lfloor t- \rfloor_h}^{x,d,h,\delta,\varepsilon,\mathcal{M}},V_{i,\lfloor t- \rfloor_h}) dt, 
\quad \quad   \quad t \in (0,T].
\end{aligned} 
\end{equation}
We first need a dimension-explicit bound on the second moments of $\hat{Z}_t^{x,d,h,\delta,\varepsilon,\mathcal{M}}$.
\rev{The constants $\tilde{c}_7,\tilde{c}_8 >0$ in Lemma~\ref{lem:momentBoundEulerFinal2} only depend on $C$ and $T$.}
\begin{lemma}\label{lem:momentBoundEulerFinal2}
	Suppose Assumption~\ref{ass:NNApproxCoeff} and Assumption~\ref{ass:jumps}(ii) hold. Then there exist constants $\tilde{c}_7,\tilde{c}_8 >0$ such that for all $d \in \N, x \in \R^d, t \in [0,T]$, $h \in (0,1)$, $\delta \in (0,1)$, $\varepsilon \in \rev{(0,1]}$ and $\M \in \N$ with $\M\geq \delta^{-2} \tilde{L} d^{\bar{q}} $ it holds that
	\begin{equation}\label{eq:momentBoundEulerFinal2}
	\E[\|\hat{Z}_t^{x,d,h,\delta,\varepsilon,\mathcal{M}}\|^2] \leq \tilde{c}_7 \|x\|^2 + \tilde{c}_8 d^p \varepsilon^{-q}.
	\end{equation}
\end{lemma}

\begin{proof}
The proof proceeds similarly as the proof of Lemma~\ref{lem:momentBoundEulerFinal}. 
Fix $d \in \N$, $x \in \R^d$, $h \in (0,1)$, 
$\delta \in (0,1)$, $\varepsilon \in \rev{(0,1]}$, $\M \in \N$ 
and write 
$\hat{Z}=\hat{Z}^{x,d,h,\delta,\varepsilon,\M}$ 
and 
$\rev{\bar{G}}(t) = \sup_{s \leq t } \E[\|\hat{Z}_s\|^2]$ for $t \in [0,T]$. 

Then from \eqref{eq:SDEEulerFinal2} we obtain for $t \in [t_n,t_{n+1}]$
\begin{equation} \label{eq:auxEq31} \begin{aligned}
\hat{Z}_t = \hat{Z}_{t_n} + \beta_{\varepsilon,d}(\hat{Z}_{t_n})(t-t_n)&  + \sigma_{\varepsilon,d}(\hat{Z}_{t_n}) (B^d_{t}-B^d_{t_n})  + \int_{t_n}^t \int_{A_\delta} \gamma_{\varepsilon,d}(\hat{Z}_{t_n},z) \tilde{N}^d(dt,dz) \\ & + (t-t_n) \int_{A_\delta} \gamma_{\varepsilon,d}(\hat{Z}_{t_n},z) \nu^d(dz) - \frac{(t-t_n)\nu^d(A_\delta)}{\M}\sum_{i=1}^\M \gamma_{\varepsilon,d}(\hat{Z}_{t_n},V_{i,t_n}).
\end{aligned} 
\end{equation}
Suppose for now $\hat{Z}_{t_n} \in L^2(\Omega,\Fc,\P)$, the last difference can be estimated in $L^2$ as follows: 
by definition of $V_{i,t_n}$ we obtain for any $x \in \R^d$ that 
$\nu^d(A_\delta)\E[\gamma_{\varepsilon,d}(x,V_{i,t_n})] = \int_{A_\delta} \gamma_{\varepsilon,d}(x,z) \nu^d(dz)$. 
Hence, by independence, elementary properties of variance and 
with the growth hypothesis on $\gamma_{\varepsilon,d}$ (Assumption~\ref{ass:NNApproxCoeff})
we obtain
\begin{equation} \label{eq:auxEq36} \begin{aligned}
\E& \left[\left\|\int_{A_\delta} \gamma_{\varepsilon,d}(\hat{Z}_{t_n},z) \nu^d(dz) - \frac{\nu^d(A_\delta)}{\M}\sum_{i=1}^\M \gamma_{\varepsilon,d}(\hat{Z}_{t_n},V_{i,t_n})\right\|^2\right]^{1/2}
 \\ & = \E\left[\left.\E\left[\left\|\int_{A_\delta} \gamma_{\varepsilon,d}(x,z) \nu^d(dz) - \frac{\nu^d(A_\delta)}{\M}\sum_{i=1}^\M \gamma_{\varepsilon,d}(x,V_{i,t_n})\right\|^2\right]\right|_{x = \hat{Z}_{t_n}}\right]^{1/2}
 \\ & = \nu^d(A_\delta) \M^{-1/2} \E\left[\sum_{j=1}^d \left.\E[|\E[\gamma_{\varepsilon,d,j}(x,V_{1,t_1})]-\gamma_{\varepsilon,d,j}(x,V_{1,t_1})|^2]\right|_{x = \hat{Z}_{t_n}}\right]^{1/2}.
\\ & \leq   \nu^d(A_\delta) \M^{-1/2} \E\left[\left.\sum_{j=1}^d\E[|\gamma_{\varepsilon,d,j}(x,V_{1,t_1})|^2]\right|_{x = \hat{Z}_{t_n}}\right]^{1/2} 
\\ & =  [\nu^d(A_\delta)]^{1/2} \M^{-1/2} \E\left[ \int_{A_\delta} \|\gamma_{\varepsilon,d}(\hat{Z}_{t_n},z)\|^2 \nu^d(d z) \right]^{1/2}
\\ & \leq C^{\rev{1/2}} [\nu^d(A_\delta)]^{1/2} \M^{-1/2} ((d^p \varepsilon^{-q})^{\rev{1/2}} +\E[\|\hat{Z}_{t_n}\|^2 ]^{1/2}).
\end{aligned} 
\end{equation}
Thus, if $\hat{Z}_{t_n} \in L^2(\Omega,\Fc,\P)$, 
then using first the triangle inequality and precisely the same arguments used to obtain \eqref{eq:auxEq16} 
and then inserting \eqref{eq:auxEq36} and employing that $\nu^d(A_\delta) \leq \M$ 
(due to \eqref{eq:AdeltaMass} and the assumption $\M \geq \delta^{-2} \tilde{L} d^{\bar{q}}$) 
we deduce
\begin{equation} \label{eq:auxEq33} \begin{aligned}
\E[\|\hat{Z}_t\|^2]^{1/2} & \leq  (1+3C^{\rev{1/2}})\E[\|\hat{Z}_{t_n}\|^2]^{1/2} + 3 (C d^p \varepsilon^{-q})^{\rev{1/2}} \\ & \quad + (t-t_n)\E\left[\left\|\int_{A_\delta} \gamma_{\varepsilon,d}(\hat{Z}_{t_n},z) \nu^d(dz) - \frac{\nu^d(A_\delta)}{\M}\sum_{i=1}^\M \gamma_{\varepsilon,d}(\hat{Z}_{t_n},V_{i,t_n})\right\|^2\right]^{1/2} 
\\ & \leq  (1+(3+T)C^{\rev{1/2}})\E[\|\hat{Z}_{t_n}\|^2]^{1/2} + (3+T) (C d^p \varepsilon^{-q})^{\rev{1/2}}.
\end{aligned} \end{equation}
Starting with $\hat{Z}_{t_0} =x$ we may now inductively obtain from \eqref{eq:auxEq33} that $\hat{Z}_{t_n} \in L^2(\Omega,\Fc,\P)$ for $n=0,1,\ldots,N$ and furthermore $\rev{\bar{G}} \in L^1([0,T])$. 
Next, we insert \eqref{eq:SDEEulerFinal2} and apply the same arguments used to obtain   \eqref{eq:auxEq17} in the first inequality and the Minkowski integral inequality combined with \eqref{eq:auxEq36} to estimate for any $t \in [0,T]$
\begin{equation} \label{eq:auxEq34} \begin{aligned}
\E[ \|\hat{Z}_t\|^2]^{1/2} & \leq \|x\| + T (C d^p \varepsilon^{-q})^{\rev{1/2}} + (C T)^{1/2} \left(\int_0^t \rev{\bar{G}}(s) ds\right)^{1/2} + 2 C^{\rev{1/2}} \left[T d^p \varepsilon^{-q} +  \int_0^t \rev{\bar{G}}(s) ds\right]^{1/2} 
\\ & \quad + \E \left[\left\|\int_0^t\int_{A_\delta} \gamma_{\varepsilon,d}(\hat{Z}_{\lfloor s- \rfloor_h},z) \nu^d(dz) - \frac{\nu^d(A_\delta)}{\M}\sum_{i=1}^\M \gamma_{\varepsilon,d}(\hat{Z}_{\lfloor s- \rfloor_h},V_{i,\lfloor s- \rfloor_h}) d s\right\|^2\right]^{1/2}
\\  & \leq \|x\| + T (C d^p \varepsilon^{-q})^{\rev{1/2}} + (C T)^{1/2} \left(\int_0^t \rev{\bar{G}}(s) ds\right)^{1/2} + 2 C^{\rev{1/2}} \left[T d^p \varepsilon^{-q} + \int_0^t \rev{\bar{G}}(s) ds\right]^{1/2} 
\\ & \quad + \int_0^t C^{\rev{1/2}} ((d^p \varepsilon^{-q})^{\rev{1/2}} +\E[\|\hat{Z}_{\lfloor s- \rfloor_h}\|^2 ]^{1/2})  d s.
\\  & \leq \|x\| + 2 T (C d^p \varepsilon^{-q})^{\rev{1/2}} + 2 (C T)^{1/2} \left(\int_0^t \rev{\bar{G}}(s) ds\right)^{1/2} + 2 C^{\rev{1/2}} \left[Td^p \varepsilon^{-q} +  \int_0^t \rev{\bar{G}}(s) ds\right]^{1/2}.
\end{aligned} \end{equation}
This bound is, up to factors of $2$, identical with \eqref{eq:auxEq17}. 
The proof can now be completed using Gr\rev{\"o}nwall's inequality as before, \rev{yielding 
$\rev{\bar{G}}(T)\leq a\exp(bT)$ with $a= 6\|x\|^2 + [24 T^2 + 48 T] Cd^p \varepsilon^{-q}$, $b=24 C T+48C$. 
This proves \eqref{eq:momentBoundEulerFinal2} with $\tilde{c}_7 = 6 \exp(bT) $ and $\tilde{c}_8 = [24 T^2 + 48 T] C \exp(bT) $. 
}
%%%%%%%% For explicit expressions of \tilde{c}_7, \tilde{c}_8:
%Using that for all $a,b,c\geq0$ it holds $(a+b+c)^2\leq 3(a^2+b^2+c^2)$ we thus obtain
%\[\begin{aligned}
%G(t) = \sup_{s \leq t } \E[\|\hat{Z}_s\|^2] & \leq \left(\|x\| + [2 T C +2 C \left[2T\right]^{1/2}] d^p \varepsilon^{-q} + [2 C T^{1/2}+2^{1/2}2C] \left(\int_0^t G(s) ds\right)^{1/2} \right)^2
%\\ & \leq 3\|x\|^2 + 3[2 T C +2 C \left[2T\right]^{1/2}]^2 (d^p \varepsilon^{-q})^2 + 3[2 C T^{1/2}+2^{1/2}2C]^2 \int_0^t G(s) ds
%\\ & \leq 6\|x\|^2 + [24 T^2 + 48 T] (C d^p \varepsilon^{-q})^2 + [24 C^2 T+48C^2] \int_0^t G(s) ds
%\end{aligned}
%\]
 %%%%%%%%%%%%%%%%%%%%
\end{proof}
The next result provides an estimate for the error arising 
from the Monte Carlo approximation for the compensator integral. 
\rev{The constant  $\tilde{c}_9 >0$ in Proposition~\ref{prop:NNApproxEuler2} only depends on $C$, $L$, $\tilde{L}$ and $T$.} 
\begin{proposition}\label{prop:NNApproxEuler2}
Suppose Assumptions~\ref{ass:LipschitzGrowth}, ~\ref{ass:jumps}(ii) and \ref{ass:NNApproxCoeff} hold. 
Then there exists \rev{a constant} $\tilde{c}_9 >0$ such that 
for all $d \in \N$, $x \in \R^d$, $h \in (0,1)$, $\delta \in (0,1)$, $\varepsilon \in \rev{(0,1]}$ 
and for $\M \in \N$ with $\M\geq \delta^{-2} \tilde{L} d^{\bar{q}}$
it holds that
	\begin{equation}\label{eq:NNApproxEuler2}
	\E\left[\sup_{t \in [0,T]}\|{Y}_t^{x,d,h,\delta} - \hat{Z}_t^{x,d,h,\delta,\varepsilon,\M} \|^2\right] \leq \tilde{c}_9 [\varepsilon^{3q+1}  d^{2p} +  \delta^{-2} d^{3p+\bar{q}} \varepsilon^{-3q} \M^{-1} ] \left(1  + \|x\|^2 \right).
	\end{equation}
\end{proposition}

\begin{proof} 
The proof is analogous to the proof of Proposition~\ref{prop:NNApproxEuler}. 
Fix $d \in \N$, $x \in \R^d$, $h \in (0,1)$, $\delta \in (0,1)$, 
$\varepsilon \in \rev{(0,1]}$ 
and $\M \in \N$ with $\M\geq \delta^{-2} \tilde{L} d^{\bar{q}}$. 
As before we simplify notation by writing 
$\hat{Z}=\hat{Z}^{x,d,h,\delta,\varepsilon,\M}$, $Y={Y}^{x,d,h,\delta} $.
Define
$\rev{\bar{G}}(t) := \E[\sup_{s \in [0,t]}\|{Y}_s^{x,d,h,\delta} - \hat{Z}_s^{x,d,h,\delta,\varepsilon,\M} \|^2]$ 
for $t \in [0,T]$. 
Then by the triangle inequality we obtain 
\begin{equation}\label{eq:auxEq37} \begin{aligned}
\rev{\bar{G}}(t) & \leq 4 \E\left[\sup_{s \in [0,t]}\left\|\int_0^s \beta^d(Y_{\lfloor r- \rfloor_h}) - \beta_{\varepsilon,d}(Z_{\lfloor r- \rfloor_h}) d r\right\|^2 \right]  + 4\E\left[\sup_{s \in [0,t]}\left\|\int_0^s \sigma^d(Y_{\lfloor r- \rfloor_h})-\sigma_{\varepsilon,d}(Z_{\lfloor r- \rfloor_h}) d B^d_r\right\|^2\right]  \\ & \quad +  4\E\left[\sup_{s \in [0,t]}\left\|\int_0^s \int_{A_\delta} \gamma^d(Y_{\lfloor r- \rfloor_h},z)-\gamma_{\varepsilon,d}(Z_{\lfloor r- \rfloor_h},z) \tilde{N}^d(dr,dz) \right\|^2\right]
\\ & \quad + 4 \E \left[\sup_{s \in [0,t]}\left\|\int_0^s\int_{A_\delta} \gamma_{\varepsilon,d}(\hat{Z}_{\lfloor r- \rfloor_h},z) \nu^d(dz) - \frac{\nu^d(A_\delta)}{\M}\sum_{i=1}^\M \gamma_{\varepsilon,d}(\hat{Z}_{\lfloor r- \rfloor_h},V_{i,\lfloor r- \rfloor_h}) d r\right\|^2\right].
\end{aligned}
\end{equation}
Denote the sum of the first three terms by $G_1(t)$ and the last term by $G_2(t)$. 
Then $G_1(t)$ can be handled by the precise same argument used in 
\eqref{eq:auxEq19}-\eqref{eq:auxEq21}.
For these terms we obtain the analogous upper bound to \eqref{eq:auxEq22} (up to a factor $4/3$): 
\begin{equation}\label{eq:auxEq38} 
\begin{aligned}
G_1(t) & \leq 2(4 t+16)\int_0^t  L \rev{\bar{G}}(r) dr + 2(4 t+16) \int_0^t  \varepsilon^{4q+1} C d^p (1+\E\left[\|\hat{Z}_{\lfloor r- \rfloor_h}\|^2\right]) d r .
\end{aligned}
\end{equation}
On the other hand, 
using Minkowski's integral inequality and \eqref{eq:auxEq36} we obtain
\begin{equation}\label{eq:auxEq39} \begin{aligned}
G_2(t) & \leq 4 \E \left[\sup_{s \in [0,t]}\left(\int_0^s\left\|\int_{A_\delta} \gamma_{\varepsilon,d}(\hat{Z}_{\lfloor r- \rfloor_h},z) \nu^d(dz) - \frac{\nu^d(A_\delta)}{\M}\sum_{i=1}^\M \gamma_{\varepsilon,d}(\hat{Z}_{\lfloor r- \rfloor_h},V_{i,\lfloor r- \rfloor_h}) \right\| d r\right)^2\right]
\\ & \leq  4 \E \left[\left(\int_0^t\left\|\int_{A_\delta} \gamma_{\varepsilon,d}(\hat{Z}_{\lfloor r- \rfloor_h},z) \nu^d(dz) - \frac{\nu^d(A_\delta)}{\M}\sum_{i=1}^\M \gamma_{\varepsilon,d}(\hat{Z}_{\lfloor r- \rfloor_h},V_{i,\lfloor r- \rfloor_h}) \right\| d r\right)^2\right]
\\ & \leq  4 \left(\int_0^t \E \left[\left\|\int_{A_\delta} \gamma_{\varepsilon,d}(\hat{Z}_{\lfloor r- \rfloor_h},z) \nu^d(dz) - \frac{\nu^d(A_\delta)}{\M}\sum_{i=1}^\M \gamma_{\varepsilon,d}(\hat{Z}_{\lfloor r- \rfloor_h},V_{i,\lfloor r- \rfloor_h}) \right\|^2\right]^{1/2} dr\right)^2
\\ & \leq  4 C [\nu^d(A_\delta)] \M^{-1} \left(\int_0^t  (d^p \varepsilon^{-q})^{\rev{1/2}} +\E[\|\hat{Z}_{\lfloor r- \rfloor_h}\|^2 ]^{1/2} dr\right)^2
\\ & \leq  8 C \delta^{-2} \tilde{L} d^{\bar{q}}  \M^{-1} \left(T^2 d^{p} \varepsilon^{-q} + T \int_0^t \E[\|\hat{Z}_{\lfloor r- \rfloor_h}\|^2 ] dr\right).
\end{aligned}
\end{equation}
Combining \eqref{eq:auxEq37}-\eqref{eq:auxEq39} 
and Gr\rev{\"o}nwall's inequality in the first step 
and applying \eqref{eq:momentBoundEulerFinal2}  in the second step 
we hence conclude 
(with $\tilde{c}=\max(2(4 T+16)C\max(T,1),8 C\tilde{L}\max(T^2,T))$, $a = \max(1,T \tilde{c}_7,1+T \tilde{c}_8)$)
\[\begin{aligned}
	\rev{\bar{G}}(t) & \leq \tilde{c}[\varepsilon^{4q+1}  d^p +  \delta^{-2} d^{p+\bar{q}} \varepsilon^{-q} \M^{-1} ] \left(1  + \int_0^T \E[\|\hat{Z}_{\lfloor r- \rfloor_h}\|^2 ] dr\right) \exp(2(4 T+16) L t)
\\ & \leq \tilde{c}[\varepsilon^{4q+1}  d^p +  \delta^{-2} d^{p+\bar{q}} \varepsilon^{-q} \M^{-1} ] \left(1  + T \tilde{c}_7 \|x\|^2 + T \tilde{c}_8 d^p \varepsilon^{-q} \right) \exp(2(4 T+16) L t)
\\ & \leq a \tilde{c}[\varepsilon^{3q+1}  d^{2p} +  \delta^{-2} d^{3p+\bar{q}} \varepsilon^{-3q} \M^{-1} ] \left(1  + \|x\|^2 \right) \exp(2(4 T+16) L t),
\end{aligned}
\]
which proves \eqref{eq:NNApproxEuler2} with $\tilde{c}_9 = a \tilde{c} \exp(2(4 T+16) L T) $.
\end{proof}

\section{DNN Approximations for jump-diffusion processes}
\label{sec:main}

This section contains our main results.
We start by specifying in Section~\ref{sec:Payoff} 
the assumptions on the path-dependent functional. 
Section~\ref{sec:MainRes} contains the main result of the article and its proof. 
In Section~\ref{sec:woPath} we then specialize this result to functionals 
which do not exhibit path-dependence and in Section~\ref{sec:solPIDE} we apply these results 
to provide expression rate estimates for PIDEs. 
Finally, Section~\ref{sec:Basket} provides an application to basket option pricing.
\subsection{Admissible Payoff}
\label{sec:Payoff}
\rev{Fix $k \in \N$.}
For each $d \in \N$ we consider a function $\varphi_{d} \colon \R^d \times \R^{kd} \to \R$. 
We aim at approximating the map
\begin{equation}\label{eq:priceMap}
(x,K)\mapsto \E[\varphi_{d}(X_T^{x,d},K)]
\end{equation}
by deep neural networks. 
In the context of mathematical finance $\varphi_{d}$ is a \textit{parametric European payoff} 
and the right-hand-side of \eqref{eq:priceMap} is the price at time $0$ of a derivative written 
on an asset with price $X_T^{x,d}$ and payoff $\varphi_{d}(\cdot,K)$ at maturity $T$ 
(at least of $\P$ is a risk-neutral measure). \rev{The parameter $K$ captures the characteristics of the payoff such as, e.g., the strike price of an option.}

More generally, 
we will be interested in approximating expectations of certain 
path-dependent functionals (or derivatives in a mathematical finance context) 
$\Phi_{d} \colon \mathfrak{D}([0,T],\R^d) \times \R^{kd} \to \R$, 
i.e. the map
\[
(x,K) \mapsto U_d(x,K):=\E[\Phi_{d}((X_s^{x,d})_{s \in [0,T]},K)]
\]
is to be approximated by deep neural networks. 
Here $\mathfrak{D}([0,T],\R^d)$ denotes the space of all c\`adl\`ag functions 
$y \colon [0,T] \to \R^d$ 
(also referred to as ``Skorokhod space'', see e.g.\ \cite{Jac2003}).

For our results on expression rates,
we make the following assumptions on 
the functional $\Phi_d$. 
The case of ``European payoffs'' $\varphi_d$ is a special case 
and the assumption simplifies in this case; see Remark~\ref{rmk:European} below. \rev{Recall that $T>0$ denotes a fixed time horizon} \rev{and denote by $q$ the constant from Assumption~\ref{ass:NNApproxCoeff}.}

\begin{assumption}\label{ass:NNApproxPayoff}
Assume there exist $C>0$, $p,\rev{\hat{q}}\geq 0$ and that 
for each $d \in \mathbb{N}$, $\varepsilon \in (0,1]$ there exist 
$\rev{D_{d,\varepsilon}} \in \N$, 
$0\leq t_1^{d,\varepsilon}<\ldots<t_{D_{d,\varepsilon}}^{d,\varepsilon}\leq T$, 
neural networks   
$\Phi_{\varepsilon,d} \colon \R^{d \rev{D_{d,\varepsilon}}} \times \R^{kd} \to \R$ 
and probability measures $\mu^d$ on $\R^d\times \R^{kd}$ 
so that for each $d \in \mathbb{N}$, $\varepsilon \in (0,1]$,
\begin{itemize}
\item[(i)] 
for all $K \in \R^{kd}$, $y \in \mathfrak{D}([0,T],\R^d)$
\[ \begin{aligned}
|\Phi_{d}(y,K)-\Phi_{\varepsilon,d}(y({t_1^{d,\varepsilon}}),\ldots,y({t_{D_{d,\varepsilon}}^{d,\varepsilon}}),K)|  
& \leq \varepsilon C d^p (1+\|K\|+\sup_{t \in [0,T]}\|y(t)\|)
		\\  \mathrm{size}(\Phi_{\varepsilon,d})  & \leq C d^p \varepsilon^{-\rev{\hat{q}}},
	\\ \rev{D_{d,\varepsilon}} + \mathrm{Lip}(\Phi_{\varepsilon,d}) & \leq C d^p \varepsilon^{-q},
	\end{aligned}
\]
    \item[(ii)] $\int_{\R^d\times \R^{kd}} (1+\|x\|^2+\|K\|^2) \mu^d(dx,dK) \leq C d^p$.
\end{itemize}
\end{assumption}
Here, for a function $g \colon \R^q \times \R^r \to \R$, 
the quantity $\mathrm{Lip}(g)$ is defined as 
\[
\mathrm{Lip}(g) 
= 
\sup_{\substack{(x_1,K_1),(x_2,K_2) \in \R^q\times \R^{r}\\ x_1 \neq x_2, K_1 \neq K_2}} 
\frac{|g(x_1,K_1)-g(x_2,K_2)|}{\|x_1-x_2\|+\|K_1-K_2\|}
\;.
\]
Assumption~\ref{ass:NNApproxPayoff} includes many important derivatives such as 
discrete and continuously monitored Asian options or discretely monitored barrier options.

\begin{remark} \label{rmk:European}
In the special case when $\Phi_d$ is in fact a \rev{so-called European-type payoff in financial models of baskets}, i.e., when 
$\Phi_{d}((X_s^{x,d})_{s \in [0,T]},K) = \varphi_{d}(X_T^{x,d},K)$ for some $\varphi_{d} \colon \R^d \times \R^{kd} \to \R$ then we may set $\rev{D_{d,\varepsilon}} = 1$ in Assumption~\ref{ass:NNApproxPayoff} and Assumption~\ref{ass:NNApproxPayoff}(i) reduces to the requirement that there exist neural networks  $\phi_{\varepsilon,d} \colon \R^{d} \times \R^{kd} \to \R$ such that for each $d \in \N$, $\varepsilon \in (0,1]$ it holds for all $(x,K) \in \R^d \times \R^{kd}$ that
\begin{equation}\label{eq:EuropeanPayoffCond} \begin{aligned}
|\varphi_{d}(x,K)-\phi_{\varepsilon,d}(x,K)|  & \leq \varepsilon C d^p (1+\|x\|+\|K\|)
\\  \mathrm{size}(\phi_{\varepsilon,d})  & \leq C d^p \varepsilon^{-\rev{\hat{q}}},
\\ \mathrm{Lip}(\phi_{\varepsilon,d}) & \leq C d^p \varepsilon^{-q}.
\end{aligned}
\end{equation}
\rev{
In typical applications in mathematical finance $\varphi_{d}(\cdot,K)$ represents  the payoff of a financial derivative with characteristics $K$ and written on $d$ underlyings. In many relevant examples the initial condition $\varphi_{d}$ can be represented {\it exactly} by a ReLU DNN. In this case one can choose $\phi_{\varepsilon,d} = \varphi_{d}$ for all $d \in \N$, $\varepsilon \in (0,1]$ and \eqref{eq:EuropeanPayoffCond} holds with $\rev{\hat{q}}=q=0$.  Examples include, e.g., basket call options, basket put options, call on max options, call on min options, and many more, we refer, e.g., to \cite[Lemma~4.6, Lemma~4.8, Lemma~4.12, Lemma~4.14]{HornungJentzen2018}.}
\end{remark}
\begin{remark}\label{rmk:mud}
\rev{The probability measure $\mu^d$ in Assumption~\ref{ass:NNApproxPayoff}	is the measure with respect to which the approximation error is measured. The only requirement imposed on $\mu^d$ in Assumption~\ref{ass:NNApproxPayoff} is that the second moments of $\mu^d$ grow at most polynomially in $d$ and therefore Theorem~\ref{thm:main} below holds for a wide range of measures. For example, $\mu^d$ could be chosen as the Lebesgue measure on $[0,1]^d \times [0,1]^{kd}$ or as $\lambda_{[0,1]^d} \otimes \delta_{K_d}$ for some fixed $K_d \in \R^{kd}$ and with $\lambda_{[0,1]^d}$ the Lebesgue measure on $[0,1]^d$ (analogously to \cite[Theorem~1.1]{HornungJentzen2018}). More generally, $\mu^d$ could be chosen as the normalized Lebesgue measure on $\mathcal{X}^d \times \mathcal{K}^d$ for any Borel measurable set $\mathcal{X}^d \times \mathcal{K}^d \subset \R^d \times \R^{kd}$ satisfying $\sup_{d \in \N} \sup_{(x,K) \in \mathcal{X}^d \times \mathcal{K}^d} (\|x\|^2 + \|K\|^2) < \infty$. }
	\end{remark}

\subsection{Main result}
\label{sec:MainRes}
We now turn to our main approximation result. 
Theorem~\ref{thm:main} shows the following expressivity result for the approximation of 
$U^d$ by deep neural networks: 
an approximation accuracy of $\varepsilon>0$ can be achieved 
by a neural network with size bounded at most 
polynomially in $d$ and in $\varepsilon^{-1}$. 
Hence, the neural network approximation does not suffer from the CoD.
Note that the probability measure $\mu^d$ may have atoms. 
Theorem~\ref{thm:main} can thus also be used to obtain DNN expression rates 
for $U_d(\cdot,K_d)$ for single values $K_d \in \R^{kd}$ \rev{(cf., e.g., Remark~\ref{rmk:mud} above)}.	
\begin{theorem}\label{thm:main}
Assume that 
\begin{itemize} 
\item 
the coefficients of the SDE \eqref{eq:SDEnew} satisfy the Lipschitz and growth conditions 
in Assumption~\ref{ass:LipschitzGrowth},
\item 
the jumps of the process satisfy Assumption~\ref{ass:jumps}, that is, 
either we are in the case of a L\'evy-driven SDE or the small jumps 
exhibit decay \eqref{eq:smallJumpDecay}, \eqref{eq:LevymeasureGrowth},
\item 
the coefficient and payoff functions satisfy the approximation hypothesis,
Assumptions~\ref{ass:NNApproxCoeff}, \ref{ass:NNApproxPayoff}.
\end{itemize} 
 
Then there exist constants $\kappa, \mathfrak{p}, \mathfrak{q} > 0$ and,
for any $d \in \N$ and target accuracy $\varepsilon \in (0,1]$ exist  
neural networks  
$U_{\varepsilon,d} \colon \R^d \times \R^{kd} \to \R$ 
such that %for any $d \in \N$ and target accuracy $\varepsilon \in (0,1]$ 
\begin{align}
\label{eq:weightsPolynomial} 
\mathrm{size}(U_{\varepsilon,d}) &  \leq \kappa d^\mathfrak{p} \varepsilon^{-\mathfrak{q}}  \;,
\\
\left(\int_{\R^d\times \R^{kd}}  |U_d(x,K) - U_{\varepsilon,d}(x,K)|^2 \mu^d(dx,dK)\right)^{1/2}&  < \varepsilon.
\end{align}
\end{theorem}
 
\begin{proof} Let $\varepsilon \in (0,1]$ be given and consider $\mathfrak{N} \in \N$, $\bar{\varepsilon} \in (0,1)$, $h \in (0,1)$, $\delta \in [0,1)$ and $\M \in \N$ to be selected later. 
Essentially the proof consists in two steps: 
in a first step we carry out various approximation procedures to construct i.i.d.\ stochastic processes 
$Z^{x,d,h,\delta,\bar{\varepsilon},\M,1},\ldots,Z^{x,d,h,\delta,\bar{\varepsilon},\M,\mathfrak{N}}$ and find $\omega \in \Omega$ such that 
\begin{equation}\label{eq:epsbound} 
\begin{aligned}
\left(\int_{\R^d\times \R^{kd}} 
\left|
U_d(x,K) - \frac{1}{\mathfrak{N}} \sum_{i=1}^\mathfrak{N} 
\Phi_{\bar{\varepsilon},d}(\mathrm{e}_{d,\bar{\varepsilon}}(Z^{x,d,h,\delta,\bar{\varepsilon},\M,i}(\omega)),K) 
\right|^2 
\mu^d(dx,dK)\right)^{1/2} 
< \varepsilon
\end{aligned}
\end{equation}
where 
$\mathrm{e}_{d,\bar{\varepsilon}}(y)
 :=
 (y({t_1^{d,\bar{\varepsilon}}}),\ldots,y({t_{D_{d,\bar{\varepsilon}}}^{d,\bar{\varepsilon}}}))$ 
for $y \in \mathfrak{D}([0,T],\R^d)$. 

In a second step we prove that 
\[
U_{\varepsilon,d}(x,K)
:= 
\frac{1}{\mathfrak{N}} 
\sum_{i=1}^\mathfrak{N} \Phi_{\bar{\varepsilon},d}(\mathrm{e}_{d,\bar{\varepsilon}}(Z^{x,d,h,\delta,\bar{\varepsilon},\M,i}(\omega)),K)
\]
is indeed a neural network with weights satisfying \eqref{eq:weightsPolynomial}.

The proof is slightly different depending on whether 
Assumption~\ref{ass:jumps}(i) or Assumption~\ref{ass:jumps}(ii) is satisfied. 
Hence, we prove the two steps separately for each case. 
In particular, 
in each of these two cases a different choice is made for 
$Z^{x,d,h,\delta,\bar{\varepsilon},\M,1},\ldots,Z^{x,d,h,\delta,\bar{\varepsilon},\M,\mathfrak{N}}$. 
In the first case this will be $\mathfrak{N}$ i.i.d.\ copies of the process introduced in \eqref{eq:SDEEulerFinal}, 
whereas in the second case instead $\mathfrak{N}$ i.i.d.\ copies of \eqref{eq:SDEEulerFinal2} will be chosen.

\textbf{The case of Assumption~\ref{ass:jumps}(i)}

Consider first the case of L\'evy-SDE, i.e.\ when Assumption~\ref{ass:jumps}(i) is satisfied. 
Let $\delta=0$, $\M=1$ and let 
$Z^{x,d,h,\delta,\bar{\varepsilon},\M,1},\ldots,Z^{x,d,h,\delta,\bar{\varepsilon},\M,\mathfrak{N}}$ 
be $\mathfrak{N}$ i.i.d.\ copies of the process 
$Z^{x,d,h,\delta,\bar{\varepsilon}}$ introduced in \eqref{eq:SDEEulerFinal}. 

\textit{Step 1:} 
Denote 
$\mathcal{Z}_i(x,K) 
 = 
 \Phi_{\bar{\varepsilon},d}(\mathrm{e}_{d,\bar{\varepsilon}}(Z^{x,d,h,\delta,\bar{\varepsilon},\M,i}),K)$, 
$\bar{\mathcal{Z}}_i(x,K) = \mathcal{Z}_i(x,K)-\Phi_{\bar{\varepsilon},d}(\mathrm{e}_{d,\bar{\varepsilon}}(0),K)$ \rev{and write $Z = Z^{x,d,h,\delta,\bar{\varepsilon}}$.}
We start by estimating the following $L^2$-approximation error.

\begin{equation}\label{eq:auxEq26} \begin{aligned}
\int_{\R^d\times \R^{kd}} 
& \E\left[\left|U_d(x,K) - \frac{1}{\mathfrak{N}} \sum_{i=1}^\mathfrak{N} \mathcal{Z}_i(x,K) \right|^2 \right] \mu^d(dx,dK) 
\\ 
& = \int_{\R^d\times \R^{kd}} \left|U_d(x,K) - \E[\mathcal{Z}_1(x,K)]\right|^2 
  + 
   \E\left[ \left|\E[\mathcal{Z}_1(x,K)]- \frac{1}{\mathfrak{N}} \sum_{i=1}^\mathfrak{N} \mathcal{Z}_i(x,K) \right|^2 \right] \mu^d(dx,dK)
\\ 
& =  \int_{\R^d\times \R^{kd}} \left|U_d(x,K) - \E[\mathcal{Z}_1(x,K)]\right|^2 
  + \frac{1}{\mathfrak{N}} \E\left[ \left|\E[\bar{\mathcal{Z}}_1(x,K)]- \bar{\mathcal{Z}}_1(x,K) \right|^2 \right] \mu^d(dx,dK).
\end{aligned}
\end{equation}
We now estimate these two terms separately. 
For the first term, 
the triangle inequality and Assumption~\ref{ass:NNApproxPayoff}(i) 
yield for any $(x,K) \in \R^d\times \R^{kd}$
\begin{equation}\label{eq:auxEq23} \begin{aligned}
& \left|U_d(x,K) - \E[\mathcal{Z}_1(x,K)]\right| \\ & \quad \leq \E[|\Phi_{d}((X_s^{x,d})_{s \in [0,T]},K)-\Phi_{\bar{\varepsilon},d}(\mathrm{e}_{d,\bar{\varepsilon}}(\rev{Z}),K)|]
\\ & \quad\leq \E[|\Phi_{d}((X_s^{x,d})_{s \in [0,T]},K)-\Phi_{\bar{\varepsilon},d}(\mathrm{e}_{d,\bar{\varepsilon}}(X^{x,d}),K)|]  +  \E[|\Phi_{\bar{\varepsilon},d}(\mathrm{e}_{d,\bar{\varepsilon}}(X^{x,d}),K)-\Phi_{\bar{\varepsilon},d}(\mathrm{e}_{d,\bar{\varepsilon}}(\rev{Z}),K)|]
\\ & \quad \leq \E\left[ \bar{\varepsilon} C d^p (1+\|x\|+\|K\|+ \sup_{t \in [0,T]}\|X_t^{x,d}\|)\right] + C d^p \bar{\varepsilon}^{-q} \E[  \|\mathrm{e}_{d,\bar{\varepsilon}}(X^{x,d})-\mathrm{e}_{d,\bar{\varepsilon}}(\rev{Z})\|]
\\ & \quad \leq \bar{\varepsilon} C d^p (1+\|x\|+\|K\|+\E[ \sup_{t \in [0,T]}\|X_t^{x,d}\|]) + C d^p \bar{\varepsilon}^{-q} \rev{D_{d,\bar{\varepsilon}}}^{1/2} \E[  \sup_{t \in [0,T]} \|X_t^{x,d}-\rev{Z}_t\|].
\end{aligned} \end{equation}
Using Lemma~\ref{lem:Euler}, 
Lemma~ \ref{lem:jumpTruncation} and 
Proposition~\ref{prop:NNApproxEuler} 
we obtain
\[
\begin{aligned} \E[ &  \sup_{t \in [0,T]} \|X_t^{x,d}-\rev{Z_t}\|] \\ &  \leq \E[  \sup_{t \in [0,T]} \|X_t^{x,d}-\bar{X}_t^{x,d,h}\|] + \E[  \sup_{t \in [0,T]} \|\bar{X}_t^{x,d,h}-Y_t^{x,d,h,\delta}\|] + \E[  \sup_{t \in [0,T]} \|Y_t^{x,d,h,\delta}-\rev{Z_t}\|]
\\ & \leq  [h(c_3 d^4 + c_4 d^2\|x\|^2)]^{1/2} + [c_5 h(d^4 + d^2\|x\|^2) + c_6 \delta^{\bar{p}} d^{\bar{q}} (\|x\|^2+d^2)]^{1/2} + [c_9 \bar{\varepsilon}^{3q+1} d^{2p} (1+  \|x\|^2)]^{1/2}.
\end{aligned}
\]

Inserting this estimate and estimate \eqref{eq:supL2} from Corollary~\ref{cor:supL2} into \eqref{eq:auxEq23} we obtain
\begin{equation}\label{eq:auxEq24} \begin{aligned}
& \left|U_d(x,K) - \E[\mathcal{Z}_1(x,K)]\right|^2 \leq 6 \bar{\varepsilon}^2 C^2 d^{2p} (\|x\|^2+\|K\|^2+\tilde{c}_3 d^2 + \tilde{c}_4 \|x\|^2) \\ & \quad + 6 C^3 d^{3p} \bar{\varepsilon}^{-3q} [h(c_3 d^4 + c_4 d^2\|x\|^2) + c_5 h(d^4 + d^2\|x\|^2) + c_6 \delta^{\bar{p}} d^{\bar{q}} (\|x\|^2+d^2) + c_9 \bar{\varepsilon}^{3q+1} d^{2p} (1+  \|x\|^2)]
\end{aligned} \end{equation}
with $\tilde{c}_3 = 2\max(\bar{c}_3,1), \tilde{c}_4 = 2\bar{c}_4$.
To estimate the second term in \eqref{eq:auxEq26}, 
we apply the Lipschitz condition and the moment bound \eqref{eq:momentBoundEulerFinal} to obtain
\begin{equation}\label{eq:auxEq25} \begin{aligned}
 \E\left[ \left|\E[\bar{\mathcal{Z}}_1(x,K)]- \bar{\mathcal{Z}}_1(x,K) \right|^2 \right] & \leq\E\left[ \left|\Phi_{\bar{\varepsilon},d}(\mathrm{e}_{d,\bar{\varepsilon}}(\rev{Z}),K) -\Phi_{\bar{\varepsilon},d}(\mathrm{e}_{d,\bar{\varepsilon}}(0),K) \right|^2 \right]
\\ & \leq C^2 d^{2p} \bar{\varepsilon}^{-2q}\E\left[ \left\|\mathrm{e}_{d,\bar{\varepsilon}}(\rev{Z}) -\mathrm{e}_{d,\bar{\varepsilon}}(0)\right\|^2 \right]
\\ & =  C^2 d^{2p} \bar{\varepsilon}^{-2q} \sum_{i=1}^{\rev{D_{d,\bar{\varepsilon}}}} \E\left[ \left\|\rev{Z_{t_{i}^{d,\varepsilon}}} \right\|^2 \right]
\\ & \leq C^3 d^{3p} \bar{\varepsilon}^{-3q}  (c_8 d^p \bar{\varepsilon}^{-q} + c_7 \|x\|^2).
\end{aligned} \end{equation}

Inserting \eqref{eq:auxEq24} and \eqref{eq:auxEq25} into \eqref{eq:auxEq26} and using the growth condition on the integral in Assumption~\ref{ass:NNApproxPayoff}(ii) we obtain
\begin{equation}\label{eq:auxEq27} \begin{aligned}
\int_{\R^d\times \R^{kd}} & \E\left[\left|U_d(x,K) - \frac{1}{\mathfrak{N}} \sum_{i=1}^\mathfrak{N} \mathcal{Z}_i(x,K) \right|^2 \right] \mu^d(dx,dK) 
\\ & \leq  \bar{c} d^r [\bar{\varepsilon}^2  + \bar{\varepsilon}^{-3q}  (2h+%\delta^{\bar{p}}+
\bar{\varepsilon}^{3q+1}) + \mathfrak{N}^{-1} \bar{\varepsilon}^{-4q}]
\end{aligned}\end{equation}
with $\bar{c} = 6 \max(C^3\max(1,\tilde{c}_3,1+\tilde{c}_4), C^4\max(c_3,c_4,c_5,c_6,c_9),C^4\max(c_8,c_7))$, $r = \max(5p,6p+4+\bar{q})$. Now choose $\bar{\varepsilon} = \varepsilon (\max(6 \bar{c},1) d^r)^{-1}$ and $h = \varepsilon^2 (9 \bar{c} d^r \bar{\varepsilon}^{-3q})^{-1}$,  
$\mathfrak{N} = \lceil 3 \varepsilon^{-2} \bar{c} d^r \bar{\varepsilon}^{-4q} \rceil $. With these choices, the bound in \eqref{eq:auxEq27} becomes
\begin{equation}\label{eq:auxEq28} \begin{aligned}
\E&\left[ \int_{\R^d\times \R^{kd}}  \left|U_d(x,K) - \frac{1}{\mathfrak{N}} \sum_{i=1}^\mathfrak{N} \Phi_{\bar{\varepsilon},d}(\mathrm{e}_{d,\bar{\varepsilon}}(Z^{x,d,h,\delta,\bar{\varepsilon},\M,i}),K) \right|^2  \mu^d(dx,dK) \right] < \frac{\varepsilon^2}{3} + \frac{\varepsilon^2}{3} +   \frac{\varepsilon^2}{3} = \varepsilon^2. 
\end{aligned}\end{equation}
Hence, there exists $\omega \in \Omega$ such that \eqref{eq:epsbound} holds. 

\textit{Step 2:}
  
Let $i \in \{1, \ldots,\mathfrak{N}\}$ and write $Z^{x,i}:=Z^{x,d,h,\delta,\bar{\varepsilon},\M,i}$. 
Denote by $L_t^{d,i} =  \int_{0}^t \int_{\R^d} G^d(z) \tilde{N}^{d,i}(dt,dz) $ the jump part of the ($i$-th i.i.d. copy of the) L\'evy process driving the SDE \eqref{eq:SDEnew}.  Let $\ell^\beta = \mathrm{depth}(\beta_{\bar{\varepsilon},d}) $  $\ell^\sigma_j = \mathrm{depth}(\sigma_{\bar{\varepsilon},d,j})$, $\ell^F_j = \mathrm{depth}(F_{\bar{\varepsilon},d,j})$ and set $\ell_{max} = \max(2,\ell^\beta,\ell^\sigma_1,\ldots,\ell^\sigma_d,\ell^F_1,\ldots,\ell^F_d)$. For any $\ell \in \N$ denote by $\mathcal{I}_{d,\ell}$ a $\ell$-layer ReLU-DNN that emulates the identity on $\R^d$. By \cite[Remark~2.4]{PETERSEN2018296} (see also \cite[Proposition~2.4]{Opschoor2020}) it can be chosen so that $\mathrm{size}(\mathcal{I}_{d,\ell}) \leq 2 d \ell$.

Then from \eqref{eq:SDEEulerFinal} we have for any $t$ such that $t \in [t_n,t_{n+1}]$ 
\begin{equation}\label{eq:auxEq44} \begin{aligned}
Z^{x,i}_{t} & = Z^{x,i}_{t_n} + \beta_{\bar{\varepsilon},d}(Z^{x,i}_{t_n})(t-t_n)  + \sigma_{\bar{\varepsilon},d}(Z^{x,i}_{t_n}) (B^{d,i}_{t}-B^{d,i}_{t_n})  + F_{\bar{\varepsilon},d}(Z^{x,i}_{t_n}) \int_{t_n}^t \int_{\R^d} G^d(z) \tilde{N}^{d,i}(dt,dz)
\\ & = \mathcal{I}_{d,\ell_{max}}(Z^{x,i}_{t_n}) + \beta_{\bar{\varepsilon},d}(\mathcal{I}_{d,\ell_{max}-\ell^\beta}(Z^{x,i}_{t_n}))(t-t_n)  + \sum_{j=1}^d \sigma_{\bar{\varepsilon},d,j}(\mathcal{I}_{d,\ell_{max}-\ell^\sigma_j}(Z^{x,i}_{t_n})) (B^{d,i}_{t,j}-B^{d,i}_{t_n,j}) \\ & \quad  + \sum_{j=1}^d F_{\bar{\varepsilon},d,j}(\mathcal{I}_{d,\ell_{max}-\ell^F_j}(Z^{x,i}_{t_n})) (L^{d,i}_{t,j} - L^{d,i}_{t_n,j}).
\end{aligned} 
\end{equation}
Next we use a result on compositions of DNNs, \cite[Remark~2.6]{PETERSEN2018296} (see also \cite[Proposition~2.2]{Opschoor2020}), which essentially states that the composition of a DNN $\phi_1$ with $L_1$ layers and a DNN $\phi_2$ with $L_2$ layers can be realized as a DNN $\phi:=\phi_1 \odot \phi_2$  with $L_1+L_2$ layers and $\mathrm{size}(\phi) \leq 2 (\mathrm{size}(\phi_1)+\mathrm{size}(\phi_2))$. This shows that the last line in \eqref{eq:auxEq44} can be realized as the (randomly weighted) sum of DNNs of the same depth evaluated at $Z^{x,i}_{t_n}$. A weighted sum of DNNs of the same depth $\ell_{max}$ can again be realized by a DNN of depth $\ell_{max}$ by \cite[Lemma~3.2]{GS20_925} and therefore we obtain
\begin{equation}\label{eq:auxEq45} 
\begin{aligned}
Z^{x,i}_{t}(\omega) &  = \Phi_{t}^i(Z^{x,i}_{t_n}(\omega))
\end{aligned} 
\end{equation}
for a neural network 
$\Phi^i_t \colon \R^d \to \R^d$ with neural network weights depending on
$t,t_n,\bar{\varepsilon},d,i,h$ and $\omega$ (but not on $x$) and satisfying
\begin{equation}\label{eq:auxEq46} 
\begin{aligned}
\mathrm{size}(\Phi^i_t) & \leq \mathrm{size}(\mathcal{I}_{d,\ell_{max}}) + \mathrm{size}(\beta_{\bar{\varepsilon},d} \odot \mathcal{I}_{d,\ell_{max}-\ell^\beta}) +
\sum_{j=1}^d \mathrm{size}(\sigma_{\bar{\varepsilon},d,j} \odot \mathcal{I}_{d,\ell_{max}-\ell^\sigma_j})
\\ 
& \quad + \sum_{j=1}^d \mathrm{size}(F_{\bar{\varepsilon},d,j} \odot \mathcal{I}_{d,\ell_{max}-\ell^F_j})
\\ 
& \leq (6 d + 8 d^2) \ell_{max} + 2\mathrm{size}(\beta_{\bar{\varepsilon},d} ) + 2 \sum_{j=1}^d \mathrm{size}(\sigma_{\bar{\varepsilon},d,j}) + \mathrm{size}(F_{\bar{\varepsilon},d,j}) 
\\ 
& \leq (1+ 6 d + 8 d^2)  2 C d^p \bar{\varepsilon}^{-\rev{\hat{q}}},
\end{aligned} 
\end{equation}
where in the last step we used Assumption~\ref{ass:NNApproxCoeff} 
and that w.l.o.g.\ each layer has at least one non-zero parameter. 

Iterating \eqref{eq:auxEq44}, 
applying \eqref{eq:auxEq45} in each time-step 
and using $Z^{x,i}_{0}=x$, 
we obtain for $t \in (t_n,t_{n+1}]$
\begin{equation}\label{eq:auxEq47} 
\begin{aligned}
Z^{x,i}_t(\omega) & = \Phi^i_t \circ \Phi^i_{t_n} \circ \Phi^i_{t_{n-1}} \circ \cdots \circ \Phi^i_{t_1}(x)
\\ & = \Psi^i_t(x)
\end{aligned} 
\end{equation}
with $\Psi^i_t = \Phi^i_t \odot \Phi^i_{t_n} \odot \Phi^i_{t_{n-1}} \odot \cdots \odot \Phi^i_{t_1}$.

From \cite[Proposition~2.2]{Opschoor2020} 
in fact we have the refined bound 
$\mathrm{size}(\phi_1 \odot \phi_2) \leq 2 \mathrm{size}(\phi_1)+\mathrm{size}_{out}(\phi_2)+\mathrm{size}(\phi_2) $ and the property $\mathrm{size}_{out}(\phi_1 \odot \phi_2) = \mathrm{size}_{out}(\phi_1)$ (provided that $\phi_1$ has at least one hidden layer). But each of the networks $\Phi^i_s$ has at least one hidden layer (as $\ell_{max}\geq 2$) and so by iteratively applying these properties and using \eqref{eq:auxEq46} we obtain 
\begin{equation}\label{eq:auxEq48} \begin{aligned}
\mathrm{size}(\Psi^i_t) 
& \leq 2 \mathrm{size}(\Phi^i_t) + \mathrm{size}_{out}(\Phi^i_{t_n} \odot \Phi^i_{t_{n-1}} 
                                \odot \cdots \odot \Phi^i_{t_1}) +
 \mathrm{size}(\Phi^i_{t_n} \odot \Phi^i_{t_{n-1}} \odot \cdots \odot \Phi^i_{t_1})
\\ 
& = 2 \mathrm{size}(\Phi^i_t) + \mathrm{size}_{out}(\Phi^i_{t_n}) +
\mathrm{size}(\Phi^i_{t_n} \odot \Phi^i_{t_{n-1}} \odot \cdots \odot \Phi^i_{t_1})
\\ 
& \leq 
2 \mathrm{size}(\Phi^i_t) + \mathrm{size}_{out}(\Phi^i_{t_n}) +
2 \mathrm{size}(\Phi^i_{t_n}) + \mathrm{size}_{out}(\Phi^i_{t_{n-1}}) + 
\mathrm{size}(\Phi^i_{t_{n-1}} \odot \cdots \odot \Phi^i_{t_1}) 
\\ 
& \leq \ldots
\\ & \leq 2 \mathrm{size}(\Phi^i_t) + 3 \sum_{k=1}^n \mathrm{size}(\Phi^i_{t_k}) \leq (2+3 n) (1+ 6 d + 8 d^2)  2 C d^p \bar{\varepsilon}^{-\rev{\hat{q}}}.
\end{aligned} \end{equation} 
Next, recall that $t_1^{d,\bar{\varepsilon}} < \ldots < t_{D_{d,\bar{\varepsilon}}}^{d,\bar{\varepsilon}}$. 
Denote by $\ell_j = \mathrm{depth}(\Psi^i_{t_j^{d,\bar{\varepsilon}}})$ for 
$j=1,\ldots,D$ with $D:=D_{d,\bar{\varepsilon}}$ 

and note that $\ell_j$ is non-decreasing in $j$. 
We now use \eqref{eq:auxEq47} to write 
\begin{equation}\label{eq:auxEq49} \begin{aligned}
\Phi_{\bar{\varepsilon},d}(\mathrm{e}_{d,\bar{\varepsilon}}(Z^{x,d,h,\delta,\bar{\varepsilon},\M,i}(\omega)),K) & = \Phi_{\bar{\varepsilon},d}((Z^{x,i}_{t_1^{d,\bar{\varepsilon}}}(\omega),\ldots,Z^{x,i}_{t_{D}^{d,\bar{\varepsilon}}}(\omega)) ,K) 
\\ & =[\Phi_{\bar{\varepsilon},d} \circ(\mathcal{I}_{d,\ell_D-\ell_1} \circ \Psi^i_{t_1^{d,\bar{\varepsilon}}}(x),\ldots,\Psi^i_{t_{D}^{d,\bar{\varepsilon}}}(x),\mathcal{I}_{kd,\ell_D}(K))] 
\\ & = \bar{\Psi}^i(x,K) 
\end{aligned} \end{equation} 
for the neural network $\bar{\Psi}^i = \Phi_{\bar{\varepsilon},d} \odot [((\mathcal{I}_{d,\ell_D-\ell_1} \odot \Psi^i_{t_1^{d,\bar{\varepsilon}}},\mathcal{I}_{d,\ell_D-\ell_2} \odot \Psi^i_{t_2^{d,\bar{\varepsilon}}},\ldots,\Psi^i_{t_{D}^{d,\bar{\varepsilon}}}),\mathcal{I}_{kd,\ell_D})_{\mathrm{d}}]$ where $(\phi_1,\ldots,\phi_k)$ denotes the parallelization of the $m \in \N$ neural networks $\phi_1,\ldots,\phi_m$ and $(\phi_1,\ldots,\phi_m)_{\mathrm{d}}$ is the parallelization with distinct inputs (see for instance \cite[Section~2.1]{Opschoor2020}). 
The network size is additive with respect to these operations in the sense that $\mathrm{size}((\phi_1,\ldots,\phi_m))= \sum_{i=1}^m \mathrm{size}(\phi_m)$ and $\mathrm{size}((\phi_1,\ldots,\phi_m)_{\mathrm{d}})= \sum_{i=1}^m \mathrm{size}(\phi_m)$.
The size of $\bar{\Psi}^i$ can thus be estimated using \eqref{eq:auxEq48} and \eqref{eq:auxEq49} by
\begin{equation}\label{eq:auxEq50} \begin{aligned}
\mathrm{size}(\bar{\Psi}^i) & \leq 2 \mathrm{size}(\Phi_{\bar{\varepsilon},d}) + 2 \mathrm{size}(((\mathcal{I}_{d,\ell_D-\ell_1} \odot \Psi^i_{t_1^{d,\bar{\varepsilon}}},\mathcal{I}_{d,\ell_D-\ell_2} \odot \Psi^i_{t_2^{d,\bar{\varepsilon}}},\ldots,\Psi^i_{t_{D}^{d,\bar{\varepsilon}}}),\mathcal{I}_{kd,\ell_D})_{\mathrm{d}})
\\ 
& = 2 \mathrm{size}(\Phi_{\bar{\varepsilon},d}) + 2\mathrm{size}(\mathcal{I}_{kd,\ell_D}) + 2 \mathrm{size}(\Psi^i_{t_{D}^{d,\bar{\varepsilon}}}) + 2 \sum_{j=1}^{D-1} \mathrm{size}(\mathcal{I}_{d,\ell_D-\ell_j} \odot \Psi^i_{t_{j}^{d,\bar{\varepsilon}}}) 
\\ 
& \leq 2 C d^p \bar{\varepsilon}^{-\rev{\hat{q}}} + 2kd\ell_D  + 
4 D (2+3 N) (1+ 6 d + 8 d^2)  2 C d^p \bar{\varepsilon}^{-\rev{\hat{q}}} + 4 \sum_{j=1}^{D-1} 2d(\ell_D-\ell_j)
\\ 
& \leq  
[1+(4 D+2kd+8 D d) (2+3 N) (1+ 6 d + 8 d^2) ] 2 C d^p \bar{\varepsilon}^{-\rev{\hat{q}}}
\\ & \leq \tilde{C} h^{-1} d^{2p+4} \bar{\varepsilon}^{-q-\rev{\hat{q}}},
\end{aligned} \end{equation} 
with $\tilde{C} = 4200\max(C,1)k T C$.

Inserting \eqref{eq:auxEq49} into the definition of $U_{\varepsilon,d}$ we thus obtain
\begin{equation}\label{eq:auxEq52}
U_{\varepsilon,d}(x,K)= \frac{1}{\mathfrak{N}} \sum_{i=1}^\mathfrak{N} \bar{\Psi}^i(x,K) = \tilde{\Psi}(x,K)
\end{equation}
for a neural network $\tilde{\Psi}$ e.g.\ obtained from \cite[Lemma~3.2]{GS20_925} and satisfying
\begin{equation}\label{eq:auxEq51} \begin{aligned}
\mathrm{size}( \tilde{\Psi})  \leq \sum_{i=1}^\mathfrak{N} \mathrm{size}( \bar{\Psi}^i) & \leq \tilde{C}   h^{-1} d^{2p+4} \bar{\varepsilon}^{-q-\rev{\hat{q}}}\mathfrak{N}
\\ & \leq \tilde{C}   [\varepsilon^2 (9 \bar{c} d^r \bar{\varepsilon}^{-3q})^{-1}]^{-1} d^{2p+4} \bar{\varepsilon}^{-q-\rev{\hat{q}}}(3 \varepsilon^{-2} \bar{c} d^r \bar{\varepsilon}^{-4q}+1)
\\ & \leq 18 \max(3\bar{c},1) \bar{c} \tilde{C} \max(6 \bar{c},1)^{8q+\rev{\hat{q}}}  \varepsilon^{-4-8q-\rev{\hat{q}}}   d^{2p+4+2r+8qr+\rev{\hat{q}}r}, 
\end{aligned} \end{equation} 
where we used \eqref{eq:auxEq50} in the second inequality and inserted the choices of $ \mathfrak{N}, h$ in the third inequality and the choice of 
$\bar{\varepsilon}$ in the last inequality.  Setting $\kappa := 18 \max(3\bar{c},1) \bar{c} \tilde{C} \max(6 \bar{c},1)^{8q+\rev{\hat{q}}} $, $\mathfrak{p} = 2p+4+2r+8qr+\rev{\hat{q}}r$, $\mathfrak{q}= 4+8q+\rev{\hat{q}}$ we have thus proved in \eqref{eq:auxEq52} the claimed neural network  representation and provided in \eqref{eq:auxEq51}  the polynomial bound \eqref{eq:weightsPolynomial} on its size. This concludes the proof of Step 2 and 
finishes the proof of the theorem in the case of Assumption~\ref{ass:jumps}(i).

\textbf{The case of Assumption~\ref{ass:jumps}(ii)}
	
Consider now the case when Assumption~\ref{ass:jumps}(ii) is satisfied. 
Let $Z^{x,d,h,\delta,\bar{\varepsilon},\M,1},\ldots,Z^{x,d,h,\delta,\bar{\varepsilon},\M,\mathfrak{N}}$ 
be 
i.i.d.\ copies of the process $\hat{Z}^{x,d,h,\delta,\bar{\varepsilon},\M}$ 
which was introduced in \eqref{eq:SDEEulerFinal2}. 
Recall that $\bar{\varepsilon} \in (0,1)$, $h \in (0,1)$, 
$\delta \in (0,1)$ and $\M \in \N$ are for the time being arbitrary and will be selected later.

\textit{Step 1:} 
Denoting $\mathcal{Z}_i(x,K) = \Phi_{\bar{\varepsilon},d}(\mathrm{e}_{d,\bar{\varepsilon}}(Z^{x,d,h,\delta,\bar{\varepsilon},\M,i}),K)$ and  $\bar{\mathcal{Z}}_i(x,K) = \mathcal{Z}_i(x,K)-\Phi_{\bar{\varepsilon},d}(\mathrm{e}_{d,\bar{\varepsilon}}(0),K)$ we obtain the same error decomposition \eqref{eq:auxEq26} as in the case above. 
Furthermore, for the first term we proceed by precisely the same arguments 
used to obtain \eqref{eq:auxEq23} and obtain for any $(x,K) \in \R^d\times \R^{kd}$
\begin{equation}\label{eq:auxEq40} \begin{aligned}
& \left|U_d(x,K) - \E[\mathcal{Z}_1(x,K)]\right| 
\\ & \quad \leq \bar{\varepsilon} C d^p (1+\|x\|+\|K\|+\E[ \sup_{t \in [0,T]}\|X_t^{x,d}\|]) + C d^p \bar{\varepsilon}^{-q} \rev{D_{d,\bar{\varepsilon}}}^{1/2} \E[  \sup_{t \in [0,T]} \|X_t^{x,d}-\hat{Z}_t^{x,d,h,\delta,\bar{\varepsilon},\M}\|].
\end{aligned} 
\end{equation}
As above, 
we now estimate the last expectation above using Lemma~\ref{lem:Euler}, 
Lemma~\ref{lem:jumpTruncation} and Proposition~\ref{prop:NNApproxEuler2}. 
This yields for $\M\geq \delta^{-2} \tilde{L} d^{\bar{q}}$ 
that
\[
\begin{aligned} \E[ &  \sup_{t \in [0,T]} \|X_t^{x,d}-\hat{Z}_t^{x,d,h,\delta,\bar{\varepsilon},\M}\|] \\ 
&  \leq 
\E[  \sup_{t \in [0,T]} \|X_t^{x,d}-\bar{X}_t^{x,d,h}\|] 
+   
\E[  \sup_{t \in [0,T]} \|\bar{X}_t^{x,d,h}-Y_t^{x,d,h,\delta}\|] 
+ 
\E[  \sup_{t \in [0,T]} \|Y_t^{x,d,h,\delta}-\hat{Z}_t^{x,d,h,\delta,\bar{\varepsilon},\M}\|]
\\ 
& \leq  [h(c_3 d^4 + c_4 d^2\|x\|^2)]^{1/2} + [c_5 h(d^4 + d^2\|x\|^2) + c_6 \delta^{\bar{p}} d^{\bar{q}} (\|x\|^2+d^2)]^{1/2} \\ & \quad \quad+ [\tilde{c}_9 [\bar{\varepsilon}^{3q+1}  d^{2p} 
 +  \delta^{-2} d^{3p+\bar{q}} \bar{\varepsilon}^{-3q} \M^{-1} ] \left(1  + \|x\|^2 \right)]^{1/2}.
\end{aligned}
\]
Inserting this estimate and \eqref{eq:supL2} from Corollary~\ref{cor:supL2} into \eqref{eq:auxEq40} we obtain
\begin{equation}\label{eq:auxEq41} \begin{aligned}
& \left|U_d(x,K) - \E[\mathcal{Z}_1(x,K)]\right|^2 \leq 6 \bar{\varepsilon}^2 C^2 d^{2p} (\|x\|^2+\|K\|^2+\tilde{c}_3 d^2 + \tilde{c}_4 \|x\|^2)  + 6 C^3 d^{3p} \bar{\varepsilon}^{-3q} [h(c_3 d^4 + c_4 d^2\|x\|^2)]   \\ & \quad + 6 C^3 d^{3p} \bar{\varepsilon}^{-3q}[c_5 h(d^4 + d^2\|x\|^2) + c_6 \delta^{\bar{p}} d^{\bar{q}} (\|x\|^2+d^2) +  \tilde{c}_9 (\bar{\varepsilon}^{3q+1}  d^{2p} 
+  \delta^{-2} d^{3p+\bar{q}} \bar{\varepsilon}^{-3q} \M^{-1} ) \left(1  + \|x\|^2 \right)]
\end{aligned} \end{equation}
with $\tilde{c}_3 = 2\max(\bar{c}_3,1), \tilde{c}_4 = 2\bar{c}_4$.

For the second term in the error decomposition \eqref{eq:auxEq26} one may use precisely the same arguments as in \eqref{eq:auxEq25} (but now with the moment bound \eqref{eq:momentBoundEulerFinal2} instead of \eqref{eq:momentBoundEulerFinal}), which yields

\begin{equation} \label{eq:auxEq42} \begin{aligned}
\E\left[ \left|\E[\bar{\mathcal{Z}}_1(x,K)]- \bar{\mathcal{Z}}_1(x,K) \right|^2 \right] & 
 \leq C^3 d^{3p} \bar{\varepsilon}^{-3q}  (\tilde{c}_8 d^p \bar{\varepsilon}^{-q} + \tilde{c}_7 \|x\|^2).
\end{aligned} \end{equation}

Estimates \eqref{eq:auxEq41} and \eqref{eq:auxEq42} can now be inserted into \eqref{eq:auxEq26}, this yields

\begin{equation}\label{eq:auxEq43} \begin{aligned}
\int_{\R^d\times \R^{kd}} & \E\left[\left|U_d(x,K) - \frac{1}{\mathfrak{N}} \sum_{i=1}^\mathfrak{N} \mathcal{Z}_i(x,K) \right|^2 \right] \mu^d(dx,dK) 
\\ & =  \int_{\R^d\times \R^{kd}} \left|U_d(x,K) - \E[\mathcal{Z}_1(x,K)]\right|^2 + \frac{1}{\mathfrak{N}} \E\left[ \left|\E[\bar{\mathcal{Z}}_1(x,K)]- \bar{\mathcal{Z}}_1(x,K) \right|^2 \right] \mu^d(dx,dK)
\\ & \leq \bar{c} d^{r} \int_{\R^d\times \R^{kd}}  (1+\|x\|^2+\|K\|^2) [\bar{\varepsilon}^2 + \bar{\varepsilon}^{-3q} (2h +  \delta^{\bar{p}} +   \bar{\varepsilon}^{3q+1} 
+  \delta^{-2}  \bar{\varepsilon}^{-3q} \M^{-1}) + \mathfrak{N}^{-1}  \bar{\varepsilon}^{-4q}]  \mu^d(dx,dK) 
\\ & \leq C \bar{c} d^{r+p} [\bar{\varepsilon}^2 + \bar{\varepsilon}^{-3q} (2h +  \delta^{\bar{p}} +   \bar{\varepsilon}^{3q+1} 
+  \delta^{-2}  \bar{\varepsilon}^{-3q} \M^{-1}) + \mathfrak{N}^{-1}  \bar{\varepsilon}^{-4q}] 
\end{aligned}\end{equation}
with $\bar{c} = 6\max(C^2\max(1+\tilde{c}_4,1,\tilde{c}_3),C^3\max(c_3,c_4,c_5,c_6,\tilde{c}_7,\tilde{c}_8,\tilde{c}_9))$, $r=6p+4+\bar{q}$.

Now set $\tilde{c} = C \bar{c}$, $\tilde{p}=r+p$ and choose $\bar{\varepsilon}  = \varepsilon (\max(8\sqrt{3} \tilde{c},1) d^{\tilde{p}})^{-1}$, $h = \varepsilon^2 (\max(48 \tilde{c},1) d^{\tilde{p}} \bar{\varepsilon}^{-3q})^{-1}$,  $\delta = h^{1/\bar{p}}$, $\M = \lceil \varepsilon^{-2} \delta^{-2} \bar{\varepsilon}^{-6q} d^{\max(\tilde{p},\tilde{q})}  \max(12 \tilde{c} , \tilde{L} ) \rceil   $, 
$\mathfrak{N} = \lceil 12 \varepsilon^{-2} \tilde{c} d^{\tilde{p}} \bar{\varepsilon}^{-4q} \rceil $. 

Denote by $R := \int_{\R^d\times \R^{kd}} |U_d(x,K) - \frac{1}{\mathfrak{N}} \sum_{i=1}^\mathfrak{N} \mathcal{Z}_i(x,K) |^2  \mu^d(dx,dK)$. 
Inserting these choices in the bound \eqref{eq:auxEq43} we obtain

\begin{equation}\label{eq:auxEq55} \begin{aligned}
\E&\left[ R \right] < \frac{\varepsilon^2}{9} + \frac{\varepsilon^2}{9} +   \frac{\varepsilon^2}{9} = \frac{\varepsilon^2}{3},  
\end{aligned}\end{equation}

i.e.\ that the bound \eqref{eq:auxEq28} also holds with the current choice of $Z^{x,d,h,\delta,\bar{\varepsilon},\M,i}$ and $\frac{\varepsilon^2}{3}$ instead of $\varepsilon^2$. Combining this with Markov's inequality, the fact that $N^d(\cdot,A_\delta)$ is a Poisson process with intensity $\nu^d(A_\delta)$ and estimate \eqref{eq:AdeltaMass} yields

\begin{equation}\label{eq:auxEq56} \begin{aligned}
\P\left(\{R \geq \varepsilon^2\} \cup \{ N^d(T,A_\delta) \geq 3 T\delta^{-2} \tilde{L} d^{\bar{q}} \} \right) & \leq  \frac{\E[R]}{ \varepsilon^2} + \frac{\E[N^d(T,A_\delta)]}{ 3 T\delta^{-2} \tilde{L} d^{\bar{q}} } \leq  \frac{1}{3} + \frac{\nu^d(A_\delta)}{ 3 \delta^{-2} \tilde{L} d^{\bar{q}} } \leq \frac{2}{3}.
\end{aligned}\end{equation}
Consequently, $\P(R< \varepsilon^2, N^d(T,A_\delta) < 3 T\delta^{-2} \tilde{L} d^{\bar{q}} ) \geq \frac{1}{3} >0$.
Thus, there exists $\omega \in \Omega$ such that \eqref{eq:epsbound} holds 
and, in addition, 
\begin{equation}\label{eq:PoissonBound}
N^d(T,A_\delta)(\omega) \leq 3 T\delta^{-2} \tilde{L} d^{\bar{q}}.
\end{equation}

\textit{Step 2:}
Let $i \in \{1, \ldots,\mathfrak{N}\}$ and write $Z^{x,i}:=Z^{x,d,h,\delta,\bar{\varepsilon},\M,i}$. Recall that $Z^{x,i}$ is an independent copy of the process $\hat{Z}^{x,d,h,\delta,\bar{\varepsilon},\M}$ introduced in \eqref{eq:SDEEulerFinal2}. Denote by $P_t^{d,i} = \int_{A_\delta} y N^{d,i}(t,d y)$ the compound Poisson process of jumps larger than $\delta$ associated to $Z^{x,i}$. Then for $t \in [t_n,t_{n+1}]$
\begin{equation} \label{eq:auxEq53} \begin{aligned}
Z^{x,i}_t & = Z^{x,i}_{t_n} + \beta_{\bar{\varepsilon},d}(Z^{x,i}_{t_n})(t-t_n)  + \sigma_{\bar{\varepsilon},d}(Z^{x,i}_{t_n}) (B^{d,i}_{t}-B^{d,i}_{t_n})  + \sum_{t_n \leq s \leq t}  \gamma_{\bar{\varepsilon},d}(Z^{x,i}_{t_n},\Delta P_s^{d,i}) \mathbbm{1}_{A_\delta}(\Delta P_s^{d,i}) \\ & \quad  - \frac{(t-t_n)\nu^d(A_\delta)}{\M}\sum_{m=1}^\M \gamma_{\bar{\varepsilon},d}(Z^{x,i}_{t_n},V_{m,t_n}^i),
\end{aligned} \end{equation}
where the first sum is only over finitely many non-zero summands, 
see for instance \cite[Section~4.3.2]{Applebaum2009} 
(the number of non-zero terms is $N(t,A_{\delta})(\omega)-N(s,A_{\delta})(\omega)$, which is finite due to \cite[Lemma~2.3.4]{Applebaum2009}). 
Let $\ell_{max} = \max(2,\ell^\beta,\ell^\sigma_1,\ldots,\ell^\sigma_d,\ell^P_{max},\ell^V_1,\ldots,\ell^V_\M)$ with $\ell^\beta,(\ell^\sigma_j)_{j=1,\ldots,d}$ as before, $\ell^V_j = \mathrm{depth}(\gamma_{\bar{\varepsilon},d}(\cdot,V_{j,t_n}^i(\omega)))$ for $j=1,\ldots,\M$ and $\ell^P_{max} = \max_{s \in [t_n,t]} \ell^P_s$, where $\ell^P_s :=\mathrm{depth}(\gamma_{\bar{\varepsilon},d}(\cdot,\Delta P_{s}^{d,i}(\omega)))\mathbbm{1}_{A_\delta}(\Delta P_s^{d,i}(\omega))$.  
Then precisely the same reasoning that we employed to obtain the neural network representation \eqref{eq:auxEq45} 
from \eqref{eq:auxEq44} can be applied here. 
More specifically for each $t \in [t_n,t_{n+1}]$, there exists a neural network $\Phi^i_t \colon \R^d \to \R^d$ with neural network weights depending on $t,t_n,\bar{\varepsilon},d,i,h,\delta,\M$ and $\omega$ (but not on $x$) such that for all $x \in \R^d$ the representation $Z^{x,i}_{t}(\omega)   = \Phi_{t}^i(Z^{x,i}_{t_n}(\omega))$ holds true and the number of non-zero weights can be estimated by
\begin{equation}\label{eq:auxEq54} \begin{aligned}
& \mathrm{size}(\Phi^i_t)  \leq \mathrm{size}(\mathcal{I}_{d,\ell_{max}}) + \mathrm{size}(\beta_{\bar{\varepsilon},d} \odot \mathcal{I}_{d,\ell_{max}-\ell^\beta}) +
\sum_{j=1}^d \mathrm{size}(\sigma_{\bar{\varepsilon},d,j} \odot \mathcal{I}_{d,\ell_{max}-\ell^\sigma_j})
\\ 
& \quad + \sum_{t_n \leq s \leq t} \mathbbm{1}_{A_\delta}(\Delta P_s^{d,i}(\omega)) \mathrm{size}(\gamma_{\bar{\varepsilon},d}(\cdot,\Delta P_s^{d,i}(\omega)) \odot \mathcal{I}_{d,\ell_{max}-\ell^P_s}) + \sum_{m=1}^\M \mathrm{size}(\gamma_{\bar{\varepsilon},d}(\cdot,V_{m,t_n}^i(\omega)) \odot \mathcal{I}_{d,\ell_{max}-\ell^V_m})
\\ 
& \leq 2 d \ell_{max} (3 + 2 d + 2 N^d(T,A_\delta)(\omega) + 2 \M)  + 2\mathrm{size}(\beta_{\bar{\varepsilon},d} ) + 2 \sum_{j=1}^d \mathrm{size}(\sigma_{\bar{\varepsilon},d,j}) + 2 (N^d(T,A_\delta)(\omega)+\M) \mathrm{size}(\gamma_{\bar{\varepsilon},d})
\\ 
& \leq 2 C d^p \bar{\varepsilon}^{-\rev{\hat{q}}} (1+N^d(T,A_\delta)(\omega)+\M+3d + 2 d^2 + 2d N^d(T,A_\delta)(\omega) + 2 d\M) 
\\ 
& \leq 12 C \max(1,2T\tilde{L})d^{p+\bar{q}+2} \bar{\varepsilon}^{-\rev{\hat{q}}} (1 + 2 \delta^{-2}  + \M),
\end{aligned} \end{equation}
where we used the bounds on the size from Assumption~\ref{ass:NNApproxCoeff} and employed \eqref{eq:PoissonBound} for the last step. For the remainder of the proof we can now repeat precisely the same arguments used to obtain first the neural network representation \eqref{eq:auxEq47} with a bound on the weights (obtained as in \eqref{eq:auxEq48}) 
\begin{equation}\label{eq:auxEq57} \begin{aligned}
\mathrm{size}(\Psi^i_t) & \leq (2+3 n) 12 C \max(1,2T\tilde{L})d^{p+\bar{q}+2} \bar{\varepsilon}^{-\rev{\hat{q}}} (1 + 2 \delta^{-2}  + \M).
\end{aligned} \end{equation} 
Letting $\tilde{C}_0:=12 C\max(1,2T\tilde{L})$ 
we then obtain the neural network representation \eqref{eq:auxEq49} with weights bounded (as in \eqref{eq:auxEq50}) by 
\begin{equation}\label{eq:auxEq58} \begin{aligned}
\mathrm{size}(\bar{\Psi}^i) &
\leq 2 \mathrm{size}(\Phi_{\bar{\varepsilon},d}) + 2\mathrm{size}(\mathcal{I}_{kd,\ell_D}) + 2 \mathrm{size}(\Psi^i_{t_{D_{d,\bar{\varepsilon}}}^{d,\bar{\varepsilon}}}) + 2 \sum_{j=1}^{D_{d,\bar{\varepsilon}}-1} \mathrm{size}(\mathcal{I}_{d,\ell_D-\ell_j} \odot \Psi^i_{t_{j}^{d,\bar{\varepsilon}}}) 
\\ 
& \leq 2 C d^p \bar{\varepsilon}^{-\rev{\hat{q}}} + 2kd\ell_D  + 
4 D_{d,\bar{\varepsilon}} (2+3 N) \tilde{C}_0 d^{p+\bar{q}+2} \bar{\varepsilon}^{-\rev{\hat{q}}} (1 + 2 \delta^{-2}  + \M) + 4 \sum_{j=1}^{D_{d,\bar{\varepsilon}}-1} 2d(\ell_D-\ell_j)
\\ 
& \leq  
2 C d^p \bar{\varepsilon}^{-\rev{\hat{q}}} + (4 D_{d,\bar{\varepsilon}}+2kd+8 D_{d,\bar{\varepsilon}} d) (2+3 N) \tilde{C}_0 d^{p+\bar{q}+2} \bar{\varepsilon}^{-\rev{\hat{q}}} (1 + 2 \delta^{-2}  + \M)
\\ 
& \leq  
2 C d^p \bar{\varepsilon}^{-\rev{\hat{q}}} + 14 k C (2+3 T h^{-1}) \tilde{C}_0 d^{2p+\bar{q}+3} \bar{\varepsilon}^{-q-\rev{\hat{q}}} (1 + 2 \delta^{-2}  + \M)
\\ 
& \leq  
 \tilde{C} h^{-1} \delta^{-2} d^{2p+\bar{q}+3} \bar{\varepsilon}^{-q-\rev{\hat{q}}} (1  + \M)
\end{aligned} \end{equation} 
with $\tilde{C} =60 C\max(1,7k\tilde{C}_0) \max(T,1) $.
Altogether, 
we obtain the representation \eqref{eq:auxEq52} also in this case 
for a neural network $\tilde{\Psi}$ satisfying (analogously to \eqref{eq:auxEq51}) 
the bound
\begin{equation}\label{eq:auxEq59} 
\begin{aligned}
\mathrm{size}( \tilde{\Psi}) 
& 
\leq  
\tilde{C} h^{-1} \delta^{-2} d^{2p+\bar{q}+3} \bar{\varepsilon}^{-q-\rev{\hat{q}}} (1  + \M)\mathfrak{N}
\\ 
&  
\leq 
\tilde{C} \varepsilon^{-2-\frac{8}{\bar{p}}} (\max(48 \tilde{c},1) d^{\tilde{p}} \bar{\varepsilon}^{-3q})^{1+\frac{4}{\bar{p}}} d^{2p+\bar{q}+3} \bar{\varepsilon}^{-q-\rev{\hat{q}}} (1 + \varepsilon^{-2} \bar{\varepsilon}^{-6q} d^{\max(\tilde{p},\tilde{q})}  \max(12 \tilde{c} , \tilde{L} ) )(1+12 \varepsilon^{-2} \tilde{c} d^{\tilde{p}} \bar{\varepsilon}^{-4q})
\\ 
& 
\leq 
4 \tilde{C} \max(48 \tilde{c},1)^{1+\frac{4}{\bar{p}}} \max(1,12 \tilde{c} , \tilde{L})  \varepsilon^{-6-\frac{8}{\bar{p}}} d^{2\tilde{p}+\frac{4\tilde{p}}{\bar{p}}+2p+\bar{q}+3+\max(\tilde{p},\tilde{q})} \bar{\varepsilon}^{-14q-\rev{\hat{q}}-\frac{12 q}{\bar{p}}} 
\\ 
&  
\leq 
\kappa \varepsilon^{-6-\frac{8}{\bar{p}}-14q-\rev{\hat{q}}-\frac{12q}{\bar{p}}}  
 d^{3\tilde{p}+\frac{4\tilde{p}}{\bar{p}}+\tilde{q}+2p+\bar{q}+3+ 14q\tilde{p} +\rev{\hat{q}}\tilde{p}+ \frac{12q\tilde{p}}{\bar{p}}}, 
\end{aligned} 
\end{equation} 
where we have set 
$\kappa = 4\tilde{C} \max(1,12 \tilde{c} , \tilde{L}) 
 \max(48 \tilde{c},1)^{1+\frac{4}{\bar{p}}} (\max(8\sqrt{3} \tilde{c},1))^{14q+\rev{\hat{q}}+\frac{12q}{\bar{p}}}$ 
and we used \eqref{eq:auxEq58} in the first inequality and 
inserted the choices of $ \mathfrak{N}, h, \M, \delta$ in the second inequality and the choice of 
$\bar{\varepsilon}$ in the last inequality.  
This finishes the proof of the Theorem also in the case of Assumption~\ref{ass:jumps}(ii).
\end{proof}
\rev{
\begin{remark}
The proof crucially relies on a refined bound for the size of compositions of neural networks, see \cite[Proposition~2.2]{Opschoor2020}, which guarantees that the constant appearing in \eqref{eq:auxEq48} does not grow exponentially in $d$ and $\varepsilon^{-1}$. Using instead the bound $\mathrm{size}(\phi_1 \odot \phi_2) \leq 2 (\mathrm{size}(\phi_1)+\mathrm{size}(\phi_2))$ would lead to a factor $2^n$ in  \eqref{eq:auxEq48} and thereby yield constants that grow exponentially in $d$ and $\varepsilon^{-1}$.
\end{remark}
}
%%%%%%%%%%%%%%%%%%%%%%%%%%%%%%%%%%%%%%%%%%%%%%%%%%
\subsection{Case without path-dependence} 
\label{sec:woPath}
%%%%%%%%%%%%%%%%%%%%%%%%%%%%%%%%%%%%%%%%%%%%%%%%%%
As a first consequence of Theorem~\ref{thm:main} we obtain a DNN approximation result for European options, i.e., functionals which only depend on the terminal value. For each $d \in \N$ let $\varphi_{d} \colon \R^d \times \R^{kd} \to \R$ be a \textit{parametric European payoff} function. In this section the function to be approximated is $U_d \colon \R^{d} \times \R^{kd} \to \R$ given by
 \[
(x,K)\mapsto U_d(x,K):=\E[\varphi_{d}(X_T^{x,d},K)].
\]
This is a special case of the situation considered in Theorem~\ref{thm:main} with $\Phi_{d}(y) = \varphi_d(y_T)$. Thus, Theorem~\ref{thm:main} can be directly applied. To facilitate reading we have written explicitly the simplifications that appear in this case for the assumption on the payoff, Assumption~\ref{ass:NNApproxPayoff}, cf. also Remark~\ref{rmk:European}. In particular, the assumptions imposed in Corollary~\ref{cor:main} are precisely the same as in Theorem~\ref{thm:main} but specialized to the case $\Phi_{d}(y) = \varphi_d(y_T)$.
\begin{corollary}\label{cor:main} 
Assume that 
\begin{itemize} 
	\item 
	the coefficients of the SDE \eqref{eq:SDEnew} satisfy the Lipschitz and growth conditions 
	in Assumption~\ref{ass:LipschitzGrowth},
	\item 
	the jumps of the process satisfy Assumption~\ref{ass:jumps}, that is, 
	either we are in the case of a L\'evy-driven SDE or the small jumps 
	exhibit decay \eqref{eq:smallJumpDecay}, \eqref{eq:LevymeasureGrowth},
	\item 
	the coefficient functions satisfy the approximation hypothesis
	Assumption~\ref{ass:NNApproxCoeff} and Assumption~ \ref{ass:NNApproxPayoff}(ii) holds, 
	\item there exist $C>0$, $p,\rev{\hat{q}}\geq 0$ and for each $d \in \mathbb{N}$, $\varepsilon \in (0,1]$ 
        there exist neural networks  $\phi_{\varepsilon,d} \colon \R^{d} \times \R^{kd} \to \R$
        such that for each $d \in \N$, $\varepsilon \in (0,1]$,
        the European payoff approximation condition \eqref{eq:EuropeanPayoffCond} holds, \rev{with $q\geq0$ as in Assumption~\ref{ass:NNApproxCoeff}}.
\end{itemize} 
Then there exist constants $\kappa, \mathfrak{p}, \mathfrak{q} > 0$ and  
neural networks $U_{\varepsilon,d} \colon \R^d \times \R^{kd} \to \R$, 
$d \in \N$, $\varepsilon \in (0,1]$ 
such that for any $d \in \N$ and target accuracy $\varepsilon \in (0,1]$ 
\begin{align}
\label{eq:weightsPolynomialcor} 
\mathrm{size}(U_{\varepsilon,d}) &  \leq \kappa d^\mathfrak{p} \varepsilon^{-\mathfrak{q}}  
\\\label{eq:errorCor}
\left(\int_{\R^d\times \R^{kd}} |U_d(x,K) - U_{\varepsilon,d}(x,K)|^2  \mu^d(dx,dK)\right)^{1/2}&  < \varepsilon.
\end{align}
\end{corollary}

%%%%%%%%%%%%%%%%%%%%%%%%%%%%%%%%%%%%%%%%%%%%%%%%%%
\subsection{Expression rate results for PIDEs} 
\label{sec:solPIDE}
%%%%%%%%%%%%%%%%%%%%%%%%%%%%%%%%%%%%%%%%%%%%%%%%%%
As a second consequence of Theorem~\ref{thm:main} we obtain a DNN expression rate result 
for the solution of the PIDE~\eqref{eq:LevyPDE} (which is identical to \eqref{eq:LevyPDEnew}). 
For each $d \in \N$ let $\varphi_{d} \colon \R^d \to \R$ be a continuous function with polynomial growth. 
With Assumptions~\ref{ass:LipschitzGrowth} and \ref{ass:PointwiseLipschitz}, 
Proposition~\ref{prop:ExUniq} ensures existence of a unique viscosity solution  
with polynomial growth of the PIDE~\eqref{eq:LevyPDE}.
We denote this solution by $u_d \in C([0,T] \times\R^d,\R)$. 
The next result proves that $u_d(0,\cdot)$ can be approximated by ReLU DNNs without the CoD.
We write $\mu^d_1(dx)$ for the $x$-marginal probability measure of $\mu^d$, i.e., 
$\mu^d_1(A) = \int_{\R^{kd}} \mu^d(A,dK)$ for $A \in \mathcal{B}(\R^d)$.  
\begin{corollary}\label{cor:PIDE} 
	Assume that 
	\begin{itemize} 
		\item 
		the coefficients of the PIDE \eqref{eq:LevyPDE} satisfy the Lipschitz and growth conditions 
		in Assumptions~\ref{ass:LipschitzGrowth} and \ref{ass:PointwiseLipschitz},
		\item 
		$\gamma^d$ satisfies Assumption~\ref{ass:jumps},
		\item 
		the coefficient functions satisfy the approximation hypothesis
		Assumption~\ref{ass:NNApproxCoeff} \rev{with $q\geq0$} and Assumption~ \ref{ass:NNApproxPayoff}(ii) holds, 
		\item there exist $C>0$, $p,\rev{\hat{q}}\geq 0$ and for each $d \in \mathbb{N}$, $\varepsilon \in (0,1]$ there exist neural networks  $\phi_{\varepsilon,d} \colon \R^{d} \to \R$ such that for each $d \in \N$, $\varepsilon \in (0,1]$, $x \in \R^d$ 
	\begin{equation}\label{eq:PIDECond} \begin{aligned}
	|\varphi_{d}(x)-\phi_{\varepsilon,d}(x)|  & \leq \varepsilon C d^p (1+\|x\|+\|K\|), 
	\\  \mathrm{size}(\phi_{\varepsilon,d})  & \leq C d^p \varepsilon^{-\rev{\hat{q}}},
	\\ \mathrm{Lip}(\phi_{\varepsilon,d}) & \leq C d^p \varepsilon^{-q}.
	\end{aligned}
	\end{equation}			
	\end{itemize} 
	Then there exist constants $\kappa, \mathfrak{p}, \mathfrak{q} > 0$ and  
	neural networks $u_{\varepsilon,d} \colon \R^d \to \R$, 
	$d \in \N$, $\varepsilon \in (0,1]$ 
	such that for any $d \in \N$ and target accuracy $\varepsilon \in (0,1]$ 
	\begin{align}
	\label{eq:weightsPolynomialcorPIDE} \mathrm{size}(u_{\varepsilon,d}) &  \leq \kappa d^\mathfrak{p} \varepsilon^{-\mathfrak{q}}  \\
	\left(\int_{\R^d} |u_d(0,x) - u_{\varepsilon,d}(x)|^2  \mu^d(dx)\right)^{1/2}&  < \varepsilon.
	\end{align}
\end{corollary}
\begin{proof} 
Fix $d \in \N$. 
Under Assumptions~\ref{ass:LipschitzGrowth}, \ref{ass:PointwiseLipschitz} we obtain from \cite[Theorem~3.4]{BBP1997} that $u_d(0,\cdot)$ has a representation in terms of stochastic integrals: for all $x \in \R^d$, $u_d(0,x) = \mathcal{Y}_0^{x}$, 
where $\mathcal{Y}_0^{x}$ is deterministic and there exist an $\R^{d \times d}$-valued progressively-measurable stochastic process $\mathcal{Z}^{x}$ and a mapping $ \mathcal{U}^{x} \colon \Omega \times [0,T] \times (\R^d \setminus \{0\}) \to \R$ with 
\begin{equation}\label{eq:StochRepr}
 \mathcal{Y}_0^{x} = \varphi_{d}(X^{x,d}_T) - \int_0^T \mathcal{Z}^{x}_t d W^d_t - \int_0^T \int_{\R^d} \mathcal{U}^{x}_t(z) \tilde{N}^d(dt,dz), 
\end{equation}
 $\E[\int_0^T \|\mathcal{Z}^{x}_t\|_{F}^2 d t ] < \infty $, $\mathcal{U}^{x}$ is $\mathcal{P} \otimes \mathcal{B}((\R^d \setminus \{0\}))$-measurable (with $\mathcal{P}$ denoting the predictable $\sigma$-algebra) and $\E[\int_0^T \int_{\R^d}  |\mathcal{U}^{x}_t(z)|^2 \nu^d(dz) d t ] < \infty $. 
These conditions guarantee that the stochastic integrals in \eqref{eq:StochRepr}
are martingales (see for instance \cite[Theorem~IV.2.2]{RevuzYor} and \cite[Theorem~4.2.3]{Applebaum2009}).
Taking expectations in \eqref{eq:StochRepr} we thus obtain
$u_d(0,x) = \mathcal{Y}_0^{x} = \E[\varphi_{d}(X^{x,d}_T)]$. 
Setting $U_d(x,K) = u_d(0,x)$ we are thus precisely in the setting of Corollary~\ref{cor:main}.
So, the claim follows from  Corollary~\ref{cor:main}. 
\end{proof}
\subsection{Application to basket option pricing}
\label{sec:Basket}
Theorem~\ref{thm:main} can be applied in 
valuation of derivative contracts on baskets in mathematical finance. 
Corollary~\ref{cor:call} shows that if market option prices are ``generated'' 
from an  (unknown) underlying market model with jumps
satisfying the Lipschitz, growth and approximation conditions 
formulated in Assumptions~\ref{ass:LipschitzGrowth}, \ref{ass:jumps}, \ref{ass:NNApproxCoeff}, 
then prices of derivative contracts can be approximated by suitable DNNs without the CoD.
\begin{corollary}\label{cor:call} 
Fix starting values $x_0^d\in \R^d$ with $\|x_0^d\| \leq C d^p $.   
Let $N \in \N$, $K_1,\ldots,K_N \in [0,\infty)$ 
and 
$w^d \in \R^d$ with $\sup_{d \in \N} \max_{i} |w^d_{i}| < \infty$ be given.
Assume that 
Assumptions~\ref{ass:LipschitzGrowth}, \ref{ass:jumps} and \ref{ass:NNApproxCoeff} are satisfied
and 
\begin{equation}\label{eq:basketPrices}
\hat{C}(T,K_i) = \E[(w^d \cdot X_T^{d,x_0^d}-K_i)_+]\;,\quad i=1,\ldots,N\;.
\end{equation} 
Then there exist constants $\kappa, \mathfrak{p}, \mathfrak{q} > 0$ and  
neural networks $\mathcal{C}_{\varepsilon,d} \colon \R \to \R$, $d \in \N$, $\varepsilon \in (0,1]$ 
such that for any $d \in \N$ and target accuracy $\varepsilon \in (0,1]$ 
\begin{align}
\label{eq:weightsPolynomialcor2} 
\mathrm{size}(\mathcal{C}_{\varepsilon,d}) &  \leq \kappa d^\mathfrak{p} \varepsilon^{-\mathfrak{q}}  
\\
\left(\frac{1}{N}\sum_{i=1}^N |\hat{C}(T,K_i) - \mathcal{C}_{\varepsilon,d}(K_i)|^2 \right)^{1/2}&  < \varepsilon.
\end{align}
\end{corollary}
\begin{proof}
Let 
$\mu^d(dx,dK)= \delta_{\{x_0^d\}} \otimes \frac{1}{N}\sum_{i=1}^N \delta_{\{K_i e_1\}}$ 
where $e_1^d=(1,0,\ldots,0) \in \R^{kd}$. 
Then, for all $d \in \N$ it holds that 
$\int_{\R^d\times \R^{kd}} (1+\|x\|^2+\|K\|^2) \mu^d(dx,dK) 
  = 
  \frac{1}{N}\sum_{i=1}^N (1+\|x_0^d\|^2+|K_i|^2) 
  \leq C d^p$ 
and therefore Assumption~ \ref{ass:NNApproxPayoff}(ii) is satisfied. 
Furthermore, $\varphi_d(x,K) = (w^d \cdot x-K_1)_+$ 
is a ReLU DNN with $L=2$, $N_1=1$, $N_2=1$, $A^2=1$, $b^2=0$, $A^1= [(w^d)^\top,-1]$, $b^1=0$. 
Setting $\phi_{\varepsilon,d}(x,K) = \varphi_d(x,K)$ for each $\varepsilon \in (0,1]$, we obtain that 
the European payoff approximation condition \eqref{eq:EuropeanPayoffCond} holds. 
Thus, the hypotheses of Corollary~\ref{cor:main} are satisfied and therefore 
there exist constants $\kappa, \mathfrak{p}, \mathfrak{q} > 0$ and  
neural networks 
$U_{\varepsilon,d} \colon \R^d \times \R^{kd} \to \R$, $d \in \N$, $\varepsilon \in (0,1]$ 
such that for any $d \in \N$, $\varepsilon \in (0,1]$ 
condition \eqref{eq:weightsPolynomialcor} and the error estimate \eqref{eq:errorCor} hold. 
Rewriting
\[
\int_{\R^d\times \R^{kd}} |U_d(x,K) - U_{\varepsilon,d}(x,K)|^2 \mu^d(dx,dK) 
= 
\frac{1}{N}\sum_{i=1}^N |\E[\phi_{\varepsilon,d}(X_T^{x_0^d,d},e_{1} K_i)] - U_{\varepsilon,d}(x_0^d,K_i)|^2,
\]
using that 
$\E[\phi_{\varepsilon,d}(X_T^{x_0^d,d},e_1 K_i)] 
 = 
 \E[(w^d \cdot X_T^{x_0^d,d}-K_i)^+] 
 = 
 \hat{C}(T,K_i)$ and setting 
$\mathcal{C}_{\varepsilon,d}(K)=U_{\varepsilon,d}(x_0^d,e_1 K)$ 
(which is a DNN satisfying \eqref{eq:weightsPolynomialcor2}) 
then yields the claim. 
\end{proof}
%%%%%%%%%%%%%%%%%%%%%%%%%%%%%%%%%%%%%%%%%%%%%%%%%%
\section{Conclusions}
\label{sec:Concl}
%%%%%%%%%%%%%%%%%%%%%%%%%%%%%%%%%%%%%%%%%%%%%%%%%%
We have shown that a certain class of deep ReLU neural networks can \rev{approximate} viscosity solutions
of \rev{a suitable class of} linear partial integrodifferential equations without the CoD.
In addition, we have shown that deep ReLU NNs can approximate expectations 
of certain path-dependent functions of stochastic differential equations with jumps 
without the CoD.
Due to the rather weak assumptions (global Lipschitz and polynomial growth of the 
\rev{characteristic} triplets of the (Feller-)L\'{e}vy process), 
the main results on DNN expression rate bounds comprise a large number of 
special cases: pure diffusion, linear advection and pure jump.

The present analysis can also serve as building block in the analysis
of nonlinear cases, as considered e.g.\ in \cite{BBP1997}. 
There, Feynman-Kac type representations of viscosity solutions of 
semilinear parabolic PDEs with integrodifferential terms have been
established via backward SDEs with jumps. 
\rev{The present analysis constitutes, together with a so-called Picard-iteration approach of \cite{BeckGononAJPicard2020}, \cite{EHutzenthalerJentzenKruse2019MLP}, \cite{hutzenthaler2017multi} 
    a foundation to develop high-dimensional %of extensions of DNN 
    approximation bounds to certain nonlinear PIDEs, 
    such as the recent work \cite{NeufWu2022}}.

\section*{Acknowledgement}
This project has been partially funded by the Deutsche Forschungsgemeinschaft (DFG, German
Research Foundation) – 464123384. 

{\small 
\bibliographystyle{amsalpha}
\bibliography{references}

\newcommand{\etalchar}[1]{$^{#1}$}
\providecommand{\bysame}{\leavevmode\hbox to3em{\hrulefill}\thinspace}
\providecommand{\MR}{\relax\ifhmode\unskip\space\fi MR }
\providecommand{\MRhref}[2]{%
  \href{http://www.ams.org/mathscinet-getitem?mr=#1}{#2}
}
\providecommand{\href}[2]{#2}
\begin{thebibliography}{GHJvW18}

\bibitem[App09]{Applebaum2009}
David Applebaum, \emph{L\'{e}vy processes and stochastic calculus}, second ed.,
  Cambridge Studies in Advanced Mathematics, vol. 116, Cambridge University
  Press, Cambridge, 2009. \MR{2512800}

\bibitem[AR01]{AsRos2001}
S\o~ren Asmussen and Jan Rosi\'{n}ski, \emph{Approximations of small jumps of
  {L}\'{e}vy processes with a view towards simulation}, J. Appl. Probab.
  \textbf{38} (2001), no.~2, 482--493. \MR{1834755}

\bibitem[Bac17]{Bach2017}
Francis Bach, \emph{Breaking the curse of dimensionality with convex neutral
  networks}, J. Mach. Learn. Res. \textbf{18} (2017), Paper No. 19, 53.
  \MR{3634886}

\bibitem[Bar93]{Barron1993}
Andrew~R. Barron, \emph{Universal approximation bounds for superpositions of a
  sigmoidal function}, IEEE Trans. Inform. Theory \textbf{39} (1993), no.~3,
  930--945. \MR{1237720}

\bibitem[BBP97]{BBP1997}
Guy Barles, Rainer Buckdahn, and Etienne Pardoux, \emph{Backward stochastic
  differential equations and integral-partial differential equations},
  Stochastics Stochastics Rep. \textbf{60} (1997), no.~1-2, 57--83.
  \MR{1436432}

\bibitem[BCJ19]{Becker2019}
Sebastian Becker, Patrick Cheridito, and Arnulf Jentzen, \emph{Deep optimal
  stopping}, J. Mach. Learn. Res. \textbf{20} (2019), Paper No. 74, 25.
  \MR{3960928}

\bibitem[BGJ20]{BeckGononAJPicard2020}
Christian Beck, Lukas Gonon, and Arnulf Jentzen, \emph{{Overcoming the curse of
  dimensionality in the numerical approximation of high-dimensional semilinear
  elliptic partial differential equations}}, Preprint, arXiv 2003.00596 (2020).

\bibitem[BGTW19]{Buehler2018}
H.~Buehler, L.~Gonon, J.~Teichmann, and B.~Wood, \emph{Deep hedging}, Quant.
  Finance \textbf{19} (2019), no.~8, 1271--1291. \MR{3977742}

\bibitem[BHJK20]{Beck2020}
Christian Beck, Martin Hutzenthaler, Arnulf Jentzen, and Benno Kuckuck,
  \emph{An overview on deep learning-based approximation methods for partial
  differential equations}, Preprint, arXiv 2012.12348 (2020).

\bibitem[BS11]{boettcher2011}
Björn Böttcher and Alexander Schnurr, \emph{{The Euler Scheme for Feller
  Processes}}, Stochastic Analysis and Applications \textbf{29} (2011), no.~6,
  1045--1056.

\bibitem[BS18]{Bayer2018}
Christian Bayer and Benjamin Stemper, \emph{Deep calibration of rough
  stochastic volatility models}, arXiv:1810.03399, 2018.

\bibitem[CKT20]{Cuchiero2019}
Christa Cuchiero, Wahid Khosrawi, and Josef Teichmann, \emph{A generative
  adversarial network approach to calibration of local stochastic volatility
  models}, Risks \textbf{8} (2020), no.~4, 101.

\bibitem[EGJS22]{EGJS18_787}
Dennis Elbr\"{a}chter, Philipp Grohs, Arnulf Jentzen, and Christoph Schwab,
  \emph{D{NN} expression rate analysis of high-dimensional {PDE}s: application
  to option pricing}, Constr. Approx. \textbf{55} (2022), no.~1, 3--71.
  \MR{4376559}

\bibitem[EHJK19]{EHutzenthalerJentzenKruse2019MLP}
Weinan E, Martin Hutzenthaler, Arnulf Jentzen, and Thomas Kruse, \emph{On
  multilevel {P}icard numerical approximations for high-dimensional nonlinear
  parabolic partial differential equations and high-dimensional nonlinear
  backward stochastic differential equations}, J. Sci. Comput. \textbf{79}
  (2019), no.~3, 1534--1571. \MR{3946468}

\bibitem[EMW22]{E2022}
Weinan E, Chao Ma, and Lei Wu, \emph{The {B}arron space and the flow-induced
  function spaces for neural network models}, Constr. Approx. \textbf{55}
  (2022), no.~1, 369--406. \MR{4376565}

\bibitem[EW21]{E2021}
Weinan E and Stephan Wojtowytsch, \emph{Kolmogorov width decay and poor
  approximators in machine learning: shallow neural networks, random feature
  models and neural tangent kernels}, Res. Math. Sci. \textbf{8} (2021), no.~1,
  Paper No. 5, 28. \MR{4198759}

\bibitem[FK85]{FujiKuni85}
Tsukasa Fujiwara and Hiroshi Kunita, \emph{Stochastic differential equations of
  jump type and {L}\'{e}vy processes in diffeomorphisms group}, J. Math. Kyoto
  Univ. \textbf{25} (1985), no.~1, 71--106. \MR{777247}

\bibitem[GGJ{\etalchar{+}}21]{Gonon2019}
Lukas Gonon, Philipp Grohs, Arnulf Jentzen, David Kofler, and David
  {\v{S}}i{\v{s}}ka, \emph{Uniform error estimates for artificial neural
  network approximations for heat equations}, Early access version available
  online. IMA J. Numer. Anal. (2021), 64 pages.

\bibitem[GHJvW18]{HornungJentzen2018}
Philipp Grohs, Fabian Hornung, Arnulf Jentzen, and Philippe von Wurstemberger,
  \emph{A proof that artificial neural networks overcome the curse of
  dimensionality in the numerical approximation of {B}lack-{S}choles partial
  differential equations}, To appear in Memoirs of the American Mathematical
  Society; arXiv:1809.02362 (2018), 124 pages.

\bibitem[Gla16]{FnmKacLevyGlau}
Kathrin Glau, \emph{A {F}eynman-{K}ac-type formula for {L}\'{e}vy processes
  with discontinuous killing rates}, Finance Stoch. \textbf{20} (2016), no.~4,
  1021--1059. \MR{3551859}

\bibitem[Gon21]{Gonon2021}
Lukas Gonon, \emph{{Random feature neural networks learn Black-Scholes type
  PDEs without curse of dimensionality}}, Preprint, arXiv 2106.08900 (2021).

\bibitem[GPW21]{Germain2021}
Maximilien Germain, Huyên Pham, and Xavier Warin, \emph{Neural networks-based
  algorithms for stochastic control and pdes in finance}, Preprint, arXiv
  2101.08068 (2021).

\bibitem[GS21]{GS20_925}
Lukas Gonon and Christoph Schwab, \emph{Deep {R}e{LU} network expression rates
  for option prices in high-dimensional, exponential {L}\'{e}vy models},
  Finance Stoch. \textbf{25} (2021), no.~4, 615--657. \MR{4318896}

\bibitem[Her17]{Hernandez2017}
A.~Hernandez, \emph{{Model calibration with neural networks}}, Risk (2017).

\bibitem[HJK{\etalchar{+}}20]{HJKNvW2020}
Martin Hutzenthaler, Arnulf Jentzen, Thomas Kruse, Tuan~Anh Nguyen, and
  Philippe von Wurstemberger, \emph{Overcoming the curse of dimensionality in
  the numerical approximation of semilinear parabolic partial differential
  equations}, Proc. A. \textbf{476} (2020), no.~2244, 630--654. \MR{4203091}

\bibitem[HJKN20]{HutzenthalerJentzenKruse2019}
Martin Hutzenthaler, Arnulf Jentzen, Thomas Kruse, and Tuan~Anh Nguyen, \emph{A
  proof that rectified deep neural networks overcome the curse of
  dimensionality in the numerical approximation of semilinear heat equations},
  SN Partial Differential Equations and Applications \textbf{1} (2020), no.~2,
  10.

\bibitem[HK05]{higham2005}
D.~Higham and P.~Kloeden, \emph{Numerical methods for nonlinear stochastic
  differential equations with jumps}, Numerische Mathematik \textbf{101}
  (2005), 101–119.

\bibitem[HK20]{hutzenthaler2017multi}
Martin Hutzenthaler and Thomas Kruse, \emph{Multilevel {P}icard approximations
  of high-dimensional semilinear parabolic differential equations with
  gradient-dependent nonlinearities}, SIAM J. Numer. Anal. \textbf{58} (2020),
  no.~2, 929--961. \MR{4075337}

\bibitem[HMT19]{Horvath2019}
Blanka Horvath, Aitor Muguruza, and Mehdi Tomas, \emph{Deep learning
  volatility}, https://ssrn.com/abstract=3322085, 2019.

\bibitem[HRSW09]{Hilber2009}
Norbert Hilber, Nils Reich, Christoph Schwab, and Christoph Winter,
  \emph{Numerical methods for {L\'{e}vy} processes}, Finance and Stochastics
  \textbf{13} (2009), 471--500.

\bibitem[JS03]{Jac2003}
Jean Jacod and Albert~N. Shiryaev, \emph{Limit theorems for stochastic
  processes}, 2nd ed., Springer, 2003.

\bibitem[KP15]{KharrPham15}
Idris Kharroubi and Huy\^{e}n Pham, \emph{Feynman-{K}ac representation for
  {H}amilton-{J}acobi-{B}ellman {IPDE}}, Ann. Probab. \textbf{43} (2015),
  no.~4, 1823--1865. \MR{3353816}

\bibitem[KS19]{KUHN20192654}
Franziska K\"{u}hn and Ren\'{e}~L. Schilling, \emph{Strong convergence of the
  {E}uler-{M}aruyama approximation for a class of {L}\'{e}vy-driven {SDE}s},
  Stochastic Process. Appl. \textbf{129} (2019), no.~8, 2654--2680.
  \MR{3980140}

\bibitem[LP21]{Laakmann2020}
Fabian Laakmann and Philipp Petersen, \emph{Efficient approximation of
  solutions of parametric linear transport equations by {R}e{LU} {DNN}s}, Adv.
  Comput. Math. \textbf{47} (2021), no.~1, Paper No. 11. \MR{4206659}

\bibitem[Mai99]{Maiorov1999}
V.~E. Maiorov, \emph{On best approximation by ridge functions}, J. Approx.
  Theory \textbf{99} (1999), no.~1, 68--94. \MR{1696577}

\bibitem[Mha96]{Mhaskar1996}
Hrushikesh~N. Mhaskar, \emph{{Neural networks for optimal approximation of
  smooth and analytic functions}}, Neural computation \textbf{8} (1996), no.~1,
  164--177.

\bibitem[NW22]{NeufWu2022}
Ariel Neufeld and Sizhou Wu, \emph{{Multilevel Picard approximation algorithm
  for semilinear partial integro-differential equations and its complexity
  analysis}}, Tech. report, 2022.

\bibitem[OPS20]{Opschoor2020}
Joost A.~A. Opschoor, Philipp~C. Petersen, and Christoph Schwab, \emph{Deep
  {R}e{LU} networks and high-order finite element methods}, Anal. Appl.
  (Singap.) \textbf{18} (2020), no.~5, 715--770. \MR{4131037}

\bibitem[OSZ22]{OSZ19_839}
J.~A.~A. Opschoor, Ch. Schwab, and J.~Zech, \emph{Exponential {R}e{LU} {DNN}
  expression of holomorphic maps in high dimension}, Constr. Approx.
  \textbf{55} (2022), no.~1, 537--582. \MR{4376568}

\bibitem[PBL10]{Platen2010}
Eckhard Platen and Nicola Bruti-Liberati, \emph{{Numerical Solution of
  Stochastic Differential Equations with Jumps in Finance}}, Springer, 2010.

\bibitem[PMR{\etalchar{+}}17]{Poggio2017WhyAW}
T.~Poggio, H.~Mhaskar, L.~Rosasco, B.~Miranda, and Q.~Liao, \emph{Why and when
  can deep-but not shallow-networks avoid the curse of dimensionality: A
  review}, International Journal of Automation and Computing \textbf{14}
  (2017), 503--519.

\bibitem[Pro04]{Protter2004}
Philip~E. Protter, \emph{{Stochastic integration and differential equations}},
  2nd ed., Springer, 2004.

\bibitem[PT97]{protter1997}
Philip Protter and Denis Talay, \emph{{The Euler scheme for Lévy driven
  stochastic differential equations}}, Ann. Probab. \textbf{25} (1997), no.~1,
  393--423.

\bibitem[PV18]{PETERSEN2018296}
Philipp Petersen and Felix Voigtlaender, \emph{Optimal approximation of
  piecewise smooth functions using deep {ReLU} neural networks}, Neural Netw.
  \textbf{108} (2018), 296 -- 330.

\bibitem[RDQ19]{RegDeDeAQ2019}
F.~Regazzoni, L.~Ded\`e, and A.~Quarteroni, \emph{Machine learning for fast and
  reliable solution of time-dependent differential equations}, J. Comput. Phys.
  \textbf{397} (2019), 108852, 26. \MR{3990714}

\bibitem[RW20]{Ruf2020}
Johannes Ruf and Weiguan Wang, \emph{Neural networks for option pricing and
  hedging: a literature review}, Preprint (2020).

\bibitem[RY99]{RevuzYor}
Daniel Revuz and Marc Yor, \emph{Continuous martingales and {B}rownian motion},
  third ed., Grundlehren der Mathematischen Wissenschaften [Fundamental
  Principles of Mathematical Sciences], vol. 293, Springer-Verlag, Berlin,
  1999. \MR{1725357}

\bibitem[RZ19]{ReisingerZhang2019}
Christoph Reisinger and Yufei Zhang, \emph{Rectified deep neural networks
  overcome the curse of dimensionality for nonsmooth value functions in
  zero-sum games of nonlinear stiff systems}, arXiv:1903.06652 (2019), 34
  pages.

\bibitem[SX22]{SiegelXu2022}
Jonathan~W. Siegel and Jinchao Xu, \emph{High-order approximation rates for
  shallow neural networks with cosine and {${\rm ReLU}^k$} activation
  functions}, Appl. Comput. Harmon. Anal. \textbf{58} (2022), 1--26.
  \MR{4357282}

\bibitem[Yar17]{Yarotsky2017}
Dimitry Yarotsky, \emph{{Error bounds for approximations with deep ReLU
  networks}}, Neural Networks \textbf{94} (2017), 103--114.

\end{thebibliography}
}
\end{document}